\documentclass[11pt,reqno]{amsart}
\usepackage{amssymb}
\usepackage{amsmath, amssymb, amsthm}
\usepackage{mathrsfs}

\newtheorem{theorem}{Theorem}[section]
\newtheorem{definition}{Definition}[section]
\newtheorem{lemma}{Lemma}[section]
\newtheorem{remark}{Remark}[section]

\newtheorem{corollary}{Corollary}[section]
\numberwithin{equation}{section}

\begin{document}
\title[Navier-Stokes equations]{Well-posedness of  three-dimensional isentropic  compressible  Navier-Stokes equations with degenerate viscosities and far field vacuum}

\author{Zhouping Xin}
\address[Z.P. Xin]{The Institute of Mathematical Sciences, The Chinese University of Hong Kong, Shatin, N.T., Hong Kong.}
\email{\tt zpxin@ims.cuhk.edu.hk}

\author{Shengguo Zhu}
\address[S. G.  Zhu]{Mathematical Institute, University of Oxford,  Oxford OX2 6GG, UK; School of Mathematical Sciences, Monash University, Clayton, 3800, Australia; The Institute of Mathematical Sciences, The Chinese University of Hong Kong, Shatin, N.T., Hong Kong.}
\email{\tt shengguo.zhu@maths.ox.ac.uk}

\begin{abstract}

In this paper, the Cauchy problem for the  three-dimensional (3-D) isentropic  compressible  Navier-Stokes equations is considered. When viscosity coefficients are given as a constant multiple of the density's power ($\rho^\delta$ with $0<\delta<1$), based on some analysis  of  the nonlinear  structure of this system,   we identify the class of initial data admitting a local regular solution with far field vacuum and  finite energy  in some inhomogeneous Sobolev spaces  by introducing some new variables and initial  compatibility conditions, which solves an open problem   of degenerate viscous flow partially mentioned by Bresh-Desjardins-Metivier \cite{bd}, Jiu-Wang-Xin \cite{jiuping} and so on.   Moreover,   in contrast to the classical theory in the case of  the constant viscosity, we show   that one can not obtain any global regular solution whose $L^\infty$ norm of $u$ decays to zero as time $t$ goes to infinity.\end{abstract}

\date{Oct.  31th, 2018}
.

\maketitle

\section{Introduction}
%Due to the high degeneracy near  vacuum, the local well-posedness of classical solutions with finite energy  to the isentropic  compressible   Navier-Stokes equations  (\textbf{CNS}) with degenerate viscosities (see (\ref{fandan})) has not been completely solved until now, which attracts many people's attention such as Bresh-Desjardin-Metivier \cite{bd}, Jiu-Wang-Xin \cite{jiuping}, Li-Pan-Zhu \cite{hyp} and so on.  In this paper, we will give some  contribution in this direction.

The time evolution of the mass density $\rho\geq 0$ and the velocity $u=\left(u^{(1)},u^{(2)},u^{(3)}\right)^\top\in \mathbb{R}^3$ of a general viscous isentropic
compressible  fluid occupying a spatial domain $\Omega\subset \mathbb{R}^3$ is governed by the following isentropic  compressible  Navier-Stokes equations (\textbf{ICNS}):
\begin{equation}
\label{eq:1.1}
\begin{cases}
\rho_t+\text{div}(\rho u)=0,\\[4pt]
(\rho u)_t+\text{div}(\rho u\otimes u)
  +\nabla
   P =\text{div} \mathbb{T}.
\end{cases}
\end{equation}
Here, $x=(x_1,x_2,x_3)\in \Omega$, $t\geq 0$ are the space and time variables, respectively. For the polytropic gases, the constitutive relation is given by
\begin{equation}
\label{eq:1.2}
P=A\rho^{\gamma}, \quad A>0,\quad  \gamma> 1,
\end{equation}
where $A$ is  an entropy  constant and  $\gamma$ is the adiabatic exponent. $\mathbb{T}$ denotes the viscous stress tensor with the  form
\begin{equation}
\label{eq:1.3}
\begin{split}
&\mathbb{T}=\mu(\rho)\big(\nabla u+(\nabla u)^\top\big)+\lambda(\rho)\text{div}u\,\mathbb{I}_3,
\end{split}
\end{equation}
 where $\mathbb{I}_3$ is the $3\times 3$ identity matrix,
\begin{equation}
\label{fandan}
\mu(\rho)=\alpha  \rho^\delta,\quad \lambda(\rho)=\beta  \rho^\delta,
\end{equation}
for some  constant $\delta\geq 0$,
 $\mu(\rho)$ is the shear viscosity coefficient, $\lambda(\rho)+\frac{2}{3}\mu(\rho)$ is the bulk viscosity coefficient,  $\alpha$ and $\beta$ are both constants satisfying

 \begin{equation}\label{10000}\alpha>0,\quad \text{and} \quad   2\alpha+3\beta\geq 0.
 \end{equation}

%In addition,  when $\alpha=\beta=0$, from
%\eqref{eq:1.1}, we naturally have the  isentropic compressible  Euler
%equations for  the  inviscid flow:
%\begin{equation}
%\label{eq:1.1E}
%\begin{cases}
%\displaystyle
%\rho_t+\text{div}(\rho u)=0,\\[8pt]
%\displaystyle
%(\rho u)_t+\text{div}(\rho u\otimes u)
%  +\nabla
%   P =0,
%\end{cases}
%\end{equation}
%which is a fundamental example of a system of hyperbolic conservation laws.

Let $\Omega=\mathbb{R}^3$. Assuming  $0<\delta<1$, we look for a smooth solution
$(\rho,u)$ with finite energy  to the Cauchy problem for (\ref{eq:1.1})-(\ref{10000}) with the following  initial data and far field behavior:
\begin{align}
&(\rho,u)|_{t=0}=(\rho_0(x)\geq 0,  u_0(x)) \ \ \ \ \ \ \  \ \ \   \  \ \text{for} \quad  x\in \mathbb{R}^3,\label{initial}\\[2pt]
&(\rho,u)(t,x)\rightarrow  (0,0) \quad  \text{as}\ \ \  |x|\rightarrow \infty \qquad   \text{for} \quad  t\geq 0.  \label{far}
\end{align}

In the theory of gas dynamics, the \textbf{CNS} can be derived  from the Boltzmann equations through the Chapman-Enskog expansion, cf. Chapman-Cowling \cite{chap} and Li-Qin  \cite{tlt}. Under some proper physical assumptions,  the viscosity coefficients and heat conductivity coefficient $\kappa$ are not constants but functions of the absolute temperature $\theta$ such as:
\begin{equation}
\label{eq:1.5g}
\begin{split}
\mu(\theta)=&a_1 \theta^{\frac{1}{2}}F(\theta),\quad  \lambda(\theta)=a_2 \theta^{\frac{1}{2}}F(\theta), \quad \kappa(\theta)=a_3 \theta^{\frac{1}{2}}F(\theta)
\end{split}
\end{equation}
for some   constants $a_i$ $(i=1,2,3)$ (see \cite{chap}).  Actually for the cut-off inverse power force models, if the intermolecular potential varies as $r^{-a}$,
 where $ r$ is intermolecular distance,  then in (\ref{eq:1.5g}):
 $$F(\theta)=\theta^{b}\quad  \text{with}\quad   b=\frac{2}{a} \in [0,+\infty).$$
  In particular, for Maxwellian molecules,
$a = 4$  and $b=\frac{1}{2}$;
 for rigid elastic spherical molecules,
$a=\infty$ and  $b=0$; while for  ionized gas, $a=1$ and $b=2$ (see \S 10 of \cite{chap}).

According to Liu-Xin-Yang \cite{taiping}, for  isentropic and polytropic fluids , such a dependence is inherited through the laws of Boyle and Gay-Lussac:
$$
P=R\rho \theta=A\rho^\gamma,  \quad \text{for \ \ constant} \quad R>0,
$$
i.e., $\theta=AR^{-1}\rho^{\gamma-1}$,  and one finds that the viscosity coefficients are functions of the density of the form $(\ref{fandan})$ with $0<\delta<1$ in many cases.

Throughout this paper, we adopt the following simplified notations, most of them are for the standard homogeneous and inhomogeneous Sobolev spaces:
\begin{equation*}\begin{split}
 & \|f\|_s=\|f\|_{H^s(\mathbb{R}^3)},\quad |f|_p=\|f\|_{L^p(\mathbb{R}^3)},\quad \|f\|_{m,p}=\|f\|_{W^{m,p}(\mathbb{R}^3)},\quad  |f|_{C^k}=\|f\|_{C^k(\mathbb{R}^3)},  \\[4pt]
 & \|f\|_{XY(t)}=\| f\|_{X([0,t]; Y(\mathbb{R}^3))},\ \  \dot{H}^\iota=\{f:\mathbb{R}^3\rightarrow \mathbb{R}\ | \|f\|^2_{\dot{H}^\iota}=\int |\xi|^{2\iota}|\hat{f}(\xi)|^2\text{d}\xi<\infty\},\\[4pt]
 & D^{k,r}=\{f\in L^1_{loc}(\mathbb{R}^3): |f|_{D^{k,r}}=|\nabla^kf|_{r}<+\infty\},\quad D^k=D^{k,2},  \\[4pt]
  & D^{1}=\{f\in L^6(\mathbb{R}^3):  |f|_{D^1}= |\nabla f|_{2}<\infty\},\quad |f|_{D^1}=\|f\|_{D^1(\mathbb{R}^3)},\\
 & \|f\|_{X_1 \cap X_2}=\|f\|_{X_1}+\|f\|_{X_2}, \quad  \int_{\mathbb{R}^3}  f \text{d}x  =\int f, \ \ X([0,T]; Y(\mathbb{R}^3))= X([0,T]; Y).
\end{split}
\end{equation*}
 A detailed study of homogeneous Sobolev spaces  can be found in \cite{gandi}.

When  $\inf_x {\rho_0(x)}>0$, the local well-posedness of classical solutions for (\ref{eq:1.1})-(\ref{far}) follows from the standard symmetric hyperbolic-parabolic structure which satisfies the well-known Kawashima's condition, c.f. \cite{KA, nash}. However, such an approach fails  in the presence of the vacuum due to the degeneracies of the time evolution and viscosities. Generally vacuum will appear in the far field under some physical requirements such as finite total mass in  $ \mathbb{R}^3$.  One of the  main issues in the presence of vacuum is   to understand the behavior of the  velocity field near the vacuum.
For the constant  viscosity flow  ($\delta=0$ in (\ref{fandan})), a remedy was suggested  by Cho-Choe-Kim \cite{CK3}, where they imposed  initially a {\it compatibility condition}:
\begin{equation*}
-\text{div} \mathbb{T}_{0}+\nabla P(\rho_0)=\sqrt{\rho_{0}} g,\quad \text{for\  \ some} \ \ g\in L^2(\mathbb{R}^3),
\end{equation*}
which  leads to  $$(\sqrt{\rho}u_t, \nabla u_t)\in L^\infty ([0,T_{*}]; L^2)$$  for a  short  time $T_{*}>0$. Then  they established  successfully   the local well-posedness of smooth solutions with vacuum in some Sobolev spaces of $\mathbb{R}^3$, and also showed that this kind of initial compatibility condition is necessary for their solution class.

For density-dependent viscosities  ($\delta>0$ in (\ref{fandan})),  the strong degeneracy of the momentum equations in  \eqref{eq:1.1}  near the vacuum creates serious difficulties for the well-posedness of both  strong and weak solutions.  Though some significant achievements    \cite{bd2, bd, ding, zhu,  lz, hyp,   taiping, vassu,  zz, zyj, sz34} have been obtained, however,  a lot of fundamental questions remain open, including the local well-posedness of classical solutions with finite energy in multi-dimensions  for $\delta \in (0,1)$. Our
result obtained in this paper has taken a first step toward this direction.

Compared with flows of constant viscosities, the analysis of the   degeneracies  in the momentum equations $(\ref{eq:1.1})_2$ requires some special attentions. Indeed,  for $\delta=0$ in  \cite{CK3},  the uniform ellipticity of the Lam\'e operator $L$ defined by
\begin{equation*}Lu \triangleq -\alpha\triangle u-(\alpha+\beta)\nabla \mathtt{div}u
\end{equation*}
plays an essential role in the regularity estimates on $u$. One can use the standard elliptic theory to estimate $|u|_{D^{k+2}}$  by the $D^k$-norm of all other terms in the  momentum equations.  However, for $\delta>0$,  the viscosity coefficients vanish in the presence of vacuum, which  makes it difficult to adapt the approach  in \cite{CK3} to the current case.

For the cases  $\delta\in (0,\infty)$,  if  $\rho>0$, $(\ref{eq:1.1})_2$ can be formally rewritten as
\begin{equation}\label{qiyi}
\begin{split}
u_t+u\cdot\nabla u +\frac{A\gamma}{\gamma-1}\nabla\rho^{\gamma-1}+\rho^{\delta-1} Lu=\psi \cdot  Q(u),
\end{split}
\end{equation}
where the quantities $\psi$ and $Q(u)$ are given by
\begin{equation}\label{operatordefinition}
\begin{split}
\psi \triangleq & \nabla \log \rho \quad \text{when}\quad \delta=1;\\
\psi \triangleq & \frac{\delta}{\delta-1} \nabla \rho^{\delta-1}\quad \text{when}\quad \delta \in (0,1)\cup (1,\infty);\\
Q(u)\triangleq &\alpha(\nabla u+(\nabla u)^\top)+\beta\mathtt{div}u\mathbb{I}_3.
\end{split}
\end{equation}
When $\delta=1$, according to  (\ref{qiyi})-(\ref{operatordefinition}),  the degeneracies of the time evolution and  viscosities on $u$ caused by the  vacuum have  been transferred to the possible singularity of the  term  $\nabla \log \rho$, which actually  can be controlled by a symmetric hyperbolic system with a source term  $\nabla \text{div}u$ in  Li-Pan-Zhu \cite{sz3}. Then via establishing a uniform a priori estimates in $L^6\cap D^1\cap D^2$  for $\nabla \log \rho$, the existence of 2-D local classical solution with far field vacuum   to  (\ref{eq:1.1}) has been obtained in \cite{sz3}, which also applies to the 2-D shallow water equations.
When $\delta> 1$,  (\ref{qiyi})-(\ref{operatordefinition}) imply that actually the velocity $u$ can be governed by a nonlinear degenerate  parabolic system without singularity   near the vacuum region.
Based on this observation, by using some hyperbolic  approach which bridges the parabolic system (\ref{qiyi}) when $\rho>0$ and the hyperbolic one  $u_t+u\cdot\nabla u=0$ when $\rho=0$,   the existence of 3-D local classical solutions with vacuum to  (\ref{eq:1.1}) was established in  Li-Pan-Zhu \cite{sz333}.
The corresponding global well-posedness in some homogeneous Sobolev spaces  has been established  by Xin-Zhu \cite{zz} under some initial smallness  assumptions.

However,  such approaches used in \cite{sz3,sz333,zz} fail to apply to the   case  $\delta\in (0,1)$. Indeed, when vacuum appears only at far fields, the velocity field $u$ is still     governed by  the  quasi-linear parabolic system (\ref{qiyi})-(\ref{operatordefinition}).  Yet,  some new essential  difficulties arise compared with the case $\delta\geq 1$:
\begin{enumerate}
\item
first, the source term contains a stronger singularity as:
$$\nabla \rho^{\delta-1}=(\delta-1)\rho^{\delta-1}\nabla \log \rho,$$
 whose behavior will become more singular than that of $\nabla \log \rho $ in \cite{sz3} due to $\delta-1<0$ when the  density $\rho \rightarrow 0$;
\item
second,  the coefficient $\rho^{\delta-1}$ in front of the Lam\'e operator $ L$ will tend to $\infty$ as $\rho\rightarrow 0$ in the far filed instead of equaling to $1$ in \cite{sz3} or tending to $0$ in \cite{sz333,zz} Then   it is  necessary  to show that the term $\rho^{\delta-1} Lu$ is well defined.
\end{enumerate}

Therefore, the three quantities
$$(\rho^{\gamma-1}, \nabla \rho^{\delta-1}, \rho^{\delta-1} Lu)$$
will play significant roles in our analysis on the higher order  regularities of the fluid velocity $u$. Due to  this observation, we first introduce a proper class of solutions called regular solutions to  the Cauchy problem (\ref{eq:1.1})-(\ref{far}).

\begin{definition}\label{d1}
 Let $T> 0$ be a finite constant. A solution $(\rho,u)$ to the Cauchy problem (\ref{eq:1.1})-(\ref{far}) is called a regular solution in $ [0,T]\times \mathbb{R}^3$ if $(\rho,u)$  satisfies this problem in the sense of distribution and:
\begin{equation*}\begin{split}
&(\textrm{A})\quad \rho>0,\quad   \rho^{\gamma-1}\in C([0,T]; H^3), \quad  \nabla \rho^{\delta-1}\in L^\infty([0,T]; L^\infty\cap D^{2});\\
&(\textrm{B})\quad u\in C([0,T]; H^3)\cap L^2([0,T]; H^4), \quad  u_t \in C([0,T]; H^1)\cap L^2([0,T] ; D^2), \\
&\qquad \ \  \rho^{\frac{\delta-1}{2}}\nabla u \in C([0,T]; L^2), \quad  \rho^{\frac{\delta-1}{2}}\nabla u_t \in  L^\infty([0,T_*];L^2),\\
&\qquad \ \ \rho^{\delta-1}\nabla u \in L^\infty([0,T]; D^1),\quad   \rho^{\delta-1}\nabla^2 u \in C([0,T]; H^1)\cap   L^2([0,T]; D^2).
\end{split}
\end{equation*}
\end{definition}

\begin{remark}\label{r1}
First, it follows from  the Definition \ref{d1}  that $\nabla \rho^{\delta-1}\in L^\infty$, which means that the vacuum occurs if and only  in the far field.

Second, we introduce some physical quantities that will be used in this paper:
\begin{align*}
m(t)=&\int \rho(t,x)\quad \textrm{(total mass)},\\
\mathbb{P}(t)=&\int \rho(t,x)u(t,x) \quad \textrm{(momentum)},\\
E_k(t)=&\frac{1}{2}\int \rho(t,x)|u(t,x)|^{2}\quad \textrm{ (total kinetic energy)},\\
E(t)=&E_k(t)+\int \frac{P}{\gamma-1} \quad \textrm{ (total energy)}.
\end{align*}
Actually, it follows from the definition that a regular solution satisfies the conservation of  total mass and   momentum (see Lemma \ref{lemmak}).  Furthermore, it satisfies the energy equality (see (4.5)).  Note that the conservation of  momentum is not  clear   for the  strong solution with vacuum  to the flows of constant viscosities \cite{CK3}. In this sense, the definition of regular solutions here is consistent with the physical background of the compressible Navier-Stokes equations.
\end{remark}

The regular solutions select  velocity in a physically reasonable way when the density approaches the  vacuum at far fields. Under the help of this notion of solutions, the  momentum equations can be reformulated into a special quasi-linear parabolic system with some possible singular source terms near the  vacuum, and  the coefficients in front  of  Lam\'e operator $ L$ will tend to $\infty$ as $\rho\rightarrow 0$ in the far filed. However, the problem becomes trackable through an  elaborate linearization and  approximation process.

Now we are ready to state the main results in this paper.  First we prove the existence of the unique regular solution with $\nabla \rho^{\delta-1}\in C([0,T]; D^1\cap D^2)$ to (\ref{eq:1.1})-(\ref{far}).
\begin{theorem}\label{th2} Let  parameters  $(\gamma,\delta, \alpha,\beta)$ satisfy
\begin{equation}\label{canshu}
\gamma>1,\quad 0<\delta<1, \quad \alpha>0, \quad 2\alpha+3\beta\geq 0.
\end{equation}
  If the initial data $( \rho_0, u_0)$ satisfies
\begin{equation}\label{th78}
\begin{split}
&\rho_0>0,\quad (\rho^{\gamma-1}_0, u_0)\in H^3, \quad \nabla \rho^{\delta-1}_0\in D^1\cap D^2,   \quad \nabla \rho^{\frac{\delta-1}{2}}_0\in L^4,\\
\end{split}
\end{equation}
and the initial  compatibility conditions:
\begin{equation}\label{th78zx}
\displaystyle
  \nabla u_0=\rho^{\frac{1-\delta}{2}}_0 g_1,\quad \   \ \  Lu_0= \rho^{1-\delta}_0g_2,\quad \ \
 \nabla \Big(\rho^{\delta-1}_0Lu_0\Big)=\rho^{\frac{1-\delta}{2}}_0g_3,
\end{equation}
for some $(g_1,g_2,g_3)\in L^2$, then there exist a  time $T_*>0$ and a unique regular solution $(\rho, u)$  in $ [0,T_*]\times \mathbb{R}^3$ to the Cauchy problem (\ref{eq:1.1})-(\ref{far}) satisfying:
\begin{equation}\label{reg11}\begin{split}
& t^{\frac{1}{2}}u\in L^\infty([0,T_*];D^4),\quad  t^{\frac{1}{2}}u_t\in L^\infty([0,T_*];D^2)\cap L^2([0,T_*] ; D^3),\\
&u_{tt}\in L^2([0,T_*];L^2),\quad   t^{\frac{1}{2}}u_{tt}\in L^\infty([0,T_*];L^2)\cap L^2([0,T_*];D^1),\\
& \rho^{1-\delta}\in L^\infty([0,T_*];  L^\infty\cap D^{1,6} \cap D^{2,3} \cap D^3), \\
&  \nabla \rho^{\delta-1}\in C([0,T]; D^1\cap D^2), \  \nabla \log \rho \in L^\infty([0,T_*];L^\infty\cap L^6\cap  D^{1,3} \cap D^2).
\end{split}
\end{equation}

Moreover, if $1<\gamma\leq 2$, $(\rho, u)$ is a classical solution to   (\ref{eq:1.1})-(\ref{far}) in $(0,T_*]\times \mathbb{R}^3$.
\end{theorem}
\begin{remark}\label{r2} The conditions (\ref{th78})-(\ref{th78zx}) in Theorem \ref{th2}  identify a class of admissible initial data that
makes the  problem \eqref{eq:1.1}--\eqref{far} solvable, which are satisfied by, for example,
$$
\rho_0(x)=\frac{1}{1+|x|^{2a}}, \quad u_0(x)\in C_0^3(\mathbb{R}^3), \quad \text{and} \quad \frac{3}{4(\gamma-1)}< a<\frac{1}{4(1-\delta)}.$$
Particularly, when  $\nabla u_0$ is  compactly supported,
  the compatibility conditions (\ref{th78zx}) are satisfied automatically.
\end{remark}

Second, we can also prove the existence of the unique regular solution with $\nabla \rho^{\delta-1}\in C([0,T]; L^q\cap D^{1,3}\cap D^2)$ to (\ref{eq:1.1})-(\ref{far}).
\begin{theorem}\label{th2-1} Let  (\ref{canshu}) hold, and $q\in (3,+\infty)$ be a fixed constant.
  If the initial data $( \rho_0, u_0)$ satisfies
\begin{equation}\label{th78-1}
\begin{split}
&\rho_0>0,\quad (\rho^{\gamma-1}_0, u_0)\in H^3, \quad \nabla \rho^{\delta-1}_0\in L^q\cap D^{1,3}\cap D^2,   \quad \nabla \rho^{\frac{\delta-1}{2}}_0\in L^6,\\
\end{split}
\end{equation}
and the initial  compatibility conditions (\ref{th78zx}), then there exist a  time $T_*>0$ and a unique regular solution $(\rho, u)$  in $ [0,T_*]\times \mathbb{R}^3$ to the Cauchy problem (\ref{eq:1.1})-(\ref{far}) satisfying:
\begin{equation}\label{reg11-1}\begin{split}
& t^{\frac{1}{2}}u\in L^\infty([0,T_*];D^4),\quad  t^{\frac{1}{2}}u_t\in L^\infty([0,T_*];D^2)\cap L^2([0,T_*] ; D^3),\\
&u_{tt}\in L^2([0,T_*];L^2),\quad   t^{\frac{1}{2}}u_{tt}\in L^\infty([0,T_*];L^2)\cap L^2([0,T_*];D^1),\\
& \rho^{1-\delta}\in L^\infty([0,T_*]; L^\infty\cap D^{1,q} \cap D^{1,6}  \cap D^{1,\infty}\cap D^{2,3} \cap D^{3}), \\
& \nabla \rho^{\delta-1}\in C([0,T]; L^q\cap D^{1,3}\cap D^2),\\
& \nabla \log \rho \in L^\infty([0,T_*];L^q \cap L^6 \cap L^\infty\cap  D^{1,3} \cap D^{2}).
\end{split}
\end{equation}

Moreover, if $1<\gamma\leq 2$, $(\rho, u)$ is a classical solution to   (\ref{eq:1.1})-(\ref{far}) in $(0,T_*]\times \mathbb{R}^3$.
\end{theorem}
\begin{remark}\label{r2-1} The conditions (\ref{th78-1}) and (\ref{th78zx})  in Theorem \ref{th2-1} identify a class of admissible initial data that
makes the  problem \eqref{eq:1.1}--\eqref{far} solvable, which are satisfied by, for example,
$$
\rho_0(x)=\frac{1}{1+|x|^{2a}}, \quad u_0(x)\in C_0^3(\mathbb{R}^3), \quad \text{and} \quad \frac{3}{4(\gamma-1)}< a<\frac{1-3/q}{2(1-\delta)}.$$
It should be pointed out that, for such kind of initial data,  when $q$ is chosen large enough,  the range of $a$ is wider than that of a shown in the example of Remark \ref{r2}. However, for general initial data,  the conclusion obtained in any one of Theorems \ref{th2}-\ref{th2-1} cannot be implied by that of the other one.
\end{remark}

\begin{remark}\label{rxiangrong}
 The compatibility conditions (\ref{th78zx}) are also necessary for the existence of  regular  solutions $(\rho, u)$ obtained in Theorems \ref{th2}-\ref{th2-1}.
In particular, the one shown in $(\ref{th78zx})_2$ $((\ref{th78zx})_3)$  plays a key role in the derivation of $u_t \in L^\infty([0,T_*];L^2(\mathbb{R}^3))$ $(\rho^{\frac{\delta-1}{2}}\nabla u_t \in  L^\infty([0,T_*];L^2(\mathbb{R}^3)))$, which will be used in the uniform estimates for  $|u|_{D^2}$ $(|u|_{D^3})$.
\end{remark}

A natural question is   whether the local solution in Theorems \ref{th2}-\ref{th2-1} can be extended globally in time, and what  the large time behavior is. In contrast to the classical theory for the constant viscosity case \cite{HX1,mat,zx}, we show the following somewhat surprising phenomenon that such an extension is impossible if the velocity field decays to zero as $t\rightarrow +\infty$ and the initial total momentum is non-zero.
First, based on the  physical quantities introduced  in Remark \ref{r1},  we define  a solution class  as follows:
\begin{definition}\label{d2}
Let $T>0$ be any constant. For   the Cauchy problem (\ref{eq:1.1})-(\ref{far}),  a classical solution     $(\rho,u)$ is said to be in $ D(T)$ if $(\rho,u)$  satisfies the following conditions:
\begin{equation*}\begin{split}
&(\textrm{A})\quad \text{Conservation of total mass:}\quad 0<m(0)=m(t)<\infty \ \text{for  any}  \ t\in [0,T];\\[2pt]
&(\textrm{B})\quad \text{Conservation of momentum:}\quad   0<|\mathbb{P}(0)|=|\mathbb{P}(t)|<+\infty \ \text{for  any} \ \ t\in [0,T];\\[2pt]
&(\textrm{C})\quad \text{Finite kinetic energy:}  \quad  0<E_k(t)<\infty  \ \text{for  any}  \ t\in [0,T].
\end{split}
\end{equation*}
\end{definition}

Then one has:
\begin{theorem}\label{th:2.20}
Let  parameters  $(\gamma,\delta, \alpha,\beta)$ satisfy
\begin{equation}\label{canshu-22}
\gamma \geq 1,\quad \delta\geq 0, \quad \alpha>0, \quad 2\alpha+3\beta\geq 0.
\end{equation}  Then for   the Cauchy problem (\ref{eq:1.1})-(\ref{far}), there is no classical solution $(\rho,u)\in D(\infty)$  with
\begin{equation}\label{eq:2.15}
\limsup_{t\rightarrow +\infty} |u(t,x)|_{\infty}=0.
\end{equation}

\end {theorem}

According to  Theorems \ref{th2}-\ref{th:2.20} and Remark \ref{r1}, one shows that
\begin{corollary}\label{th:2.20-c}
 Let (\ref{canshu}) hold and
\begin{equation}\label{decaycanshu}1< \gamma \leq \frac{3}{2}, \quad  \delta\in \big[\gamma-1,1\big).
\end{equation} Assume that  $m(0)>0$ and $|\mathbb{P}(0)|>0$.
Then  for   the Cauchy problem (\ref{eq:1.1})-(\ref{far}),  there is no global  regular  solution $(\rho,u)$ defined in Definition \ref{d1}
satisfying the following decay
\begin{equation}\label{eq:2.15}
\limsup_{t\rightarrow +\infty} |u(t,x)|_{\infty}=0.
\end{equation}

\end {corollary}

Moreover,  according to  Theorem\ref{th:2.20}, Remark \ref{r1} and \cite{CK3, guahu, HX1}, for constant viscosity flow,  one can give the following example: the classical solution exists globally and keeps the conservation of total mass,  but can not keep the conservation of momentum for all the time $t\in (0,\infty)$.\begin{corollary}\label{th:2.20-HLX}
Let  $\delta=0$ in (\ref{fandan}).
For any given numbers $M>0$, $\iota \in (\frac{1}{2},1]$, and $\overline{\rho}>0$, suppose that the initial data $(\rho_0,u_0)$ satisfies
\begin{equation}\label{initialhxl}
\begin{split}
&E(0)<\infty,\quad u_0\in \dot{H}^\iota \cap D^1\cap D^3, \quad \rho^{\frac{1}{2}}_0\in H^1, \\
& (\rho_0,P(\rho_0))\in H^3, \quad  0\leq \inf \rho_0 \leq \sup \rho_0 \leq \overline{\rho},\quad \|u\|_{\dot{H}^\iota}\leq M,
\end{split}
\end{equation}
and the compatibility condition
$$
-\alpha \triangle u_0-(\alpha+\beta) \nabla \text{div}u_0+\nabla P(\rho_0)=\rho_0 g_4,
$$
for some $g_4\in D^1$ with $\rho^{\frac{1}{2}} g_4 \in L^2$. Then there exists a positive constant $\zeta$ depending on $\alpha$, $\beta$, $A$, $\gamma$, $\overline{\rho}$, $\iota$ and $M$ such that if
\begin{equation}\label{smallness}
E(0)\leq \zeta,
\end{equation}
the Cauchy problem  (\ref{eq:1.1})-(\ref{far}) has a unique global classical solution $(\rho,u)$ in $(0,\infty)\times \mathbb{R}^3$ satisfying, for any $0<\tau <T<\infty$,
\begin{align}
& E(t) \leq E(0), \ \  t   \in [0,\infty); \quad   0\leq \rho(t,x)\leq 2\overline{\rho}, \quad (t,x)   \in [0,\infty)\times \mathbb{R}^3;\label{lagrangian1q}\\[6pt]
&\rho^{\frac{1}{2}} \in C([0,T];H^1); \quad m(t)=m(0), \quad    t   \in [0,\infty);\label{lagrangian2q}\\[6pt]
&\sup_{t\in [0,T]}\big(\min\{1,t\}\big)^3  \int \rho |\dot{u}|^2\leq 2E(0)^{\frac{1}{2}};\quad   (\rho,P(\rho)) \in C([0,T];H^3);  \label{lagrangian3q}\\[6pt]
&  u\in C([0,T]; D^1\cap D^3)\cap L^2([0,T] ; D^4)\cap L^\infty([\tau,T];D^4); \label{lagrangian4q}\\[6pt]
&  u_t\in L^\infty([0,T]; D^1)\cap L^2([0,T] ; D^2)\cap L^\infty([\tau,T];D^2)\cap H^1([\tau,T];D^1),\label{lagrangian5q}
\end{align}
for $\dot{u}=u_t+u\cdot \nabla u$, and the following large-time behavior
\begin{equation}\label{largetimehlx}
\begin{split}
\lim_{t\rightarrow \infty} \int \big(|\rho|^b+\rho^{\frac{1}{2}}|u|^4+|\nabla u|^2\big)(t,x)\text{d}x=0,\quad \text{for \ any \ constant } \ \ b>\gamma.
\end{split}
\end{equation}

Furthermore, if  $m(0)>0$ and $|\mathbb{P}(0)|>0$, then the solution obtained above can not keep the conservation of the momentum for all the time $t\in (0,\infty)$.
\end {corollary}

\begin{remark}\label{commentoniinitial}
Note that the classes of initial data for Corollary 1.2 is a sub-class of initial data for the global well-posedness theory in \cite{HX1}. Indeed, compared with the assumptions on initial data in \cite{HX1}, additional conditions,
$$
\rho^{\frac{1}{2}}_0\in H^1, \quad m(0)>0, \quad \text{and}  \quad |\mathbb{P}(0)|>0,$$
are required here to keep  the conservation of positive total  mass, and satisfy the  assumptions of Theorem \ref{th:2.20}.

\end{remark}

\begin{remark}\label{zhunbei2}
Note that for the regular solution $(\rho,u)$ obtained in Theorems \ref{th2}-\ref{th2-1}, $u$  stays in the inhomogeneous Sobolev space $H^3$ instead of the homogenous one $D^1\cap D^3$ in \cite{CK3} for  constant viscous flows.

 It is also worth pointing out  that recently,  if the initial density is compactly supported,  Li-Wang-Xin \cite{ins} prove that  classical solutions with finite energy to the Cauchy problem of the compressible Navier-Stokes systems with constant viscosities do not exist in general in   inhomogeneous Sobolev space for any short time, which indicates in particular that the homogeneous Sobolev space is crucial as studying the well-posedness (even locally in time) for the Cauchy problem of the compressible Navier-Stokes systems in the presence of such kind of the vacuum.

 Based on the conclusion obtained in Theorems \ref{th2}-\ref{th2-1} and \cite{ins}, first there is a natural question  that whether one can obtain the local-in-time existence  of the classical solutions with finite energy in   inhomogeneous Sobolev space  to the Cauchy problem of the compressible Navier-Stokes systems with constant viscosities   under the assumption that the initial density is positive but decays to zero in the far field, or not.   Second, can  the conclusion obtained in \cite{ins}  be applied to the  degenerate  system considered here?
Due to the obvious difference on the structure between the constant viscous flows and degenerate viscous flows,   such  questions are not easy and  will be discussed in our future work Xin-Zhu \cite{XZ-L2}.

\end{remark}

\begin{remark}\label{additionalinformation}
Under the assumption $\inf \rho_0=0$,  in Li-Xin \cite{lz}, the  global existence of weak solutions  to the Cauchy problem (\ref{eq:1.1})-(\ref{far}) has been established for the cases
$ 3/4<\delta< 1, \ 1<\gamma <6\delta-3$
in the  three dimensional space, and
$1/2 <\delta<1,\ \gamma>1$
in the two dimensional space.

\end{remark}

We now outline the organization of the rest of this paper. In Section $2$, we list some basic lemmas  to be used later.

Section $3$ is devoted to proving Theorem \ref{th2}.  First,    in order to analyze  the behavior of the possible singular term $\nabla \rho^{\delta-1}$ clearly,   we enlarge the original problem (\ref{eq:1.1})-(\ref{far}) into (\ref{eq:cccq})-(\ref{sfanb1}) in terms of the following  variables
$$
\phi=\frac{A\gamma}{\gamma-1} \rho^{\gamma-1},\quad \psi=\frac{\delta}{\delta-1}\nabla \rho^{\delta-1}=\frac{\delta}{\delta-1}\Big(\frac{A\gamma}{\gamma-1}\Big)^{\frac{1-\delta}{\gamma-1}}\nabla \phi^{\frac{\delta-1}{\gamma-1}}=(\psi^{(1)},\psi^{(2)},\psi^{(3)}), \quad u
$$
 in \S 3.1. Then  the behavior of the velocity $u$ can be controlled by the following equations:
$$
u_t+u\cdot\nabla u +\nabla \phi+a\phi^{2e}Lu=\psi \cdot Q(u),
$$
where $
 a=\Big(\frac{A\gamma}{\gamma-1}\Big)^{-2e},\quad \text{and} \quad e=\frac{\delta-1}{2(\gamma-1)}<0$. Due to the fact that $\phi^{2e}$  has an  uniformly positive lower bound in the whole space, then
for this  special quasi-linear parabolic system, one can find formally that, even though the coefficients $a\phi^{2e}$ in front  of  Lam\'e operator $ L$ will tend to $\infty$ as $\rho\rightarrow 0$ in the far filed, yet this structure could give a better a priori estimate on $u$ in $H^3$ than those of \cite{CK3,sz3,sz333,zz}  if one can control  the possible singular term  $\psi$ in $D^1\cap D^2$. According to $(\ref{eq:1.1})_2$ and (\ref{kuzxc}), $\psi$ is governed by
\begin{equation}\label{psi-1}
\psi_t+\sum_{l=1}^3 A_l(u) \partial_l\psi+B(u)\psi+\delta a\phi^{2e}\nabla \text{div} u=0,
\end{equation}
where  the definitions of $A_l$ ($l=1,2,3$) and  $B$  can be found in \S 3.2.1. This implies that actually the subtle source term $\psi$  could be controlled by a     symmetric hyperbolic  system with a possible singular higher order term $\delta a\phi^{2e}\nabla \text{div} u$. In order to close the estimates, we need to control $\phi^{2e}\nabla \text{div} u$ in $D^1\cap D^2$, which  can be obtained by regarding the momentum equations as the following inhomogeneous Lam\'e  equations:
$$
a L(\phi^{2e}u)=-u_t-u\cdot\nabla u -\nabla \phi+\psi\cdot Q(u) -\frac{\delta-1}{\delta}G(\psi,u)=W,$$
with
$$
G(\psi,u)= \alpha  \psi\cdot \nabla u+\alpha \text{div}(u\otimes \psi)+(\alpha+\beta)\big(\psi \text{div}u+\psi \cdot (\nabla u)+u  \cdot \nabla \psi\big).
$$
In fact, one has
\begin{equation}\label{singularelliptic}
\begin{split}
|\phi^{2e}\nabla^2u|_{D^1}\leq & C(|\psi|_\infty|\nabla^2u|_2+|\phi^{2e}\nabla^3u|_2)\\
\leq & C(|\psi|_\infty|\nabla^2u|_2+|\nabla \psi|_3|\nabla u|_6+|\nabla^2\psi|_2|u|_\infty+|W|_{D^1}),
\end{split}
\end{equation}
for some constant  $C>0$ independent of the lower bound of the density provided that
$$
\phi^{2e}u\rightarrow 0\qquad \text{as} \qquad |x| \rightarrow +\infty,
$$
which can be  verified  in  an approximation process from non-vacuum flows  to the flow with   far field vacuum. Similar calculations can be done for  $|\phi^{2e}\nabla^2u|_{D^2}$.  Thus, it seems that at least, the reformulated system (\ref{eq:cccq}):
\begin{equation*}
\begin{cases}
\displaystyle
\phi_t+u\cdot \nabla \phi+(\gamma-1)\phi \text{div} u=0,\\[12pt]
\displaystyle
u_t+u\cdot\nabla u +\nabla \phi+a\phi^{2e}Lu=\psi \cdot Q(u),\\[12pt]
\displaystyle
\psi_t+\nabla (u\cdot \psi)+(\delta-1)\psi\text{div} u +\delta a\phi^{2e}\nabla \text{div} u=0,
 \end{cases}
\end{equation*}
can provide a closed a priori estimates for the desired regular solution. Next we need to find a proper linear scheme to verify the energy estimate strategy discussed above.

  Second, in \S 3.2,   we give  an  elaborate linearization (\ref{li4})  of the nonlinear one (\ref{eq:cccq})-(\ref{sfanb1}) based on a careful analysis on the structure of the nonlinear equations (\ref{eq:cccq}), and the  global approximate solutions for this linearized problem when $\phi(0,x)=\phi_0$ has positive lower bound $\eta$ are established.  The choice of the linear scheme for this problem needs to be  careful due to the appearance of the far field vacuum.
Some  necessary structures should be preserved  in order  to establish the desired a priori estimates according to the strategy mentioned in the above paragraph.  For the  problem (\ref{eq:cccq}), a crucial point is how to deal with the estimates on $\psi$. According to the analysis in the above paragraph, we need to keep the two factors $\phi^{2e}$  and $\nabla \text{div} u$ of the source term $\delta a\phi^{2e}\nabla \text{div} u$  in equations (\ref{psi-1}) in the same step. Then let $v=(v^{(1)},v^{(2)}, v^{(3)})\in \mathbb{R}^3$ be  a  known vector and $g$ be a known real (scalar) function satisfying $(v(0,x), g(0,x))=(u_0, (\phi_0)^{2e})$ and (\ref{vg}). It seems that one should consider the following linear equations:
  \begin{equation}
\begin{cases}
\label{eq:cccq-fenxi}
\displaystyle
\phi_t+v\cdot \nabla \phi+(\gamma-1)\phi \text{div} v=0,\\[12pt]
\displaystyle
u_t+v\cdot\nabla v +\nabla \phi+a\phi^{2e}Lu=\psi \cdot Q(v),\\[12pt]
\displaystyle
\psi_t+\sum_{l=1}^3 A_l(v) \partial_l\psi+B(v)\psi+\delta a g \nabla \text{div} v=0.
 \end{cases}
\end{equation}
 However,  it should be pointed out that, in (\ref{eq:cccq-fenxi}),   the relationship
 $$
 \psi=\frac{a\delta}{\delta-1}\nabla \phi^{2e}
  $$
  between $\psi$ and $\phi$ has been destroyed due to the term  $g \nabla \text{div} v$ in   $(\ref{eq:cccq-fenxi})_3$. Then the estimates for the linear scheme (\ref{eq:cccq-fenxi}) encounter an obvious difficulty when one considers the   $L^2$ estimate on $u$:
  \begin{equation}\label{relationvainish}
  \begin{split}
&\frac{1}{2} \frac{d}{dt}|u|^2_2+a\alpha |\phi^{e}\nabla u|^2_2+a(\alpha+\beta)|\phi^{e}\text{div} u|^2_2\\
=&-\int \big(v\cdot \nabla v +\nabla \phi  +a  \underbrace{\nabla \phi^{2e}}_{\neq \psi}\cdot Q(u) -\psi \cdot Q(v) \big)\cdot u.
  \end{split}
  \end{equation}
The factor $\nabla \phi^{2e}$ does not coincide with $\frac{\delta-1}{a\delta}\psi$ in this linear scheme, which means that there is no way to control the term $a  \nabla \phi^{2e}\cdot Q(u)$ in the above energy estimates.
In order to overcome this difficulty, in  (\ref{li4}),   we  first linearize the equation of $h=\phi^{2e}$ as:
\begin{equation}\label{h}h_t+v\cdot \nabla h+(\delta-1)g\text{div} v=0,
\end{equation}
 and then use $h$ to define $\psi=\frac{a\delta}{\delta-1}\nabla h$ again. The linearized equations for $u$ are chosen as
$$
u_t+v\cdot\nabla v +\nabla \phi+a \sqrt{h^2+\epsilon^2} Lu=\psi \cdot Q(v)
$$
for any positive constant $\epsilon >0$. Here the appearance of $\epsilon$ is used to compensate the lack of lower bound of $h$. It follows from (\ref{h}) and the relation  $\psi=\frac{a\delta}{\delta-1}\nabla h$ that
$$
\psi_t+\sum_{l=1}^3 A_l(v) \partial_l\psi+(\nabla v)^\top\psi+a\delta \big(g\nabla \text{div} v+\nabla g \text{div} v\big)=0,
$$
which  turns out to be enough to get desired estimates on $\psi$.

In \S 3.3,    the a priori estimates independent of the lower bound $(\epsilon,\eta)$ of the solutions $(\phi^{\epsilon,\eta},h^{\epsilon,\eta},\psi^{\epsilon,\eta},u^{\epsilon,\eta})$  to  the linearized problem   (\ref{li4}) are established. In order to deal with the limit process from our linear problem to the nonlinear one, we need some uniform estimates on the following new quantities:
$$
 \varphi=h^{-1},\quad f=\psi \varphi =\frac{a\delta}{\delta-1}\nabla h/h=(f^{(1)},f^{(2)},f^{(3)}).$$
An observation used in this subsection is that the initial assumption (\ref{th78qq})  implies that
$$
\varphi(0,x) \in L^\infty\cap D^{1,6} \cap D^{2,3} \cap D^3,\ \  f(0,x) \in L^\infty\cap L^6\cap  D^{1,3} \cap D^2.
$$
According to the definitions of $(\varphi,f)$ and the equation of $h$,  one can also show that $(\varphi,f)$ can be controlled by some hyperbolic equations without degenerate or singular source terms.
Based on these facts, we can obtain some uniform estimates on $\varphi$ and $f$, which are sufficient for the strong compactness argument  used  in \S 3.5.
   In \S 3.4,  one obtains the uniformly local-in-time well-posedness  of the  linearized problem when $\phi_0>\eta>0$ but without the artificial viscosity by passing to the limit $\epsilon \rightarrow 0$.

In \S 3.5, based on the above analysis for the choice of the linearization  and a new formulation of our problem (see (\ref{li6zx})),   the unique solvability of the classical solution away from vacuum  to  the nonlinear reformulated problem (\ref{eq:cccq})-(\ref{sfanb1}) through an iteration process  is given, whose life span is uniformly positive with respect to the lower bound $\eta$ of $\phi_0$.
%For example, let  $(\phi^k,\psi^k,u^k)$  and $(\phi^{k+1},\psi^{k+1},u^{k+1})$ be the solutions of the above linear iteration in the  $k$-th  step  and the $(k+1)$-th step  respectively. Then for the derivation of  the strong compactness for $u$ in some sense, we need to deal with the following  terms:
%$$
%I_1=(\phi^k)^{2e} L u^k-(\phi^{k-1})^{2e} Lu^{k-1}=((\phi^k)^{2e}-(\phi^{k-1})^{2e})L u^k+... \quad \text{and} \quad I_2=\psi^{k+1}-\psi^k.
%$$
%Actually, the estimate on $\psi^{k+1}-\psi^k$ will be very complicated due to the weighted   higher order terms of $u$ in their equations. A consideration for removing this difficulty is to multiply $\phi^{-2e}$ on both sides of $(\ref{eq:cccq-fenxi})_3$, then one has
%$$
%\phi^{-2e}(  u_t+v\cdot\nabla v +\nabla \phi)+aLu= \phi^{-2e}  \psi \cdot Q(v).
%$$
%If $\phi^{-2e}  \psi = \frac{a\delta}{\delta-1} \nabla \log \phi^{2e}$, then the similar argument on this quantity  in \S 3.5 of  \cite{sz3} can be applied to our case.  However,  the relationship between $\psi$ and $\phi$ has been destroyed due to the term  $g \nabla %\text{div} v$ in the linear scheme  $(\ref{eq:cccq-fenxi})_2$.
%Second, besides the consideration for the a priori estimates,  we must make sure that this linear scheme is good enough for the strong compactness (see \S 3.5) in some sense from the linear problem to the nonlinear one.
Actually, the behavior of $u$ in this subsection is controlled by the following linear equations:
$$
\varphi(u_t+v\cdot \nabla v+\nabla\phi)+aLu=f \cdot  Q(v),$$
which is also a special quasi-linear parabolic system with some possible singular source terms $f\cdot  Q(u)$ near the  vacuum, and  the coefficients $\varphi$ in front  of the time evolution operator $\dot{u}=u_t+u\cdot \nabla u $ will tend to $0$ as $\varphi\rightarrow 0$ in the far filed. However, based on the uniform estimates on $(\phi,u,h,\psi)$ and $(\varphi,f)$ established in \S 3.4, this structure  can avoid some difficulties on the strong convergence of terms such as
$
g\nabla \text{div} v,
$
and $hLu$
in the linear scheme  (\ref{li4}). The details can be found in the proof of this subsection.

Based on the conclusions of \S 3.5, one can   recover the solution of the  nonlinear reformulated problem allowing vacuum in the far field by passing to the limit as $\eta \rightarrow 0$ in  \S 3.6.
Then  in \S 3.7,  one can show that the existence result for the reformulated problem indeed implies Theorem \ref{th2}. Theorem \ref{th2-1} can be proved by a similar argument as used in Theorem \ref{th2}.

 Finally, Section $4$ is devoted to the proof of the non-existence theories of global regular solutions with $L^\infty$ decay on $u$ shown in Theorem \ref{th:2.20} and Corollaries \ref{th:2.20-c}-\ref{th:2.20-HLX}.
Furthermore,  it should be pointed out that our framework  in this paper can be applied   to other physical dimensions, say 1 and 2, with some minor modifications.

\section{Preliminaries}

In this section, we list  some basic lemmas  to be used later.
The first one is the  well-known Gagliardo-Nirenberg inequality.
\begin{lemma}\cite{oar}\label{lem2as}\
For $p\in [2,6]$, $q\in (1,\infty)$, and $r\in (3,\infty)$, there exists some generic constant $C> 0$ that may depend on $q$ and $r$ such that for
$$f\in H^1(\mathbb{R}^3),\quad \text{and} \quad  g\in L^q(\mathbb{R}^3)\cap D^{1,r}(\mathbb{R}^3),$$
it holds that
\begin{equation}\label{33}
\begin{split}
&|f|^p_p \leq C |f|^{(6-p)/2}_2 |\nabla f|^{(3p-6)/2}_2,\quad |g|_\infty\leq C |g|^{q(r-3)/(3r+q(r-3))}_q |\nabla g|^{3r/(3r+q(r-3))}_r.
\end{split}
\end{equation}
\end{lemma}
Some special cases of this inequality are
\begin{equation}\label{ine}\begin{split}
|u|_6\leq C|u|_{D^1},\quad |u|_{\infty}\leq C|u|^{\frac{1}{2}}_6|\nabla u|^{\frac{1}{2}}_{6}, \quad |u|_{\infty}\leq C\|u\|_{W^{1,r}}.
\end{split}
\end{equation}

The second lemma gives some compactness results obtained via the Aubin-Lions Lemma.
\begin{lemma}\cite{jm}\label{aubin} Let $X_0\subset X\subset X_1$ be three Banach spaces.  Suppose that $X_0$ is compactly embedded in $X$ and $X$ is continuously embedded in $X_1$. Then the following statements hold.

\begin{enumerate}
\item[i)] If $J$ is bounded in $L^p([0,T];X_0)$ for $1\leq p < +\infty$, and $\frac{\partial J}{\partial t}$ is bounded in $L^1([0,T];X_1)$, then $J$ is relatively compact in $L^p([0,T];X)$;\\

\item[ii)] If $J$ is bounded in $L^\infty([0,T];X_0)$  and $\frac{\partial J}{\partial t}$ is bounded in $L^p([0,T];X_1)$ for $p>1$, then $J$ is relatively compact in $C([0,T];X)$.
\end{enumerate}
\end{lemma}

The third  one  can be found in Majda \cite{amj}.
\begin{lemma}\cite{amj}\label{zhen1}
Let  $r$, $a$ and $b$  be constants such that
$$\frac{1}{r}=\frac{1}{a}+\frac{1}{b},\quad \text{and} \quad 1\leq a,\ b, \ r\leq \infty.$$  $ \forall s\geq 1$, if $f, g \in W^{s,a} \cap  W^{s,b}(\mathbb{R}^3)$, then it holds that
\begin{equation}\begin{split}\label{ku11}
&|\nabla^s(fg)-f \nabla^s g|_r\leq C_s\big(|\nabla f|_a |\nabla^{s-1}g|_b+|\nabla^s f|_b|g|_a\big),\\
\end{split}
\end{equation}
\begin{equation}\begin{split}\label{ku22}
&|\nabla^s(fg)-f \nabla^s g|_r\leq C_s\big(|\nabla f|_a |\nabla^{s-1}g|_b+|\nabla^s f|_a|g|_b\big),
\end{split}
\end{equation}
where $C_s> 0$ is a constant depending only on $s$, and $\nabla^s f$ ($s>1$) is the set of  all $\partial^\zeta_x f$  with $|\zeta|=s$. Here $\zeta=(\zeta_1,\zeta_2,\zeta_3)\in \mathbb{R}^3$ is a multi-index.
\end{lemma}

The following lemma is important in the derivation of  the a priori estimates  in Section $3$, which can be found in Remark 1 of \cite{bjr}.
\begin{lemma}\cite{bjr}\label{1}
If $f(t,x)\in L^2([0,T]; L^2)$, then there exists a sequence $s_k$ such that
$$
s_k\rightarrow 0, \quad \text{and}\quad s_k |f(s_k,x)|^2_2\rightarrow 0, \quad \text{as} \quad k\rightarrow+\infty.
$$
\end{lemma}

The following regularity estimate for the Lam$\acute{ \text{e}}$ operator is standard in harmonic analysis.
\begin{lemma}\cite{harmo}\label{zhenok}
If $u\in D^{1,q}(\mathbb{R}^3)$ with $1< q< +\infty$ is a weak solution to the problem
\begin{equation}
\label{ok}
\begin{cases}
\displaystyle
-\alpha\triangle u-(\alpha+\beta)\nabla \text{div}u =Lu=Z, \\[8pt]
\displaystyle
 u\rightarrow 0 \quad \text{as} \quad  |x|\rightarrow +\infty,
\end{cases}
\end{equation}
then it holds that
$$
|u|_{D^{k+2,q}} \leq C |Z|_{D^{k,q}},
$$
where the constant $C>0$ depends only on $\alpha$,  $\beta$  and $q$.
\end{lemma}

The final  lemma is useful  to improve a weak convergence to the strong convergence.
\begin{lemma}\cite{amj}\label{zheng5}
If the function sequence $\{w_n\}^\infty_{n=1}$ converges weakly  to $w$ in a Hilbert space $X$, then it converges strongly to $w$ in $X$ if and only if
$$
\|w\|_X \geq \lim \text{sup}_{n \rightarrow \infty} \|w_n\|_X.
$$
\end{lemma}

%%%%%%%%%%%%%%%%%%%section 3%%%%%%%%%%%%%%%%%%%%%%%%
\section{Local-in-time well-posedness  of regular solutions}
This section is devoted to  proving Theorem \ref{th2}. To this end, we first reformulate the original Cauchy problem (\ref{eq:1.1})-(\ref{far})  as  (\ref{eq:cccq})-(\ref{sfanb1}) below in terms of some new variables, and then establish the local well-posedness of the smooth solution to (\ref{eq:cccq})-(\ref{sfanb1}). In the end of this section, one can show that the existence result for the reformulated problem indeed implies Theorem \ref{th2}.

\subsection{Reformulation}
In terms of   variables
\begin{equation}\label{bianliang}\phi=\frac{A\gamma}{\gamma-1} \rho^{\gamma-1},\quad \psi=\frac{\delta}{\delta-1}\nabla \rho^{\delta-1}=\frac{\delta}{\delta-1}\Big(\frac{A\gamma}{\gamma-1}\Big)^{\frac{1-\delta}{\gamma-1}}\nabla \phi^{\frac{\delta-1}{\gamma-1}}=(\psi^{(1)},\psi^{(2)},\psi^{(3)})
\end{equation}
and $u$,  the system (\ref{eq:1.1})  can be rewritten as
\begin{equation}
\begin{cases}
\label{eq:cccq}
\displaystyle
\phi_t+u\cdot \nabla \phi+(\gamma-1)\phi \text{div} u=0,\\[12pt]
\displaystyle
u_t+u\cdot\nabla u +\nabla \phi+a\phi^{2e}Lu=\psi \cdot Q(u),\\[12pt]
\displaystyle
\psi_t+\nabla (u\cdot \psi)+(\delta-1)\psi\text{div} u +\delta a\phi^{2e}\nabla \text{div} u=0,
 \end{cases}
\end{equation}
where
\begin{equation} \label{xishu}
\begin{split}
& a=\Big(\frac{A\gamma}{\gamma-1}\Big)^{\frac{1-\delta}{\gamma-1}},\quad \text{and} \quad e=\frac{\delta-1}{2(\gamma-1)}<0.
\end{split}
\end{equation}
The  initial data is given by
\begin{equation} \label{sfana1}
\begin{split}
&(\phi, u,\psi)|_{t=0}=(\phi_0, u_0, \psi_0)=\Big(\frac{A\gamma}{\gamma-1} \rho^{\gamma-1}_0(x),   u_0(x), \frac{\delta}{\delta-1}\nabla \rho^{\delta-1}_0(x)\Big),\quad x\in \mathbb{R}^3.
\end{split}
\end{equation}
$(\phi, u,\psi)$ is assumed to satisfy the far field behavior:
\begin{equation}\label{sfanb1}
\begin{split}
(\phi, u,\psi)\rightarrow (0,0,0),\quad \text{as}\quad  |x|\rightarrow +\infty,\quad t \geq 0.
\end{split}
\end{equation}

 To prove Theorem \ref{th2}, our first step is to establish the following well-posedness to the reformulated problem (\ref{eq:cccq})-(\ref{sfanb1}).
\begin{theorem}\label{th1} Let (\ref{canshu}) hold. If the initial data $( \phi_0, u_0, \psi_0)$ satisfies:
\begin{equation}\label{th78qq}
\begin{split}
&\phi_0>0,\quad  (\phi_0, u_0)\in H^3, \quad  \psi_0 \in  D^1\cap D^2,\quad \nabla \phi^e_0\in L^4,
\end{split}
\end{equation}
and the initial  compatibility conditions:
\begin{equation}\label{th78zxq}
\displaystyle
 \nabla u_0=\phi^{-e}_0g_1,\quad     aLu_0=\phi^{-2e}_0g_2,\quad \nabla \big(a\phi^{2e}_0Lu_0\big)=\phi^{-e}_0g_3,
\end{equation}
for some $(g_1,g_2,g_3)\in L^2$,
then there exist a  time $T_*>0$ and a unique classical solution $\Big(\phi, u,\psi=\frac{a\delta}{\delta-1}\nabla \phi^{2e}\Big)$ to the Cauchy problem (\ref{eq:cccq})-(\ref{sfanb1}), satisfying
\begin{equation}\label{zhengzeA}
\begin{split}
& \phi \in C([0,T_*];H^3),\ \ \nabla \phi/ \phi \in L^\infty([0,T_*];L^\infty\cap L^6\cap  D^{1,3} \cap D^2),\\
& \psi \in C([0,T_*]; D^1\cap D^2), \ \   \phi^{-2e} \in L^\infty([0,T_*];L^\infty\cap D^{1,6} \cap D^{2,3} \cap D^3), \\
&  u\in C([0,T_*]; H^3)\cap L^2([0,T_*] ; H^4), \ \     \phi^{2e}\nabla u \in L^\infty([0,T_*]; D^1),\\
& \phi^{2e}\nabla^2 u \in C([0,T_*]; H^1),\ \  \phi^{2e}\nabla^2 u\in  L^2([0,T_*] ; D^2),\ \phi^{e}\nabla u\in  C([0,T_*]; L^2),\\
&  \phi^{e}\nabla u_t\in  L^\infty([0,T_*]; L^2),\ \ u_t \in C([0,T_*]; H^1)\cap L^2([0,T_*]; D^2),\\
&  (\phi^{2e}\nabla^2 u)_t \in  L^2([0,T_*] ; L^2),\ \  u_{tt}\in L^2([0,T_*];L^2),\  t^{\frac{1}{2}}u\in L^\infty([0,T_*];D^4),\\ & t^{\frac{1}{2}}u_t\in L^\infty([0,T_*];D^2)\cap L^2([0,T_*] ; D^3),\  t^{\frac{1}{2}}u_{tt}\in L^\infty([0,T_*];L^2)\cap L^2([0,T_*];D^1).
\end{split}
\end{equation}
\end{theorem}

This theorem will be proved  in the subsequent four Subsections $3.2$-$3.6$.
\subsection{Linearization}\label{linear2}

 In order to solve  the nonlinear problem (\ref{eq:cccq})-(\ref{sfanb1}), we  need to consider the corresponding  linearized problem.  Before this,  it is necessary to analyze the  structure of  the equations $(\ref{eq:cccq})$.
\subsubsection{Structure of the nonlinear equations (\ref{eq:cccq}).}
First,  due to the definition of $\psi$, if $\psi\in D^{1,2}(K)\cap D^{2,2}(K)$ for any compact set $K\subset \mathbb{R}^3$,
a direct calculation shows that
$$\partial_i \psi^{(j)}=\partial_j \psi^{(i)} \quad \text{for}\ i,j=1,2,3$$
 in the sense of distributions.  Thus, $(\ref{eq:cccq})_2$
can be rewritten as
\begin{equation}\label{kuzxc}
\psi_t+\sum_{l=1}^3 A_l \partial_l\psi+B\psi+\delta a\phi^{2e}\nabla \text{div} u=0,
\end{equation}
where  $A_l=(a^l_{ij})_{3\times 3}$ for  $i,j,l=1,2,3$,
are symmetric  with
$$a^l_{ij}=u^{(l)}\quad \text{for}\ i=j;\quad \text{otherwise}\  a^l_{ij}=0, $$
 and $B=(\nabla u)^\top+(\delta-1)\text{div}u\mathbb{I}_3$. This implies that  the subtle source term $\psi$  could be controlled by  the     symmetric hyperbolic  system (\ref{kuzxc}).

Second,  for equations $(\ref{eq:cccq})_3$, note that the coefficients $a \phi^{2e}$ in  front of the Lam\'e operator $ L$ will tend to $\infty$ as $\phi \rightarrow 0$ in the far filed. In order to make full use of this special structure,  though system (\ref{eq:cccq}) for $(\phi,u,\psi)$ has been already a closed one,  it  is helpful to get  some more precise estimates by introducing two auxiliary quantities:
\begin{equation}\label{guanxizai}
\varphi=\phi^{-2e},\quad f=\psi \varphi=\frac{2a e \delta }{\delta-1}\frac{\nabla \phi}{\phi}=(f^{(1)},f^{(2)},f^{(3)}).
\end{equation}

Next,  we will show formally the time evolution mechanisms of  $(\varphi,f)$  based on the initial regularities (\ref{th78qq}) and  system (\ref{eq:cccq}).
On the one hand, it follows  from (\ref{bianliang}) and (\ref{guanxizai}) that
\begin{equation}\label{guanxishi}
\psi_0= \frac{a\delta}{\delta-1}   \nabla \varphi^{-1}_0=-  \frac{a\delta}{\delta-1} \varphi^{-2}_0\nabla \varphi_0, \quad   f_0=\varphi_0 \psi_0,
\end{equation}
%According to  (\ref{th78qq}), we have
%\begin{equation}\label{shenqi}
%\begin{split}
%|\varphi_0|_\infty<& \infty,\quad  |\nabla \varphi_0|_6\leq |\varphi_0|^2_\infty |\psi_0|_6< \infty,\\
%|\nabla^2 \varphi_0|_3\leq& |\varphi_0|^2_\infty |\varphi^{-2}_0\nabla^2 \varphi_0|_3
%\leq |\varphi_0|^2_\infty\big(|\nabla^2 \varphi^{-1}_0|_3+|\varphi^{-3}_0|\nabla \varphi_0|^2|_3\big)\\
%\leq&|\varphi_0|^2_\infty\big(|\nabla \psi_0|_3+|\psi_0|^2_6|\varphi_0|_\infty\big)<\infty,
%\end{split}
%\end{equation}
%which, quickly implies that
%\begin{equation}\label{shenqi2}
%|\varphi^{-2}_0\nabla^2 \varphi_0|_3< \big(|\nabla \psi_0|_3+|\psi_0|^2_6|\varphi_0|_\infty\big)<\infty.
%\end{equation}
%Then for $\nabla^3 \varphi_0$, we have
%\begin{equation}\label{shenqi3}
%\begin{split}
%|\nabla^3 \varphi_0|_2\leq& |\varphi_0|^2_\infty |\varphi^{-2}_0\nabla^3 \varphi_0|_2\\
%\leq& C|\varphi_0|^2_\infty\big(|\nabla^3 \varphi^{-1}_0|_2+|\varphi^{-3}_0|\nabla \varphi_0||\nabla^2 \varphi_0| |_2+|\varphi^{-4}_0|\nabla \varphi_0|^3|_2\big)\\
%\leq& C|\varphi_0|^2_\infty\big(|\nabla^2 \psi_0|_2+|\psi_0|_3|\varphi^{-2}\nabla \varphi_0|_6|\varphi_0|_\infty+|\psi_0|^3_6|\varphi_0|^2_\infty\big)<\infty,
%\end{split}
%\end{equation}
which, along with (\ref{th78qq}), implies  that
\begin{equation}\label{chushi1}
\begin{split}
\varphi_0=&\phi^{-2e}_0 \in L^\infty\cap D^{1,6} \cap D^{2,3} \cap D^3,\ \  f_0=\frac{2a e \delta }{\delta-1}\nabla \phi_0/ \phi_0 \in L^\infty\cap L^6\cap  D^{1,3} \cap D^2.
\end{split}
\end{equation}

On the other hand,  it follows  from (\ref{bianliang}) and equations $(\ref{eq:cccq})$ that
\begin{equation}\label{qian1}
\begin{split}
\varphi_t+u\cdot \nabla \varphi-(\delta-1)\varphi\text{div} u=&0,\\
f_t+\nabla (u \cdot f)+a\delta\nabla \text{div} u=&0.
\end{split}
\end{equation}

 If $f \in D^{1,2}(K)\cap D^{2,2}(K)$ for any compact set $K\subset \mathbb{R}^3$,  a direct calculation shows that
$$\partial_i f^{(j)}=\partial_j f^{(i)} \quad \text{for}\ i,j=1,2,3$$
 in the sense of distributions.  Then $(\ref{qian1})_2$
can be written as
\begin{equation}\label{qian3}
f_t+\sum_{l=1}^3 A_l \partial_lf+B^*f+a\delta\nabla \text{div} u=0,
\end{equation}
where  $B^*=(\nabla u)^\top$. Thus $f$  satisfies   the     symmetric hyperbolic  system (\ref{qian3}).

It should be  pointed out that  a key observation here is that the structure in $(\ref{qian1})$-$(\ref{qian3})$ for $(\varphi,f)$ makes it possible to show  that  the subtle term $a \phi^{2e}Lu$ is well defined in $H^2$ when vacuum appears in the far field.

\subsubsection{Linearized problem for uniformly positive initial density and artificial viscosity.}
 Motivated by the above observations,  we will consider the following linearized problem for $(\phi,u,h)$:
\begin{equation}\label{li4}
\begin{cases}
\displaystyle
\phi_t+v\cdot \nabla \phi+(\gamma-1)\phi \text{div} v=0,\\[6pt]
\displaystyle
%f_t+\sum_{l=1}^3 A_l(v) \partial_l\psi+B^*(v)\psi+\nabla \text{div} v=0,\\[10pt]
%\displaystyle
u_t+v\cdot\nabla v +\nabla \phi+a \sqrt{h^2+\epsilon^2} Lu=\psi \cdot Q(v),\\[10pt]
\displaystyle
h_t+v\cdot \nabla h+(\delta-1)g\text{div} v=0,\\[6pt]
\displaystyle
(\phi,u,h)|_{t=0}=(\phi_0,u_0,h_0)=\big(\phi_0, u_0, (\phi_0)^{2e}\big),\quad x\in \mathbb{R}^3,\\[10pt]
\displaystyle
(\phi,u,h)\rightarrow (\phi^\infty,0,h^\infty),\quad \text{as}\quad  |x|\rightarrow +\infty,\quad t\geq 0,
 \end{cases}
\end{equation}
 where $\epsilon$  and $\phi^\infty$ are both  positive constants,
$$ h^\infty=(\phi^\infty)^{2e},\quad \psi=\frac{a\delta}{\delta-1}\nabla h,$$
 $v=(v^{(1)},v^{(2)}, v^{(3)})\in \mathbb{R}^3$ is a  known vector and $g$ is a known real (scalar) function satisfying $(v(0,x), g(0,x))=(u_0,h_0)=(u_0, (\phi_0)^{2e})$ and:
\begin{equation}\label{vg}
\begin{split}
&g\in L^\infty\cap C([0,T]\times \mathbb{R}^3),\quad \nabla g \in C([0,T]; H^2),\quad   g_t\in C([0,T];H^2), \\
& \nabla g_{tt} \in L^2([0,T]; L^2),   \ \  v\in C([0,T]; H^3)\cap L^2([0,T]; H^4), \quad t^{\frac{1}{2}}v\in L^\infty([0,T];D^4),\\
&  v_t \in C([0,T]; H^1)\cap L^2([0,T]; D^2),\quad  v_{tt}\in L^2([0,T];L^2),\\
&   t^{\frac{1}{2}}v_t\in L^\infty([0,T];D^2)\cap L^2([0,T]; D^3),\quad  t^{\frac{1}{2}}v_{tt}\in L^\infty([0,T];L^2)\cap L^2([0,T];D^1),
\end{split}
\end{equation}
where $T>0$ is an arbitrary constant.

Now  the following global well-posedness in  $[0,T]\times \mathbb{R}^3$ of a classical solution to (\ref{li4}) can be obtained by the standard theory (\cite{CK3,oar,amj}) at least when $\phi_0$ is uniformly positive and   $\epsilon>0$.

 \begin{lemma}\label{lem1} Let (\ref{canshu}) hold and $\epsilon>0$.
 Assume  that $(\phi_{0},u_0,h_0=(\phi_0)^{2e})$ satisfies
 \begin{equation}\label{zhenginitial}\begin{split}
& \eta<\phi_0,\ \  \phi_0-\phi^\infty \in   H^3,\  \  \ u_0\in H^3,
\end{split}
\end{equation}
 for some constant $\eta>0$.
 Then for any $T>0$,  there exists a unique classical solution $(\phi,u,h)$ in $[0,T]\times \mathbb{R}^3$ to  (\ref{li4}) such that
\begin{equation}\label{reggh}\begin{split}
&\phi-\phi^\infty\in C([0,T]; H^3), \quad h\in L^\infty\cap C([0,T]\times \mathbb{R}^3),\quad \nabla h \in C([0,T]; H^2),\\
& h_t\in C([0,T];H^2), \ \   u\in C([0,T]; H^3)\cap L^2([0,T]; H^4), \\
&   u_t \in C([0,T]; H^1)\cap L^2([0,T]; D^2),\ \  u_{tt}\in L^2([0,T];L^2),\quad  t^{\frac{1}{2}}u\in L^\infty([0,T];D^4),\\
&  t^{\frac{1}{2}}u_t\in L^\infty([0,T];D^2)\cap L^2([0,T]; D^3),\quad   t^{\frac{1}{2}}u_{tt}\in L^\infty([0,T];L^2)\cap L^2([0,T];D^1).
\end{split}
\end{equation}
\end{lemma}
\begin{remark}\label{initialdateofh}
For the initial assumption on $h_0$, due to
$$
\eta<\phi_0,\ \  \phi_0-\phi^\infty \in   H^3 \quad \text{and} \quad h_0=(\phi_0)^{2e},$$
it holds that
$$
h_0\in L^\infty  \quad \text{and} \quad   \nabla h_0\in H^2.
$$
\end{remark}

Next  we are going to establish the uniform a priori estimates independent of $(\epsilon, \eta)$  for  the unique solution $(\phi,u,h)$ to (\ref{li4})  obtained in Lemma \ref{lem1}.

\subsection{A priori estimates independent of  $(\epsilon, \eta)$.} Let $(\phi_{0},u_0,h_0=(\phi_0)^{2e})$ be a given initial data satisfying the hypothesis of Lemma \ref{lem1}, and assume that there exists a constant  $c_0>0$ independent o f $\eta$ such that
\begin{equation}\label{houmian}\begin{split}
 IN_0=2+\phi^\infty+\|\phi_0-\phi^\infty\|_{3}+|\nabla h_0|_{D^1\cap D^2}+\|u_0\|_{3}+|g_1|_2+|g_2|_2 \ \ \ &\\
+|g_3|_2+\|h^{-1}_0\|_{L^\infty\cap D^{1,6} \cap D^{2,3} \cap D^3}+\|\nabla h_0 /h_0\|_{L^\infty \cap L^6\cap D^{1,3}\cap D^2}\leq& c_0,
\end{split}
\end{equation}
where
\begin{equation}\label{g123definition}
\displaystyle
g_1=\phi^{e}_0 \nabla u_0,\quad     g_2=a\phi^{2e}_0 Lu_0\quad \text{and} \quad  g_3= \phi^{e}_0\nabla \big(a\phi^{2e}_0Lu_0\big).
\end{equation}
\begin{remark}\label{c0}
The choice of the constant $c_0$,  independent of  $\eta$,   will be verified in the limit  process from the non-vacuum problem to the one with far field vacuum in Subsection \S 3.6 (see (\ref{co-verfify})).
\end{remark}
\begin{remark}\label{Additional information}
First, it follows  from (\ref{g123definition}), $\phi_0>\eta$ and
the far field behavior of  $(\phi,u,h)$  shown in  $(\ref{li4})_5$
 that
\begin{equation}\label{xintuoq}\begin{cases}
\displaystyle
a L(\phi^{2e}_0 u_0)= g_2-\frac{\delta-1}{\delta}G(\psi_0, u_0), \\[10pt]
\displaystyle
\phi^{2e}_0 u_0\rightarrow 0 \quad \text{as} \quad |x|\rightarrow +\infty,
\end{cases}
\end{equation}
where
\begin{equation*}
\begin{split}
G=& \alpha  \psi_0\cdot \nabla u_0+\alpha \text{div}(u_0\otimes \psi_0)+(\alpha+\beta)\big(\psi_0 \text{div}u_0+\psi_0 \cdot (\nabla u_0)+u_0  \cdot \nabla \psi_0\big).
\end{split}
\end{equation*}
Then  Lemma \ref{zhenok} and (\ref{houmian}) imply  that
\begin{equation}\label{gai11qq}
\begin{split}
|\phi^{2e}_0 u_0|_{D^2}\leq C\big(|g_2|_2 +|G(\psi_0, u_0)|_2\big)
%\leq & C\big(|u_t|_2+|v|_6|\nabla v|_3 +|\nabla \phi|_2+|\psi|_\infty |\nabla v|_2\big)\\
%&+\big(|\psi|_\infty|\sqrt{h}\nabla u|_2|\varphi|^{\frac{1}{2}}_\infty+|\nabla \psi|_3| u|_6\big)\\
\leq C_1<&+\infty, \\
|\phi^{2e}_0\nabla^2u_0|_2\leq C(|\phi^{2e}_0 u_0|_{D^2}+|\nabla \psi_0|_6 |u_0|_3+|\psi_0|_\infty| \nabla u_0|_2|)
\leq C_1< & +\infty,
\end{split}
\end{equation}
where $C_1$ is a positive constant depending  on $(c_0, A, \alpha, \beta, \gamma, \delta)$, but is independent of $(\epsilon,\eta)$.

Second, due to $\nabla^2 \phi^{2e}_0\in L^2$ and (\ref{gai11qq}), one gets easily
\begin{equation}\label{gai11qqcv}
\begin{split}
|\phi^e_0\nabla^2 \phi_0|_2+| \phi^{e}_0\nabla (\psi_0\cdot Q(u_0))|_2\leq C_1<+\infty.
\end{split}
\end{equation}
\end{remark}

Now we fix $T>0$, and   assume that there exist some time $T^*\in (0,T]$ and constants $c_i$ ($i=1,...,5$) such that
$$1< c_0\leq c_1 \leq c_2 \leq c_3 \leq c_4\leq c_5, $$
and
\begin{equation}\label{jizhu1}
\begin{split}
\sup_{0\leq t \leq T^*}\|\nabla g(t)\|^2_{D^1\cap D^2}\leq c^2_1,\ \ \sup_{0\leq t \leq T^*}\| v(t)\|^2_{1}+\int_{0}^{T^*} \Big( | v|^2_{D^2}+|v_t|^2_{2}\Big)\text{d}t \leq& c^2_2,\\[2pt]
\sup_{0\leq t \leq T^*}\big(|v|^2_{D^2}+|v_t|^2_{2}+|g\nabla^2v|^2_{2}\big)(t)+\int_{0}^{T^*} \Big( |v|^2_{D^3}+|v_t|^2_{D^1}\Big)\text{d}t \leq& c^2_3,\\[2pt]
\sup_{0\leq t \leq T^*}\big(|v|^2_{D^3}+|v_t|^2_{D^1}+| g_t|^2_{D^1}\big)(t)+\int_{0}^{T^*} \Big( | v|^2_{D^4}+|v_t|^2_{D^2}+|v_{tt}|^2_2\Big)\text{d}t \leq& c^2_4,\\[2pt]
\sup_{0\leq t \leq T^*}\big(|g\nabla^2v|^2_{D^1}+|g_t|^2_\infty\big)(t)+\int_{0}^{T^*} \Big(|(g\nabla^2 v)_t|^2_{2}+|g\nabla^2v|^2_{D^2}\Big)\text{d}t \leq& c^2_4,\\[2pt]
\text{ess}\sup_{0\leq t \leq T^*}t\big(|v_t|^2_{D^2}+|v|^2_{D^4}+|v_{tt}|^2_{2}\big)(t)+\int_{0}^{T^*} t\big(|v_{tt}|^2_{D^1}+|v_{t}|^2_{D^3}\big)\text{d}t \leq& c^2_5.
\end{split}
\end{equation}
$T^*$ and  $c_i$ ($i=1,...,5$) will be determined  later (see \eqref{dingyi45}), and
depend only on $c_0$ and the fixed constants $(A, \alpha, \beta, \gamma, \delta, T)$.

 In the following we are going to establish a series of uniform local (in time) estimates independent of  $(\epsilon, \eta)$  listed as Lemmas  \ref{2}-\ref{5}.
Hereinafter,  $C\geq 1$ will denote  a generic positive constant depending only on fixed constants $(A, \alpha, \beta, \gamma, \delta, T)$.

\subsubsection{The a priori estimates for $\phi$.}
 Now we estimate $\phi$.

\begin{lemma}\label{2} Let $(\phi,u,h)$ be the unique classical solution to (\ref{li4}) in $[0,T] \times \mathbb{R}^3$. Then
\begin{equation}\label{diyi}
\begin{split}
\|\phi(t)-\phi^\infty\|^2_3\leq Cc^2_0,\quad  |\phi_t(t)|_2\leq Cc_0c_2,\quad  |\phi_t(t)|_{D^1}\leq& Cc_0c_3,\\
  |\phi_t(t)|_{D^2}\leq Cc_0c_4,\quad
  |\phi_{tt}(t)|_2\leq Cc^3_4,\quad  \int_0^t \|\phi_{tt}\|^2_1 \text{d}s\leq& Cc^2_0c^2_4,
\end{split}
\end{equation}
\text{for} $0\leq t \leq T_1=\min (T^{*}, (1+c_4)^{-2})$.
\end{lemma}

\begin{proof}
First, the stand energy estimates argument for transport equations and (\ref{jizhu1}) give
\begin{equation}\label{gb}\begin{split}
\|\phi(t)-\phi^\infty\|_{3}\leq& \Big(\|\phi_0-\phi^\infty\|_{3} +\phi^\infty\int_0^t \|\nabla v\|_{3}\text{d}s\Big)\exp\Big(C\int_0^t \| v\|_{4}\text{d}s\Big)\\
\leq & C c_{0} \quad \text{for}\quad 0\leq t \leq T_1=\min (T^{*}, (1+c_4)^{-2}).
\end{split}
\end{equation}

Second, it follows  from  the equation $(\ref{li4})_1$ and (\ref{ine}) that,  for $0\leq t \leq T_1$,
\begin{equation}\label{zhen6}
\begin{cases}
\ \  |\phi_t(t)|_2\leq  C\|v\|_1(\|\nabla \phi\|_1+|\phi|_\infty)\leq  Cc_0c_2,\\[6pt]
 |\phi_t(t)|_{D^1}\leq  C\|v\|_2(\|\nabla \phi\|_1+|\phi|_\infty) \leq  Cc_0c_3,\\[6pt]
  |\phi_t(t)|_{D^2}\leq  C\|v\|_3(\|\nabla \phi\|_2+|\phi|_\infty)
\leq  Cc_0c_4.
\end{cases}
\end{equation}

At last, using the relation
$$\phi_{tt}=-v_t\cdot \nabla \phi-v\cdot \nabla \phi_t-(\gamma-1)\phi_t\text{div} v-(\gamma-1)\phi\text{div} v_t,$$
and the assumption (\ref{jizhu1}), one has,  for $0\leq t \leq T_1$, that
\begin{equation}\label{zhen7}
\begin{split}
|\phi_{tt}(t)|_2
\leq C\big(\|v_t\|_1(|\phi|_\infty+\|\nabla \phi\|_2)+\|  \phi_t\|_1\|v\|_2\big)\leq & Cc^3_4,\\
\int_0^t \|\phi_{tt}\|^2_1\text{d}s
\leq C\int_0^t \big(\|v_t\|^2_2(\|\nabla \phi\|_2+|\phi|_\infty)+\|v\|^2_3\|\phi_t\|^2_2\big) \text{d}s
\leq & Cc^2_0c^2_4.
\end{split}
\end{equation}
\end{proof}

\subsubsection{The a priori estimates for $\psi$.}
Next, we estimate $\psi$, which will be used to  deal with the degenerate elliptic operator.
\begin{lemma}\label{3} Let $(\phi,u,h)$ be the unique classical solution to (\ref{li4}) in $[0,T] \times \mathbb{R}^3$. Then
\begin{equation}\label{psi}
\begin{split}
&|\psi(t)|^2_\infty+\|\psi(t)\|^2_{D^1\cap D^2}\leq Cc^2_0,\quad |\psi_t(t)|_2\leq Cc^2_3,\quad |h_t(t)|^2_\infty\leq Cc^3_3c_4,\\
&|\psi_t(t)|^2_{D^1}+\int_0^t\big( |\psi_{tt}|^2_{2}+|h_{tt}|^2_{6}\big)\text{d}s\leq Cc^4_4,\quad \text{for}\quad  0\leq t \leq T_1.
\end{split}
\end{equation}
\end{lemma}
\begin{proof}Due to $\psi=\frac{a\delta}{\delta-1}\nabla h$  and  the equation $(\ref{li4})_3$, $\psi$ satisfies the following system:
\begin{equation}\label{kuzxclinear}
\psi_t+\sum_{l=1}^3 A_l(v) \partial_l\psi+B^*(v)\psi+a\delta \big(g\nabla \text{div} v+\nabla g \text{div} v\big)=0.
\end{equation}

First,  set $\varsigma=(\varsigma_1,\varsigma_2,\varsigma_3)^\top$ ($1\leq |\varsigma|\leq 2$ and $\varsigma_i=0,1,2$). Applying  $\partial_{x}^{\varsigma} $ to $(\ref{kuzxclinear})$,
%\begin{equation}\label{hyp}\begin{split}
%&(\partial_{x}^{\varsigma}  \psi)_t+\sum_{l=1}^3 A_l \partial_l\partial_{x}^{\varsigma}  \psi+B\partial_{x}^{\varsigma}  \psi+ \delta a\partial_{x}^{\varsigma} (h^{2e} \nabla \text{div} v) \\
%=&\Big(-\partial_{x}^{\varsigma} (B\psi)+B\partial_{x}^{\varsigma}  \psi\Big)+\sum_{l=1}^3 \Big(-\partial_{x}^{\varsigma} (A_l \partial_l \psi)+A_l \partial_l\partial_{x}^{\varsigma}  \psi\Big)=f+f_2.
%\end{split}
%\end{equation}
multiplying by $2\partial_{x}^{\varsigma} \psi$ and then integrating over $\mathbb{R}^3$, one can get

\begin{equation}\label{zhenzhen}\begin{split}
\frac{d}{dt}|\partial_{x}^{\varsigma}  \psi|^2_2
\leq & \Big(\sum_{l=1}^{3}|\partial_{l}A_l|_\infty+|B^*|_\infty\Big)|\partial_{x}^{\varsigma}  \psi|^2_2+|\Theta_\varsigma |_2|\partial_{x}^{\varsigma}  \psi|_2,
\end{split}
\end{equation}
where
\begin{equation*}
\begin{split}
\Theta_\varsigma=\partial_{x}^{\varsigma} (B^*\psi)-B^*\partial_{x}^{\varsigma}  \psi+\sum_{l=1}^{3}\big(\partial_{x}^{\varsigma} (A_l \partial_l \psi)-A_l \partial_l\partial_{x}^{\varsigma}  \psi\big)+ a\delta \partial_{x}^{\varsigma}\big(g\nabla \text{div} v+\nabla g \text{div} v\big).
\end{split}
\end{equation*}

For $|\varsigma|=1$, it is easy  to obtain
\begin{equation}\label{zhen2}
\begin{split}
|\Theta_\varsigma |_2\leq &C\big(|\nabla^2 v|_2(|\psi|_\infty+|\nabla g|_\infty)+ |\nabla v|_\infty( |\nabla\psi|_2+|\nabla^2 g|_2)+|g\nabla^2 v|_{D^1}\big).
\end{split}
\end{equation}
Similarly, for $|\varsigma|=2$, one has
\begin{equation}\label{zhen2}
\begin{split}
|\Theta_\varsigma |_2\leq& C\big(|\nabla v|_\infty(|\nabla^2\psi|_2+|\nabla^3 g|_2)+|\nabla^2 v|_3(|\nabla\psi|_6+|\nabla^2 g|_6)\big)\\
&+C|\nabla^3 v|_2 (|\psi|_\infty+|\nabla g|_\infty)+C|g\nabla \text{div}v|_{D^2}.
\end{split}
\end{equation}
It follows from
(\ref{zhenzhen})-(\ref{zhen2}) and Gagliardo-Nirenberg inequality that
\begin{equation*}
\frac{d}{dt}\|\psi(t)\|_{D^1\cap D^2}\leq Cc_4\|\psi(t)\|_{D^1\cap D^2}+C|g\nabla \text{div}v|_{D^2}+C c^2_4,
\end{equation*}
which, along with the Gronwall's inequality,  implies that for $0\leq t \leq T_1$,
\begin{equation}\label{uu2}\begin{split}
\|\psi(t)\|_{D^1\cap D^2}\leq&  \Big(c_0+Cc^2_4t+C\int_0^t |g\nabla \text{div}v|_{D^2} \text{d}s\Big) \exp(Cc_4t)\leq Cc_0.
\end{split}
\end{equation}

Second, according to  equations $(\ref{kuzxclinear})$,
for  $0\leq t \leq T_1$, it holds that
\begin{equation}\label{uu3}\begin{cases}
|\psi_t(t)|_2\leq C\big(|v|_{\infty}| \psi|_{D^1}+|\nabla v|_2|\psi|_{\infty}+|g\nabla^2 v|_2+|\nabla g|_\infty |\nabla v|_2\big)\leq Cc^2_3,\\[10pt]
|\nabla \psi_t(t)|_{2}\leq C \big(\|v\|_{3}(\|\psi\|_{D^1\cap D^2}+\|\nabla g\|_{D^1\cap D^2})+|g\nabla^2 v|_{D^1}\big) \leq Cc^2_4.
\end{cases}
\end{equation}
Similarly, using
$$\psi_{tt}=-\nabla (v \cdot \psi)_t-a\delta \big(g\nabla \text{div} v+\nabla g \text{div} v\big)_t,$$
for $0\leq t \leq T_1$,  one gets
\begin{equation}\label{uu4}\begin{split}
\int_0^t |\psi_{tt}|^2_{2} \text{d}s
\leq& C\int_0^t \big(|v_t|^2_6|\nabla \psi|^2_3+|\nabla v|^2_\infty|\psi_t|^2_{2}+|v|^2_\infty|\nabla \psi_t|^2_2+|\psi|^2_\infty|\nabla v_t|^2_{2}\big) \text{d}s\\
&+\int_0^t\big(|(g \nabla \text{div} v)_t|^2_{2}+|\nabla g|^2_\infty|\nabla v_t|^2_2+|\nabla v|^2_\infty|\nabla g_t|^2_2\big) \text{d}s
\leq Cc^4_4.
\end{split}
\end{equation}

Finally, it follows from  (\ref{ine}) and (\ref{jizhu1}) that
\begin{equation}\label{qianline1}
\begin{split}
|g\text{div} v|_\infty\leq & C|g\text{div} v|^{\frac{1}{2}}_{D^1}|g\text{div} v|^{\frac{1}{2}}_{D^2}
\leq  C\big(|\nabla g|_\infty|\nabla v|_2+|g\nabla^2 v|_2\big)^{\frac{1}{2}}\\
&\cdot \big(|\nabla^2 g|_2|\nabla v|_\infty+|\nabla g|_\infty|\nabla^2 v|_2+|g\nabla^2v|_{D^1}\big)^{\frac{1}{2}}
\leq  Cc^{\frac{3}{2}}_3c^{\frac{1}{2}}_4.
\end{split}
\end{equation}
Then, together   with $(\ref{li4})_3$,  one gets easily   that for $0\leq t \leq T_1$,
\begin{equation}\label{qianlineY}
\begin{split}
|h_t(t)|_\infty\leq C(|v|_\infty|\psi|_\infty+|g\text{div}v|_\infty)\leq& Cc^{\frac{3}{2}}_3c^{\frac{1}{2}}_4,\\
\int_0^t |h_{tt}|^2_6\text{d}s\leq C\int_0^t \big(|v|_\infty|\psi_t|_6+|v_t|_6|\psi|_\infty+|g_{t}|_\infty|\nabla v|_6+|g\nabla v_t|_6\big)^2 \text{d}s\leq & Cc^4_4,
\end{split}
\end{equation}
where  one has used the fact that
\begin{equation*}
\begin{split}
|g\nabla v_t|_6\leq& C\big(|\nabla g|_\infty|\nabla v_t|_2+| g\nabla^2 v_t|_2)\\
\leq &C\big(|\nabla g|_\infty|\nabla v_t|_2+| (g\nabla^2 v)_t|_2+| g_t|_\infty|\nabla^2 v|_2)\leq Cc^2_4.
\end{split}
\end{equation*}
\end{proof}

\subsubsection{The a priori estimates for two $h$-related auxiliary variables.}

In order to obtain the uniformly a priori estimates and life span independent of the lower bound $\eta$ of $\phi_0$ for the solutions to the corresponding nonlinear problem, it  is helpful to give some more precise estimates for  another two new  $h$-related quantities $\varphi$ and $f$:
$$
\varphi=h^{-1},\quad f=\psi \varphi =\frac{a\delta}{\delta-1}\nabla h/h=(f^{(1)},f^{(2)},f^{(3)}).
$$
\begin{lemma}\label{2zx} Let $(\phi,u,h)$ be the unique classical solution to (\ref{li4}) in $[0,T] \times \mathbb{R}^3$. Then
\begin{equation}\label{diyizx}
\begin{split}
\|\varphi(t)\|^2_{D^{1,6}\cap D^{2,3}\cap D^3}+\|f(t)\|^2_{L^\infty\cap L^6\cap D^{1,3}\cap D^2}\leq & Cc^4_0,\\
h(t,x)>\frac{1}{2c_0},\quad \frac{2}{3}\eta^{-2e}<\varphi(t,x)<2|\varphi_0|_\infty\leq & 2c_0,\\
\|\varphi_t(t)\|^2_{L^6\cap D^{1,3}\cap D^2}+\|f_t(t)\|^2_{L^3\cap D^{1}}\leq & Cc^{10}_4,
\end{split}
\end{equation}
for $0\leq t \leq T_2=\min\{T_1, (1+Cc_4)^{-4}\}$.
\end{lemma}

\begin{proof}\textbf{Step 1:} Estimates on $\varphi$.
It is easy to see  that $\varphi$ satisfies the following equation
\begin{equation}\label{qianline1}
\varphi_t+v\cdot \nabla \varphi-(\delta-1)g\varphi^2\text{div} v=0.
\end{equation}

First, along with the particle path $X(t;x_0)$ defined by
\begin{equation}\begin{cases}
\label{particle1}
\frac{d}{ds}X(t;x_0)=v(s,X(t; x_0)),\quad  0\leq t\leq T;\\[8pt]
X(0;x_0)=x_0, \quad \qquad \qquad \ \ \quad   x\in \mathbb{R}^3,
\end{cases}
\end{equation}
one has
\begin{equation}\label{qianline2}
\begin{split}
\displaystyle
\varphi(t,X(t;x_0))=\varphi_0(x_0)\Big(1+(1-\delta)\varphi_0(x_0)\int_0^t g\text{div}v (s,X(s;x_0))\text{d}s\Big)^{-1},
\end{split}
\end{equation}
which, along with (\ref{qianline1}), implies that  for $T_2=\min\{T_1, (1+Cc_4)^{-4}\}$,
\begin{equation}\label{qianline3}
\begin{split}
\displaystyle
\frac{2}{3}\eta^{-2e}<\varphi(t,x)<2|\varphi_0|_\infty\leq 2c_0,\quad \text{for} \quad  [t,x]\in [0,T_2]\times \mathbb{R}^3.
\end{split}
\end{equation}

Second, by the standard energy estimates for transport equations, one can obtain
\begin{equation*}\begin{split}
\frac{d}{dt}|\nabla\varphi|_{6}\leq& CF(t)|\nabla\varphi|_{6}+C|\varphi|^2_\infty\big( |g\nabla^2 v|_{6}+ |\nabla v|_{\infty} |\nabla g|_6\big),\\[2pt]
\frac{d}{dt}|\nabla^2\varphi|_{3}\leq& CF(t)|\nabla^2\varphi|_{3}+C|\nabla \varphi|_6\big(|\nabla^2 v|_6+|\nabla \varphi|_6|g\text{div} v|_\infty\big)\\
&+C|\varphi|^2_\infty \big(|g\nabla^2 \text{div} v|_{3}+|\nabla^2 g|_{3}|\nabla v|_\infty+|\nabla g|_\infty|\nabla^2 v|_3\big)\\
&+C|\varphi|_\infty|\nabla \varphi|_6\big(|g \nabla^2 v|_6+|\nabla g|_\infty|\nabla v|_6\big),\\[2pt]
%\end{split}
%\end{equation*}
%and
%\begin{equation*}\begin{split}
\frac{d}{dt}|\nabla^3\varphi|_{2}\leq& CF(t)|\nabla^3\varphi|_{2}+C\big(|\nabla \varphi|_6 |\nabla^3 v|_{3}+|\nabla^2 \varphi|_3|\nabla^2 v|_6\big)\\
&+C|\varphi|^2_\infty (|g\nabla^3\text{div} v|_{2}+|\nabla v|_\infty|\nabla^3 g|_2+|\nabla^2 g|_{6}|\nabla^2 v|_3+|\nabla g|_\infty|\nabla^3 v|_2)\\
&+C|g\text{div}v|_\infty|\nabla \varphi|_6|\nabla^2 \varphi|_3+C|\nabla \varphi|^2_6(|\nabla g|_\infty|\nabla v|_6+|g\nabla^2 v|_6)\\
&+C|\varphi|_\infty|\nabla^2 \varphi|_3(|\nabla g|_\infty|\nabla v|_6+|g\nabla^2 v|_6)\\
&+C|\varphi|_\infty|\nabla \varphi|_6(|\nabla g|_\infty|\nabla^2 v|_3+|\nabla^2 g|_6|\nabla v|_6+|g\nabla^2\text{div} v|_3),
\end{split}
\end{equation*}
where  $F(t)=|\nabla v|_{\infty}+|\varphi|_\infty|g\text{div}v|_\infty$.
It follows from the Gronwall's inequality that
$$
|\varphi(t)|^2_{D^{1,6}}+|\varphi(t)|^2_{D^{2,3}}+|\varphi(t)|^2_{D^3}\leq Cc^{2}_0,\quad \text{for}\ 0\leq t \leq T_2.
$$

Finally, due to the equation (\ref{qianline1}),
 for $0\leq t \leq T_2$, it holds that
\begin{equation}\label{zhen6zx}
\begin{split}
|\varphi_t(t)|_6\leq & C\big(|v|_\infty|\nabla \varphi|_6+|\varphi|^2_\infty|g\text{div} v|_6\big)\leq Cc^4_3,\\
|\nabla \varphi_t(t)|_{3}\leq & C\big(|v|_\infty|\nabla^2 \varphi|_{3}+|\nabla v|_6|\nabla \varphi|_6+|g\nabla v|_6|\nabla \varphi|_6|\varphi|_\infty\big)\\
&+C|\varphi|^2_\infty\big(|\nabla  g|_{\infty}|\nabla v|_3+|g\nabla^2 v|_3\big)\leq Cc^4_3,\\
%\end{split}
%\end{equation}
%and
%\begin{equation}\label{zhen6zxs}
%\begin{split}
|\nabla^2\varphi_t(t)|_{2}\leq & C\big(|v|_\infty|\nabla^3 \varphi|_{2}+|\nabla v|_6 |\nabla^2 \varphi|_3+|\nabla^2 v|_3|\nabla \varphi|_6\big)\\
&+C|\varphi|^2_\infty\big(|\nabla  g|_{\infty}|\nabla^2 v|_2+|g\nabla^3 v|_2+|\nabla v|_\infty|\nabla^2 g|_2\big)\\
&+C|\varphi|_\infty|\nabla \varphi|_6(|g\nabla^2 v|_3+|\nabla  g|_{\infty}|\nabla v|_3)\\
&+C|g\nabla v|_6 \big(|\nabla \varphi|^2_6+|\varphi|_\infty|\nabla^2 \varphi|_3\big)\leq Cc^4_4.
\end{split}
\end{equation}

\textbf{Step 2:} Estimates on $f$.
Based on the estimates obtained in Lemma \ref{3} and the above step, due to
$f=\psi \varphi$, it is easy to check that for $0\leq t \leq T_2$,
\begin{equation*}
\begin{split}
|f(t)|_\infty \leq& Cc^{2}_0,\ \
|f(t)|_6 \leq |\psi|_6 |\varphi|_\infty  \leq Cc^{2}_0,\\
 |\nabla f(t)|_3\leq&  C\big(|\nabla \varphi|_6|\psi|_6+|\varphi|_\infty|\nabla \psi|_3\big)\leq Cc^{2}_0,\\
|\nabla^2 f(t)|_2\leq&  C\big(|\varphi|_\infty |\nabla^2 \psi|_2+|\psi|_6|\nabla^2 \varphi|_3+|\nabla \varphi|_6|\nabla \psi|_3\big)
\leq Cc^{2}_0,\\
 |f_t(t)|_3\leq & C\big(|\psi|_{6}| \varphi_t|_{6}+|\varphi|_\infty|\psi_t|_{3}\big)(t)\leq Cc^{5}_4,\\
 |\nabla f_t(t)|_{2}\leq & C \big(|\psi|_{6}| \nabla\varphi_t|_{3}+|\varphi|_\infty|\nabla \psi_t|_{2}+|\nabla\psi|_{3}| \varphi_t|_{6}+|\nabla \varphi|_6|\psi_t|_{3}\big)(t) \leq Cc^{5}_4.
\end{split}
\end{equation*}
\end{proof}
It follows from the definition that $f$ satisfies the following equations:
\begin{equation}\label{qian3linear}
\begin{split}
&f_t+\sum_{l=1}^3 A_l(v) \partial_lf+B^*(v)f+a\delta \big(g \varphi \nabla \text{div} v+\varphi\nabla g \text{div}v\big) -(\delta-1)g \varphi f\text{div}v=0.
\end{split}
\end{equation}

\subsubsection{The a priori estimates for $u$.}
Based on the estimates of $\phi$ and $h$ obtained in Lemmas \ref{2}-\ref{2zx},  we are now ready to give  the lower order estimates for  the velocity $u$ as follows.

 \begin{lemma}\label{4} Let $(\phi,u,h)$ be the unique classical solution to (\ref{li4}) in $[0,T] \times \mathbb{R}^3$. Then
\begin{equation}\label{uu}
\begin{split}
|\sqrt{h}\nabla u(t)|^2_2+\| u(t)\|^2_{1}+\int_{0}^{t} \Big( \|\nabla u\|^2_{1}+|u_t|^2_{2}\Big)\text{d}s \leq& Cc^{4}_0,\\
(|u|^2_{D^2}+|h \nabla^2 u|^2_2+|u_t|^2_{2})(t)+
\int_{0}^{t} \Big( |u|^2_{D^3}+|u_t|^2_{D^1})\text{d}s \leq& Cc^{9}_2c_3,
\end{split}
\end{equation}
for $0 \leq t \leq T_3=\min(T_2,(1+c_4)^{-14})$.
 \end{lemma}
\begin{proof} \textbf{Step 1:} Estimate on  $|u|_2$. Multiplying  $(\ref{li4})_2$ by $u$ and integrating over $\mathbb{R}^3$, one gets from Gagliardo-Nirenberg inequality, H\"older's inequality and Young's inequality that
\begin{equation}\label{zhu1}
\begin{split}
&\frac{1}{2} \frac{d}{dt}|u|^2_2+a\alpha |(h^2+\epsilon^2)^{\frac{1}{4}}\nabla u|^2_2+a(\alpha+\beta)|(h^2+\epsilon^2)^{\frac{1}{4}}\text{div} u|^2_2\\
=&-\int \big(v\cdot \nabla v +\nabla \phi  +a\nabla \sqrt{h^2+\epsilon^2}\cdot Q(u) -\psi \cdot Q(v) \big)\cdot u \\
\leq & C\big(|v|_\infty|\nabla v|_2+ |\nabla\phi|_2+ |\psi|_\infty |\sqrt{h}\nabla u|_2|\varphi|^{\frac{1}{2}}_\infty+ |\psi|_\infty|\nabla v|_2\big)|u|_2\\
\leq& Cc^{3}_3|u|^2_2+Cc_3+\frac{1}{2}a\alpha |\sqrt{h}\nabla u|^2_2,
\end{split}
\end{equation}
%Integrating (\ref{zhu3}) over $(0,t)$, for $0 \leq t \leq T_2$, it gives
%\begin{equation*}
%\begin{split}
%|u(t)|^2_2+\frac{\alpha a}{2}\int_0^t\ |\sqrt{h}\nabla u(s)|^2_2\text{d}s\leq  Cc^{2K}_0\int_0^t |u(s)|^2_2 \text{d}s+C|u_0|^2_2+Cc^{4}_3t,
%\end{split}
%\end{equation*}
which, along with the Gronwall's inequality,  implies immediately that  for $0 \leq t \leq T_2$,
\begin{equation}\label{zhu5}
\begin{split}
|u(t)|^2_2+\frac{a \alpha}{2}\int_0^t |\sqrt{h}\nabla u|^2_2\text{d}s\leq  C\big(|u_0|^2_2+c_3 t\big)\exp(Cc^{3}_3t)\leq Cc^2_0.
\end{split}
\end{equation}

\textbf{Step 2:} Estimate on $|\nabla u|_2$. Multiplying $ (\ref{li4})_2$ by $u_t$ and integrating over $\mathbb{R}^3$, one gets  from the Gagliardo-Nirenberg, H\"older's and Young's inequalities that
\begin{equation}\label{zhu6}
\begin{split}
&\frac{1}{2} \frac{d}{dt}\Big(a\alpha|(h^2+\epsilon^2)^{\frac{1}{4}}\nabla u|^2_2+a(\alpha+\beta)|(h^2+\epsilon^2)^{\frac{1}{4}}\text{div} u|^2_2\Big)+| u_t|^2_2\\
=&-\int \big(v\cdot \nabla v +\nabla \phi +a\nabla \sqrt{h^2+\epsilon^2}\cdot Q(u) +\psi \cdot Q(v)\big)\cdot u_t \\
&-\int a \frac{h}{\sqrt{h^2+\epsilon^2}}h_t\big(\alpha|\nabla u|^2+(\alpha+\beta)|\text{div}u|^2\big)  \\
\leq & C\big(|v|_\infty|\nabla v|_2+  | \nabla \phi|_2+ |\psi|_\infty |\sqrt{h}\nabla u|_2|\varphi|^{\frac{1}{2}}_\infty+|\psi|_\infty|\nabla v|_2\big)|u_t|_2\\
&+C|h_t|_\infty|\varphi|_\infty|\sqrt{h}\nabla u|^2_2
\leq Cc^{3}_4|\sqrt{h}\nabla u|^2_2+Cc^4_3+\frac{1}{2}| u_t|^2_2,
\end{split}
\end{equation}
%Integrating (\ref{zhu6}) over $(\tau,t)$ $(\tau \in( 0,t))$  for any  $0<t\leq T_2$,  we have
%\begin{equation}\label{zhu11}
%\begin{split}
%&a\alpha|\sqrt{h}\nabla u(t)|^2_2+\frac{1}{2}\int_\tau^t| u_t|^2_2\text{d}s
%\leq Cc^{3}_4\int_\tau^t |\sqrt{h}\nabla u|^2_2\text{d}s+  C|\sqrt{h}\nabla u(\tau)|^2_2+Cc^{4}_3t.
%\end{split}
%\end{equation}
%From the regularities (\ref{reggh}), we have
%\begin{equation}\label{kaka}\begin{split}
%&\lim \sup_{\tau\rightarrow 0}|\sqrt{h}\nabla u(\tau)-\sqrt{h}_0\nabla u_0|^2_2\\
%\leq &C\lim \sup_{\tau\rightarrow 0}\Big(|\varphi|^{\frac{1}{2}}_\infty|h(\tau)-h_0|_\infty|\nabla u(\tau)|_2+|\sqrt{h}_0|_\infty|\nabla u(\tau)-\nabla u_0|_2\Big)=0.
%\end{split}
%\end{equation}
which, along with the Gronwall's inequality,   implies that  for $0\leq t\leq T_2$,
\begin{equation}\label{shan1}
\begin{split}
|\sqrt{h}\nabla u(t)|^2_2+\int_0^t| u_t|^2_2\text{d}s
\leq C(c^2_0+c^{4}_3t)\exp (Cc^{3}_4t)\leq Cc^2_0, \ \text{and} \   |u(t)|_{D^1} \leq  Cc^{\frac{3}{2}}_0.
\end{split}
\end{equation}

It follows from the definitions of the Lam$\acute{ \text{e}}$ operator $L$ and  $\psi$ that
\begin{equation}\label{xintuo}
\begin{split}
a L(\sqrt{h^2+\epsilon^2} u)=&a \sqrt{h^2+\epsilon^2} Lu- G(\nabla \sqrt{h^2+\epsilon^2},u)\\
=&-u_t-v\cdot\nabla v -\nabla \phi+\psi\cdot Q(v) -G(\nabla \sqrt{h^2+\epsilon^2},u).
\end{split}
\end{equation}
Then Lemma \ref{zhenok} implies that
\begin{equation}\label{gai11}
\begin{split}
|\sqrt{h^2+\epsilon^2}u(t)|_{D^2}\leq& C\big(\big|u_t+v\cdot\nabla v +\nabla \phi-\psi\cdot Q(v)|_2+|G(\nabla \sqrt{h^2+\epsilon^2},u)\big|_2\big)\\
%\leq & C\big(|u_t|_2+|v|_6|\nabla v|_3 +|\nabla \phi|_2+|\psi|_\infty |\nabla v|_2\big)\\
%&+\big(|\psi|_\infty|\sqrt{h}\nabla u|_2|\varphi|^{\frac{1}{2}}_\infty+|\nabla \psi|_3| u|_6\big)\\
\leq & C\big(|u_t|_2+c^{\frac{7}{2}}_2c^{\frac{1}{2}}_3\big), \\
 |\sqrt{h^2+\epsilon^2}\nabla^2u(t)|_2
\leq & C(|\sqrt{h^2+\epsilon^2} u|_{D^2}+|\nabla \psi|_3 |u|_6+|\psi|_\infty| \nabla u|_2|+|\psi|^2_\infty|  u|_2|\varphi|_\infty)\\
\leq  &C(|\sqrt{h^2+\epsilon^2} u|_{D^2}+c^4_0\big),
\end{split}
\end{equation}
for $0\leq t \leq T_2$, where one has used the fact that
\begin{equation*}\begin{split}
|v\cdot\nabla v|_2\leq& C|v|_6|\nabla v|_3\leq C|\nabla v|^{\frac{3}{2}}_2|\nabla^2 v|^{\frac{1}{2}}_2,\ \ |\nabla \psi|_3| u|_6 \leq Cc_0|\nabla u|_2 \leq Cc^{\frac{5}{2}}_0.
\end{split}
\end{equation*}
According to (\ref{shan1}) and  (\ref{gai11}), one gets
\begin{equation*}
\displaystyle
\int_0^{t} \big( |h\nabla^2u|^2_{2} + | \nabla^2u|^2_{2}\big)\text{d}s\leq Cc^{4}_0,\quad \text{for} \quad 0\leq t \leq T'=\min(T_2,(1+c_4)^{-8}).
\end{equation*}

\textbf{Step 3:} Estimate on $|u|_{D^2}$. First, applying $\partial_t$ to  $(\ref{li4})_2$ yields
\begin{equation}\label{zhu7}
\begin{split}
u_{tt}+a \sqrt{h^2+\epsilon^2} Lu_t=-(v\cdot\nabla v)_t -\nabla \phi_t-\frac{ah}{\sqrt{h^2+\epsilon^2} } h_t Lu+(\psi\cdot Q(v))_t.
\end{split}
\end{equation}
Second, multiplying (\ref{zhu7}) by $u_t$ and integrating over $\mathbb{R}^3$ lead to
\begin{equation}\label{zhu8}
\begin{split}
&\frac{1}{2} \frac{d}{dt}|u_t|^2_2+a\alpha|(h^2+\epsilon^2)^{\frac{1}{4}}\nabla u_t|^2_2+a(\alpha+\beta)|(h^2+\epsilon^2)^{\frac{1}{4}}\text{div} u_t|^2_2\\
=&\int  \Big(-(v\cdot \nabla v)_t -\nabla \phi_t -a\nabla \sqrt{h^2+\epsilon^2} \cdot Q(u)_t\\
&-\frac{ah}{\sqrt{h^2+\epsilon^2} }  h_t Lu+(\psi\cdot Q(v))_t \Big)\cdot u_t \\
\leq& C\big(|v|_{\infty}|\nabla v_t|_{2}+|v_t|_2 |\nabla v |_\infty +|\nabla \phi_t|_2+|\psi|_\infty |\sqrt{h}\nabla u_t|_2|\varphi|^{\frac{1}{2}}_\infty\big)|u_t|_2\\
& + C\big(|h_t|_\infty|\nabla^2 u|_2 +|\psi|_\infty|\nabla v_t|_2+|\psi_t|_2 |\nabla v|_\infty\big) |u_t|_2.
\end{split}
\end{equation}

Integrating (\ref{zhu8}) over $(\tau,t)$ $(\tau \in( 0,t))$ and using  Young's inequality,  one has
\begin{equation}\label{zhu13}
\begin{split}
&\frac{1}{2}|u_t(t)|^2_2+\frac{a\alpha}{2}\int_\tau^t |\sqrt{h}\nabla u_t(s)|^2_2 \text{d}s\\
\leq & \frac{1}{2}|u_t(\tau)|^2_2+Cc^{4}_4\int_0^t  |u_t(s)|^2_2 \text{d}s+Cc^{2}_4t+Cc^{4}_0,\quad  \text{for} \quad  0 \leq t \leq T'.
\end{split}
\end{equation}

It follows from the momentum equations $ (\ref{li4})_2$ that
\begin{equation}\label{zhu15}
\begin{split}
|u_t(\tau)|_2\leq C\big( |v|_\infty |\nabla v|_2+|\nabla \phi|_2+|(h+\epsilon)L u|_{2}+|\psi|_\infty|\nabla v|_2\big)(\tau),
\end{split}
\end{equation}
which, along with the assumption (\ref{vg}), Lemma \ref{lem1} and (\ref{houmian})-(\ref{g123definition}), implies that
\begin{equation}\label{zhu15vb}
\begin{split}
\lim \sup_{\tau\rightarrow 0}|u_t(\tau)|_2
\leq &C\big( |v_0|_\infty |\nabla v_0|_2+|\nabla \phi_0|_2+|g_2|_{2}+|Lu_0|_2+|\psi_0|_\infty|\nabla v_0|_2\big)\\
\leq & Cc^2_0.
\end{split}
\end{equation}

Letting $\tau\rightarrow 0$ in (\ref{zhu13}), one gets from the Gronwall's inequality that for $0 \leq t \leq T'$,
\begin{equation}\label{zhu14}
\begin{split}
&|u_t(t)|^2_2+\int_0^t|\sqrt{h}\nabla u_t|^2_2\text{d}s
\leq  C(c^{2}_4t+c^{4}_0)\exp\big(Cc^{4}_4 t\big)\leq Cc^{4}_0, \ \   \int_0^t|\nabla u_t|^2_2\text{d}s\leq Cc^5_0.
\end{split}
\end{equation}

It follows from (\ref{gai11}) that for $0 \leq t \leq T'$,
\begin{equation}\label{zhu14we}
\begin{split}
|\sqrt{h^2+\epsilon^2}u(t)|_{D^2}\leq & Cc^{\frac{7}{2}}_2c^{\frac{1}{2}}_3,\ \
|h\nabla^2u(t)|_2 \leq Cc^{\frac{7}{2}}_2c^{\frac{1}{2}}_3,\quad |u(t)|_{D^2} \leq Cc^{\frac{9}{2}}_2c^{\frac{1}{2}}_3.
\end{split}
\end{equation}
By the classical estimates for elliptic systems in Lemma \ref{zhenok} and (\ref{xintuo}), one gets
\begin{equation}\label{jiabiao}
\begin{split}
|\sqrt{h^2+\epsilon^2} u(t)|_{D^3}\leq &C \big(|u_t+v\cdot\nabla v+\nabla \phi-\psi\cdot Q(v)|_{D^1}+|G(\nabla \sqrt{h^2+\epsilon^2},u)|_{D^1}\big)\\
%\leq &C\big(| u_t|_{D^1}+\|v\|^2_2+|\phi|_{D^2}+\|\psi\|_{D^1\cap D^2}(\|v\|_2+\|u\|_2)\big)\\
\leq & C\big(|u_t|_{D^1}+c^{6}_3\big),\\
 |\sqrt{h^2+\epsilon^2}\nabla^3u(t)|_{2}
\leq & C\big(|\sqrt{h^2+\epsilon^2} u(t)|_{D^3} +\|u\|_2\|\psi\|_{D^1\cap D^2}\big)\\
&+C\|u\|_1(1+\|\psi\|^3_{D^1\cap D^2})(1+|\varphi|^2_{\infty})\\
\leq &  C\big(|\sqrt{h^2+\epsilon^2} u(t)|_{D^3}+c^{7}_3\big),
\end{split}
\end{equation}
which, along with (\ref{zhu14})-(\ref{zhu14we}), implies that,
\begin{equation*}
\displaystyle
\int_0^{t}\big(
|h\nabla^3u|^2_{2}
+
|h\nabla^2u|^2_{D^1}
+
|u|^2_{D^3}\big)\text{d}s \leq Cc^{7}_0,
\end{equation*}
for $ 0\leq t \leq T_3=\min(T',(1+c_4)^{-14})=\min(T_2,(1+c_4)^{-14})$.

\end{proof}

Next some estimates on the higher order derivatives of the velocity $u$ are established  in the following two  lemmas.
\begin{lemma}\label{5zx} Let $(\phi,u,h)$ be the unique classical solution to (\ref{li4}) in $[0,T] \times \mathbb{R}^3$. Then,
\begin{equation}\label{lipan1}
\begin{split}
(|\sqrt{h}\nabla u_t|^2_2+|u_t|^2_{D^1}+|u|^2_{D^3}+|h\nabla^2 u|^2_{D^1})(t)+\int_{0}^{t}|u_t|^2_{D^{2}}\text{d}s\leq& Cc^{16}_3, \\
\int_{0}^{t}\Big(|u_{tt}|^2_{2}+|u|^2_{D^{4}}+|h\nabla^2u|^2_{D^2}+|(h\nabla^2 u)_t|^2_{2}\Big)\text{d}s\leq& Cc^{11}_0,
\end{split}
\end{equation}
for $0\leq t \leq T_4=\min (T_3, (1+c_4)^{-20})$.
 \end{lemma}
\begin{proof}
Multiplying  (\ref{zhu7}) by $u_{tt}$ and integrating over $\mathbb{R}^3$ give
\begin{equation}\label{zhu19}
\begin{split}
& \frac{1}{2}\frac{d}{dt}\Big(a\alpha|(h^2+\epsilon^2)^{\frac{1}{4}}\nabla u_t|^2_2+a(\alpha+\beta)|(h^2+\epsilon^2)^{\frac{1}{4}}\text{div} u_t|^2_2\Big)+| u_{tt}|^2_2\\
=&\int \Big(-(v\cdot \nabla v)_t-\nabla \phi_t-\frac{a h}{\sqrt{h^2+\epsilon^2}} h_t Lu +\frac{a h}{\sqrt{h^2+\epsilon^2}}\nabla h\cdot Q(u_t) \Big)\cdot u_{tt} \\
&+\int \Big(\frac{a h}{\sqrt{h^2+\epsilon^2}}h_t\big(\alpha|\nabla u_t|^2+(\alpha+\beta)|\text{div}u_t|^2\big)+(\psi \cdot Q(v))_t\cdot u_{tt} \Big) \\
\leq& C\big(|v_t|_{2}|\nabla v|_{\infty}+|v|_\infty |\nabla v_t|_2+  |\nabla \phi_t|_2+|h_t|_\infty|\nabla^2 u|_2+|\psi_t|_2|\nabla v|_\infty\big)|u_{tt}|_2\\
&+ C\big(|\psi|_\infty |\nabla v_t|_2+|\psi|_\infty|\varphi|^{\frac{1}{2}}_\infty|\sqrt{h}\nabla u_t|_2\big) |u_{tt}|_2+ C|h_t|_\infty|\varphi|_\infty|\sqrt{h}\nabla u_t|^2_2.\\
%\leq &C\big(1+|\psi|^2_\infty|\varphi|_\infty|u_t|^2_2+|\psi|_\infty |v|_\infty|\varphi|_\infty+|\nabla v|_\infty\big)|\sqrt{h}\nabla u_t|^2_2+C|v_t|^2_{2}\|\nabla v\|^2_{2}.
%&+C\|v\|^2_2 |\nabla v_t|^2_2+C |\nabla \phi_t|^2_2
%+C|\psi_t|^2_2\|\nabla v\|^2_2 +C\|\psi\|^2_{D^1\cap D^2} |\nabla v_t|^2_2.
\end{split}
\end{equation}
Integrating (\ref{zhu19}) over $(\tau,t)$ shows that for $0\leq t \leq T_3$,
\begin{equation}\label{zhu22g}
\begin{split}
|\sqrt{h}\nabla u_t(t)|^2_2+\int_{\tau}^t| u_{tt}|^2_2\text{d}s\leq  C|(h^2+\epsilon^2)^{\frac{1}{4}}\nabla u_t(\tau)|^2_2+Cc^{14}_4t+Cc^{3}_4\int_0^t  |\sqrt{h}\nabla u_t|^2_2\text{d}s.
\end{split}
\end{equation}

On the other hand,  it follows from the momentum equations $ (\ref{li4})_2$ that
\begin{equation}\label{zhu15wsx}
\begin{split}
|\sqrt{h}\nabla u_t(\tau)|_2\leq& \big( |\sqrt{h}\nabla \big(v\cdot \nabla v+ \nabla \phi+ a \sqrt{h^2+\epsilon^2} Lu-\psi\cdot Q(v)\big)|_{2}\big)(\tau).
\end{split}
\end{equation}
Then by the assumption (\ref{vg}), Lemma \ref{lem1},  (\ref{houmian})-(\ref{g123definition}) and  (\ref{gai11qq})-(\ref{gai11qqcv}), one has
\begin{equation}\label{zhu15vbwsx}
\begin{split}
\lim \sup_{\tau\rightarrow 0}|\sqrt{h}\nabla u_t(\tau)|_2\leq &C( |\sqrt{h}_0\nabla (u_0\cdot \nabla u_0)|_2+|\sqrt{h}_0\nabla^2\phi_0|_2\\
&+|\sqrt{h}_0\nabla (\psi_0\cdot Q(u_0))|_2+| g_3|_{2}+\epsilon |\nabla Lu_0|_2)\\
\leq& C\big( |\phi^e_0u_0|_6|\nabla^2 u_0|_3+|\nabla u_0|_\infty |\phi^e_0\nabla u_0|_2|+|\phi^e_0\nabla^2\phi_0|_2\big)\\
+& C\big( |\phi^e_0\nabla^2u_0|_2|\psi_0|_\infty+|\nabla \psi_0|_3 |\phi^e_0\nabla u_0|_6+c_0\big)\leq Cc^3_0,
\end{split}
\end{equation}
which implies that
$$
\lim \sup_{\tau\rightarrow 0}|\sqrt{\epsilon}\nabla u_t(\tau)|_2\leq  \lim \sup_{\tau\rightarrow 0}\sqrt{\epsilon} |\varphi|^{\frac{1}{2}}_\infty|\sqrt{h}\nabla u_t(\tau)|_2\leq Cc^{\frac{7}{2}}_0.$$

Letting $\tau \rightarrow 0$ in (\ref{zhu22g}) and using the Gronwall's inequality, one can obtain
\begin{equation}\label{zhu22wsx}
\begin{split}
|\sqrt{h}\nabla u_t(t)|^2_2+\int_{0}^t| u_{tt}(s)|^2_2\text{d}s\leq  C(c^7_0+c^{14}_4t)\exp(Cc^{3}_4 t)\leq Cc^{7}_0, \quad  |\nabla u_t(t)|^2_2\leq Cc^{8}_0,
\end{split}
\end{equation}
for $0\leq t \leq T_3$, which, along with (\ref{jiabiao}), implies that
\begin{equation*}
\begin{split}
|\sqrt{h^2+\epsilon^2}u|_{D^3}
+
|\sqrt{h^2+\epsilon^2}\nabla^3u |_{2}
+
|h\nabla^2u|_{D^1}\leq Cc^{7}_3,\ \ | \nabla^3u|_{2}\leq& Cc^{8}_3.
\end{split}
\end{equation*}
It follows from (\ref{zhu7}) that
\begin{equation}\label{zhu77qq}
\begin{split}
a\sqrt{h^2+\epsilon^2}Lu_t=&-u_{tt}-(v\cdot\nabla v)_t -\nabla \phi_t-\frac{a h}{\sqrt{h^2+\epsilon^2}} h_t Lu+(\psi\cdot Q(v))_t\\
=&  a L(\sqrt{h^2+\epsilon^2} u_t)+G(\nabla \sqrt{h^2+\epsilon^2},u_t).
\end{split}
\end{equation}

Now  applying  Lemma \ref{zhenok} to \eqref{xintuo} and (\ref{zhu77qq}),  one gets for $0 \leq t \leq T_3$,
\begin{equation}\label{gaijia1}
\begin{split}
|\sqrt{h^2+\epsilon^2} u_t(t)|_{D^2}\leq& C\Big|u_{tt}+(v\cdot\nabla v)_t+\nabla \phi_t-(\psi\cdot Q(v))_t+\frac{a h}{\sqrt{h^2+\epsilon^2}} h_t Lu\Big|_2\\
%\leq& C\Big(|u_{tt}|_2+|v_t|_6|\nabla v|_3+|v|_\infty|\nabla v_t|_2+|\nabla \phi_t|_2+|\psi|_\infty|\nabla v_t|_2\Big)\\
%&+C\Big(|\psi_t|_6|\nabla v|_3+|\psi|_\infty|\nabla u_t|_2+|u_t|_6|\nabla \psi|_3\Big)
&+C|G(\nabla \sqrt{h^2+\epsilon^2},u_t)|_2\leq C(|u_{tt}(t)|_2+c^{7}_4),\\
|\sqrt{h^2+\epsilon^2} \nabla^2 u_t(t)|_2\leq& C(|\sqrt{h^2+\epsilon^2}  u_t(t)|_{D^2}+\| \psi\|_{D^1\cap D^2}| \nabla u_t|_2+|\psi|^2_\infty|  u_t|_2|\varphi|_\infty)\\
 \leq & C(|\sqrt{h^2+\epsilon^2}  u_t|_{D^2}+c^8_4),\\
%\end{split}
%\end{equation}
%and
%\begin{equation}\label{gaijia1xc}
%\begin{split}
|(h\nabla^2 u)_t(t)|_2
\leq & C\big(|h\nabla^2 u_t|_2+|h_t|_\infty |\nabla^2u|_2\big)\leq  C\big(|u_{tt}(t)|_2+c^8_4\big),\\
|u(t)|_{D^4}\leq &C \big|(h^2+\epsilon^2)^{-\frac{1}{2}}(u_t+ v\cdot\nabla v+ \nabla \phi-\psi\cdot Q(v))\big|_{D^2}\\
%\leq & C\|\varphi\|_{L^\infty\cap D^{1,6}\cap D^{2,3}\cap D^3}\big(|u_t|_{D^1\cap D^2}+\| v\|^2_3+\| \phi\|_3+\|\psi\|_{D^1\cap D^2}\|v\|_3\big)
\leq  &C \big(c^2_0|u_{tt}(t)|_2+c^{10}_4\big).
\end{split}
\end{equation}

Using the momentum equations $ (\ref{li4})_2$, for multi-index $\xi\in R^3$ with $|\xi|=2$, one has
\begin{equation}\label{jiegou1}
\begin{split}
a L(\sqrt{h^2+\epsilon^2} \nabla^\xi u)=& a \sqrt{h^2+\epsilon^2}  \nabla^\xi Lu-G(\nabla \sqrt{h^2+\epsilon^2} ,\nabla^\xi u)\\
=&-\sqrt{h^2+\epsilon^2}  \nabla^\xi \big( (h^2+\epsilon^2)^{-\frac{1}{2}}  (u_t+v\cdot\nabla v +\nabla \phi-\psi\cdot Q(v) ) \big)\\
&-G(\nabla \sqrt{h^2+\epsilon^2} ,\nabla^\xi u ),
%&a (2\alpha+\beta) \triangle (h\nabla \text{div}u)=\nabla \text{div}\big(u_t+v\cdot \nabla v+\nabla \phi-\psi \cdot Q(v)\big)\\
%&-a\alpha \nabla (\nabla h \cdot \triangle u)-a(\alpha+\beta)\nabla (\nabla h\cdot \nabla \text{div}u)-a (2\alpha+\beta) \nabla h  \triangle \text{div}u\\
%&+2a (2\alpha+\beta) (\nabla h \cdot \nabla) \nabla \text{div}u+a (2\alpha+\beta)\triangle  h\cdot  \nabla\text{div}u,
\end{split}
\end{equation}
which  implies that
\begin{equation}\label{gaijia1SD}
\begin{split}
&|\sqrt{h^2+\epsilon^2}  \nabla^2 u(t)|_{D^2}\\
\leq& C \big|\sqrt{h^2+\epsilon^2}  \nabla^\xi \big( (h^2+\epsilon^2)^{-\frac{1}{2}}  (u_t+v\cdot\nabla v +\nabla \phi-\psi\cdot Q(v) )) \big|_{2}\\
&+C\big(|\psi|_\infty |u|_{D^3}+|\nabla\psi|_{3}|\nabla^2 u|_6+|\nabla^2 u|_2|\psi|^2_\infty|\varphi|_\infty\big)\\
%\leq & C\big(|u_t|_{D^2}+\| v\|^2_3+|\nabla^3 \phi|_2+\|\psi\|_{D^1\cap D^2}(\|v\|_3+\|u\|_3)\big)\\
\leq &C \big(|u_t(t)|_{D^2}+c^{9}_4\big).
\end{split}
\end{equation}
Thus, for $T_4=\min (T_3, (1+c_4)^{-20})$, it holds that
\begin{equation*}\begin{split}
&\int_0^{T_4}\big( |h\nabla^2 u_t|^2_2+|u_t|^2_{D^2}+|u|^2_{D^4}+|h\nabla^2 u|^2_{D^2}+|(h\nabla^2 u)_t|^2_2\big)\text{d}t\leq Cc^{11}_0.
\end{split}
\end{equation*}
\end{proof}

\begin{lemma}\label{5} Let $(\phi,u,h)$ be the unique classical solution to (\ref{li4}) in $[0,T] \times \mathbb{R}^3$. Then,
\begin{equation}\label{final5}
\begin{split}
&t(|u_t(t)|^2_{ D^2}+|u_{tt}(t)|^2_2+|u(t)|^2_{D^4})+\int_{0}^{t}s(|u_{tt}|^2_{D^1}+|u_{t}|^2_{D^3})\text{d}s\leq Cc^8_4,
\end{split}
\end{equation}
for  $0\leq t \leq T_5=\min (T_4, (1+c_5)^{-14})$.

 \end{lemma}

\begin{proof}

Now applying $\partial_t$ to (\ref{zhu7}) yields
\begin{equation}\label{zhu25}
\begin{split}
u_{ttt}+&a\sqrt{h^2+\epsilon^2} Lu_{tt}
=-2v_t\cdot\nabla v_t -v_{tt}\cdot \nabla v-v\cdot \nabla v_{tt}-\nabla \phi_{tt}-\frac{a \epsilon^2 h^2_t}{(h^2+\epsilon^2)^{\frac{3}{2}}}Lu\\
&-\frac{a h}{\sqrt{h^2+\epsilon^2}}h_{tt}Lu-2\frac{a h}{\sqrt{h^2+\epsilon^2}}h_{t}Lu_t+2\psi_t\cdot Q(v_t)+\psi_{tt}\cdot Q(v)+ \psi\cdot Q(v_{tt}).
\end{split}
\end{equation}
Multiplying  (\ref{zhu25}) by $u_{tt}$ and integrating over $\mathbb{R}^3$ give
\begin{equation}\label{zhu27}
\begin{split}
&\frac{1}{2} \frac{d}{dt}|u_{tt}|^2_2+a\alpha|(h^2+\epsilon^2)^{\frac{1}{4}}\nabla u_{tt}|^2_2+a(\alpha+\beta)|(h^2+\epsilon^2)^{\frac{1}{4}}\text{div} u_{tt}|^2_2\\
=&\int \Big(-2v_t\cdot\nabla v_t -v_{tt}\cdot \nabla v-v\cdot \nabla v_{tt}-\nabla \phi_{tt}-\frac{a h}{\sqrt{h^2+\epsilon^2}}\nabla h \cdot Q(u)_{tt}\Big)\cdot u_{tt}\\
&-\int \Big(\frac{a \epsilon^2 h^2_t}{(h^2+\epsilon^2)^{\frac{3}{2}}}Lu+\frac{a h}{\sqrt{h^2+\epsilon^2}}h_{tt}Lu\Big)\cdot u_{tt}  \\
&+\int\Big(-2\frac{a h}{\sqrt{h^2+\epsilon^2}}h_{t}Lu_t+2\psi_t\cdot Q(v_t)+\psi_{tt}\cdot Q(v)+ \psi\cdot Q(v_{tt}) \Big)\cdot u_{tt}  \\
\leq& C\big(|\nabla v_t|_6|v_t|_3+ |\nabla v|_\infty | v_{tt}|_2+|v|_\infty| \nabla v_{tt}|_2+|h_{tt}|_6|\nabla^2 u|_3 \big)|u_{tt}|_2\\
&+  C\big(|\phi_{tt}|_2+|\psi|_\infty|u_{tt}|_2  \big)|\sqrt{h}\nabla u_{tt}|_2|\varphi|^{\frac{1}{2}}_\infty+C|h_t|^2_\infty|\varphi|_\infty|\nabla^2 u|_2|u_{tt}|_2\\
&+C|h_t|_\infty|\nabla^2 u_t|_2 |u_{tt}|_2+ C\big(|\psi_t|_3|\nabla v_t|_6+| \psi_{tt}|_2|\nabla v|_\infty+ |\psi|_\infty|\nabla v_{tt}|_2\big)| u_{tt}|_2.
\end{split}
\end{equation}

Multiplying both sides of (\ref{zhu27}) by $t$  and  integrating over $(\tau,t)$, one can get
\begin{equation}\label{zhu29}
\begin{split}
&t|u_{tt}(t)|^2_2+\frac{a\alpha}{2}\int_\tau^t s|\sqrt{h}\nabla u_{tt}|^2_2\text{d}s
\leq \tau |u_{tt}(\tau)|^2_2+Cc^4_4+ Cc^3_5\int_\tau^t s|u_{tt}|^2_2\text{d}s,
\end{split}
\end{equation}
where $0\leq t \leq T_5=\min (T_4, (1+c_5)^{-20})$.

It follows  from \eqref{zhu22wsx} and Lemma \ref{1} that
there exists a sequence $s_k$ such that
$$
s_k\rightarrow 0, \quad \text{and}\quad s_k |u_{tt}(s_k,x)|^2_2\rightarrow 0, \quad \text{as} \quad k\rightarrow+\infty.
$$
Taking $\tau=s_k$ and letting $ k\rightarrow+\infty$ in (\ref{zhu29}),
%\begin{equation}\label{zhu30}
%\begin{split}
%&t|u_{tt}(t)|^2_2+\int_0^ts|\sqrt{h}\nabla u_{tt}|^2_2\text{d}s\leq Cc^3_4c_5+ Cc^2_4c_5\int_\tau^t s|u_{tt}|^2_2\text{d}s
%\end{split}
%\end{equation}
due to the Gronwall's inequality, one has
\begin{equation}\label{zhu30qq}
\begin{split}
t|u_{tt}(t)|^2_2+\int_0^ts|\sqrt{h}\nabla u_{tt}|^2_2\text{d}s
\leq &Cc^4_4 \exp (c^3_5t)
\leq  Cc^4_4\\
 \int_0^ts|\nabla u_{tt}|^2_2\text{d}s\leq&  Cc^{5}_4.
\end{split}
\end{equation}

According to (\ref{gaijia1}), (\ref{gaijia1SD}) and (\ref{zhu30qq}), it holds that
\begin{equation}\label{gaijia1zx}
\begin{split}
t^{\frac{1}{2}}|\nabla^2 u_t(t)|_2 \leq Cc_0c^{2}_4,\quad
t^{\frac{1}{2}}|\nabla^4 u(t)|_{2}
\leq & Cc^2_0c^{2}_4.
\end{split}
\end{equation}

Then it follows from   Lemma \ref{zhenok} and  \eqref{zhu77qq} that  for $0 \leq t \leq T_5$,
\begin{equation}\label{dada}
\begin{split}
|\sqrt{h^2+\epsilon^2}u_t(t)|_{D^3}\leq& C\Big|u_{tt}+(v\cdot\nabla v)_t+\nabla \phi_t-(\psi\cdot Q(v))_t+\frac{ah}{\sqrt{h^2+\epsilon^2}} h_t Lu\Big|_{D^1}\\
&+C|G(\nabla \sqrt{h^2+\epsilon^2},u_t)|_{D^1}\\
%\leq& C\Big(|u_{tt}|_{D^1}+\|v\|_3\| v_t\|_1+|v|_\infty| v_t|_{D^2}+| \phi_t|_{D^2}+\|\psi\|_{D^1\cap D^2}| v_t|_{D^2}\Big)(t)\\
%&+C\Big(\|v\|_3\|\psi_t\|_1+|\psi|_\infty| u_t|_{D^2}+\|u_t\|_2|\nabla^2 \psi|_2\Big)(t)\\
\leq& C\big(|\nabla u_{tt}|_{2}+c^{10}_4+c_4(| v_t|_{D^2}+| u_t|_{D^2})\big),\\
|\sqrt{h^2+\epsilon^2}\nabla^3 u_t(t)|_{2}\leq& C\big(|\sqrt{h^2+\epsilon^2} u_t|_{D^3} +|u_t|_\infty|\nabla^2 \psi|_2+|\nabla u_t|_6|\nabla \psi|_3\\
&+|\nabla^2 u_t|_2|\psi|_\infty+|\nabla u_t|_2\|\psi\|^2_{D^1\cap D^2}|\varphi|_\infty+|u_t|_2|\psi|^3_\infty|\varphi|^2_2 \big),
\end{split}
\end{equation}
which, along with (\ref{lipan1}) and  (\ref{zhu30qq})-(\ref{dada}), implies that
\begin{equation}\label{chun1}
\begin{split}
&\int_0^{T_5} t\big(|\sqrt{h^2+\epsilon^2}u_t|^2_{D^3}+|h\nabla^3u_t|^2_2+|\nabla^3u_t|^2_2\big)\text{d}t
\leq  Cc^7_4.
\end{split}
\end{equation}

\end{proof}

Then by Lemmas \ref{2}-\ref{5}, for $0 \leq t \leq T_5=\min (T^*, (1+Cc_5)^{-20})$, one has
\begin{equation*}
\begin{split}
\big(\|\phi-\phi^\infty\|^2_3+\|\phi_t\|^2_2+  |\phi_{tt}|^2_2\big)(t)+  \int_0^t \|\phi_{tt}\|^2_1 \text{d}s\leq & Cc^6_4,\\
  \|\psi(t)\|^2_{D^1\cap D^2}\leq Cc^2_0,\quad  |\psi_t(t)|_2\leq& Cc^2_3,\\
  |h_t(t)|^2_\infty\leq   Cc^3_3c_4,\ \
h(t,x)>\frac{1}{2c_0},\ \  |\psi_t(t)|^2_{D^1}+\int_0^t\big( |\psi_{tt}|^2_{2}+|h_{tt}|^2_{6}\big)\text{d}s\leq &  Cc^4_4, \\
\frac{2}{3}\eta^{-2e}< \varphi, \  \big(\|\varphi\|^2_{L^\infty\cap D^{1,6}\cap D^{2,3}\cap D^3}+\|f\|^2_{L^\infty\cap L^6 \cap D^{1,3}\cap D^2} \big)(t)\leq & Cc^4_0,\\
\big(\|\varphi_t\|^2_{L^6\cap D^{1,3}\cap D^2}+\|f_t\|^2_{L^3\cap D^{1}}\big)(t)\leq  &Cc^{10}_4,\\
|\sqrt{h}\nabla u|^2_2+\| u(t)\|^2_{1}+\int_{0}^{t} \Big( \|\nabla u\|^2_{1}+|u_t|^2_{2}\Big)\text{d}s \leq& Cc^{4}_0,\\
\big(|u|^2_{D^2}+|h\nabla^2 u|^2_2+|u_t|^2_{2}\big)(t)+\int_{0}^{t} \Big( |u|^2_{D^3}+|h\nabla^2 u|^2_{D^1}+|u_t|^2_{D^1}\Big)\text{d}s \leq& Cc^{9}_2c_3,\\
%\end{split}
%\end{equation*}
%and
%\begin{equation*}
%\begin{split}
\big(|u_t|^2_{D^1}+|\sqrt{h}\nabla u_t|^2_2+|u|^2_{D^3}+|h\nabla^2 u|^2_{D^1}\big)(t)+\int_{0}^{t}|u_t|^2_{D^{2}}\text{d}s\leq& Cc^{16}_3, \\
\int_{0}^{t}\Big(|u_{tt}|^2_{2}+|u|^2_{D^{4}}+|h\nabla u|^2_{D^1}+|h\nabla^2 u|^2_{D^2}+|(h\nabla^2 u)_t|^2_{2}\Big)\text{d}s\leq& Cc^{11}_0, \\
t\big(|u_t|^2_{ D^2}+|u_{tt}|^2_2+|u|^2_{D^4}\big)(t)+\int_{0}^{t}(s|u_{tt}|^2_{D^1}+s|u_{t}|^2_{D^3})\text{d}s\leq& Cc^8_4.
\end{split}
\end{equation*}

Therefore, defining  the time
$$T^*=\min (T, (1+C^{\frac{357}{2}}c^{576}_0)^{-20})$$ and constants
\begin{equation}\label{dingyi45}
\begin{split}
&c_1=C^{\frac{1}{2}}c_0, \quad  c_2=C^{\frac{1}{2}}c^2_0,\quad c_3=C^{\frac{11}{2}}c^{18}_0,\quad   c_4= C^{\frac{89}{2}}c^{144}_0, \ \ c_5=C^{\frac{357}{2}}c^{576}_0,
\end{split}
\end{equation}
one can obtain
\begin{equation}\label{jkk}
\begin{split}
\big(\|\phi-\phi^\infty\|^2_3+\|\phi_t\|^2_2+  |\phi_{tt}|^2_2\big)(t)+  \int_0^t \|\phi_{tt}\|^2_1 \text{d}s\leq & c^2_5,\\
\|\psi(t)\|^2_{D^1\cap D^2}\leq c^2_1,\quad |h_t|^2_\infty+ |\psi_t(t)|_2\leq & c^2_4,\\
 |\psi_t(t)|^2_{D^1}+\int_0^t\big( |\psi_{tt}|^2_{2}+|h_{tt}|^2_{2}\big)\text{d}s\leq & c^2_5,\\
h>\frac{1}{2c_0},\ \  \frac{2}{3}\eta^{-2e}<\varphi,\ \ \big(\|\varphi\|^2_{L^\infty\cap D^{1,6}\cap D^{2,3}\cap D^3}+\|f\|^2_{L^\infty\cap L^6 \cap D^{1,3}\cap D^2} \big)(t)\leq & c^2_2,\\
\big(\|\varphi_t\|^2_{L^6\cap D^{1,3}\cap D^2}+\|f_t\|^2_{L^3\cap D^{1}}\big)(t)\leq & c^{3}_5,\\
|\sqrt{h}\nabla u|^2_2+\| u(t)\|^2_{1}+\int_{0}^{t} \Big( \|\nabla u\|^2_{1}+|u_t|^2_{2}\Big)\text{d}s \leq& c^{2}_2,\\
\big(|u|^2_{D^2}+|h\nabla^2 u|^2_2+|u_t|^2_{2}\big)(t)+\int_{0}^{t} \Big( |u|^2_{D^3}+|h\nabla^2 u|^2_{D^1}+|u_t|^2_{D^1}\Big)\text{d}s \leq& c^2_3,\\
%\end{split}
%\end{equation}
%and
%\begin{equation}\label{jkuu}
%\begin{split}
\big(|\sqrt{h}\nabla u_t|^2_2+|u_t|^2_{D^1}+|u|^2_{D^3}+|h\nabla u|^2_{D^1}+|h\nabla^2 u|^2_{D^1}\big)(t)+\int_{0}^{t}|u_t|^2_{D^{2}}\text{d}s\leq& c^{2}_4, \\
\int_{0}^{t}\Big(|u_{tt}|^2_{2}+|u|^2_{D^{4}}+|h\nabla^2 u|^2_{D^2}+|(h\nabla^2 u)_t|^2_{2}\Big)\text{d}s\leq& c^{2}_4, \\
t\big(|u_t|^2_{ D^2}+|u_{tt}|^2_2+|u|^2_{D^4}\big)(t)+\int_{0}^{t}(s|u_{tt}|^2_{D^1}+s|u_{t}|^2_{D^3})\text{d}s\leq& c^2_5,
\end{split}
\end{equation}
for $0\leq t \leq T^*$.
In another word, given fixed $c_0$ and $T$, there are positive constants $T^*$, $c_i$ $(i=1, 2, 3, 4,5)$, depending only on $c_0$ and $T$, such that if \eqref{jizhu1} holds for $(g,v)$, then \eqref{jkk} holds for the classical solution to
\eqref{li4} on $[0, T^*]\times \mathbb{R}^3$.

\subsection{Passing to the  limit $\epsilon\rightarrow 0$} With the help of the $(\epsilon,\eta)$-independent estimates established in (\ref{jkk}), we now establish the local existence result for the following linearized problem
without artificial viscosity (i.e., $\epsilon=0$) under the assumption $\phi_0\geq \eta$,
\begin{equation}\label{li4*}
\begin{cases}
\displaystyle
\phi_t+v\cdot \nabla \phi+(\gamma-1)\phi \text{div} v=0,\\[6pt]
\displaystyle
%f_t+\sum_{l=1}^3 A_l(v) \partial_l\psi+B^*(v)\psi+\nabla \text{div} v=0,\\[10pt]
%\displaystyle
u_t+v\cdot\nabla v +\nabla \phi+a hLu=\psi \cdot Q(v),\\[10pt]
\displaystyle
h_t+v\cdot \nabla h+(\delta-1)g\text{div} v=0,\\[6pt]
\displaystyle
(\phi,u,h)|_{t=0}=(\phi_0,u_0,h_0)=\big(\phi_0,u_0,(\phi_0)^{2e}\big),\quad x\in \mathbb{R}^3,\\[10pt]
\displaystyle
(\phi,u,h)\rightarrow (\phi^\infty,0, h^\infty=(\phi^\infty)^{2e}),\quad \text{as}\quad  |x|\rightarrow +\infty,\quad t>0.
 \end{cases}
\end{equation}

\begin{lemma}\label{lem1q}
 Let (\ref{canshu}) hold.
Assume that  the initial data $(\phi_{0}, u_0,h_0=(\phi_0)^{2e})$  satisfies the hypothesis of Lemma \ref{lem1}, and there exists a positive constant $c_0$ independent of $\eta$ such that (\ref{houmian}) holds.  Then there exist a time $T^*>0$ independent of $\eta$, and a unique classical solution
$$\Big(\phi, u, h, \psi=\frac{a\delta}{\delta-1}\nabla h\Big)$$
 in $[0,T^*]\times \mathbb{R}^3$ to  (\ref{li4*}) satisfying  (\ref{reggh}) with $T$ replaced by $T^*$. Moreover,  $(\phi,u,h)$  satisfies the   estimates in (\ref{jkk}) independent of $\eta$.
\end{lemma}
\begin{proof} We shall prove the existence, uniqueness and time continuity in two steps.

\noindent\textbf{Step 1:} Existence. First, it follows from Lemmas \ref{lem1}--\ref{5} that  for every $\epsilon>0$ and $\eta>0$,   there exist a time $T^*>0$ independent of $(\epsilon,\eta)$, and   a unique strong solution $(\phi^{\epsilon,\eta}, u^{\epsilon,\eta}, h^{\epsilon,\eta})(t,x)$ in $[0,T^*]\times \mathbb{R}^3$  to the linearized problem (\ref{li4}) satisfying  the estimates in  (\ref{jkk}), which are independent of  $(\epsilon,\eta)$.

Second, using the characteristic method and the  standard energy estimates for transport equations,  and  $(\ref{li4})_3$, one gets easily that
\begin{equation}\label{related}
\|h^{\epsilon,\eta}(t)\|_{L^\infty\cap D^1}+ |h^{\epsilon,\eta}_t(t)|_{2}\leq C(\eta, \alpha, \beta, \gamma, \delta, T, \phi_0,u_0),\quad \text{for} \quad 0\leq t\leq T^*.
\end{equation}

Then,  by virtue or  the uniform estimates in (\ref{jkk}) independent of  $(\epsilon,\eta)$, estimates in (\ref{related}) independent of $\epsilon$,  and  the compactness in Lemma \ref{aubin} (see \cite{jm}), one gets that for any $R> 0$,  there exists a subsequence of solutions (still denoted by) $(\phi^{\epsilon,\eta}, u^{\epsilon,\eta}, h^{\epsilon,\eta})$, which  converges  to  a limit $(\phi^\eta,  u^\eta, h^\eta) $ in the following  strong sense:
\begin{equation}\label{ert1}\begin{split}
&(\phi^{\epsilon,\eta}, u^{\epsilon,\eta}, h^{\epsilon,\eta}) \rightarrow (\phi^\eta,  u^\eta, h^\eta) \quad \text{ in } \ C([0,T^*];H^2(B_R)),\quad \text{as}\quad \epsilon \rightarrow 0.
\end{split}
\end{equation}

Again, due to the uniform estimates in (\ref{jkk}) independent of  $\eta$ and the estimates in (\ref{related}) independent of $\epsilon$,  there exists a subsequence (of subsequence chosen above) of solutions (still denoted by) $(\phi^{\epsilon,\eta}, u^{\epsilon,\eta}, h^{\epsilon,\eta})$, which    converges to  $(\phi^\eta,  u^\eta, h^\eta) $ as $\epsilon \rightarrow 0$ in the following  weak  or  weak* sense:
\begin{equation}\label{ruojixian}
\begin{split}
(\phi^{\epsilon,\eta}-\phi^\infty, u^{\epsilon,\eta})\rightharpoonup  (\phi^\eta-\phi^\infty, u^\eta) \quad &\text{weakly* \ in } \ L^\infty([0,T^*];H^3),\\
(\phi^{\epsilon,\eta}_t,\psi^{\epsilon,\eta},h^{\epsilon,\eta}_t)\rightharpoonup ( \phi^\eta_t,\psi^\eta,h^{\eta}_t) \quad &\text{weakly* \ in } \ L^\infty([0,T^*];H^2),\\
u^{\epsilon,\eta}_t \rightharpoonup   u^\eta_t \quad &\text{weakly* \ in } \ L^\infty([0,T^*];H^1),\\
(\phi^{\epsilon,\eta}_{tt},\nabla^3 \varphi^{\epsilon,\eta},\nabla^2f^{\epsilon,\eta}) \rightharpoonup  (\phi^{\eta}_{tt},\nabla^3 \varphi^{\eta},\nabla^2f^{\eta}) \quad &\text{weakly* \ in } \ L^\infty([0,T^*];L^2),\\
(\nabla^2 \varphi^{\epsilon,\eta}_t,\nabla f^{\epsilon,\eta}_t) \rightharpoonup  (\nabla^2 \varphi^{\eta}_t,\nabla f^{\eta}_t) \quad &\text{weakly* \ in } \ L^\infty([0,T^*];L^2),\\
t^{\frac{1}{2}}(\nabla^2u^{\epsilon,\eta},u^{\epsilon,\eta}_{tt},\nabla^4u^{\epsilon,\eta}) \rightharpoonup  t^{\frac{1}{2}}(\nabla^2u^{\eta}, u^{\eta}_{tt},\nabla^4u^{\eta}) \quad &\text{weakly* \ in } \ L^\infty([0,T^*];L^2),\\
(\nabla^2\varphi ^{\epsilon,\eta},\nabla \varphi^{\epsilon,\eta}_{t},\nabla f^{\epsilon,\eta},f^{\epsilon,\eta}_t) \rightharpoonup  (\nabla^2\varphi ^{\eta},\nabla \varphi^{\eta}_{t},\nabla f^{\eta},f^\eta_t) \quad &\text{weakly* \ in } \ L^\infty([0,T^*];L^3),\\
(\nabla \varphi^{\epsilon,\eta}, \varphi^{\epsilon,\eta}_t,f^{\epsilon,\eta}) \rightharpoonup  (\nabla \varphi^{\eta}, \varphi^{\eta}_t,f^{\eta}) \quad &\text{weakly* \ in } \ L^\infty([0,T^*];L^6),\\
(h^{\epsilon,\eta}, h^{\epsilon,\eta}_t,\varphi^{\epsilon,\eta},f^{\epsilon,\eta}) \rightharpoonup  (h^{\eta}, h^{\eta}_t,\varphi^\eta,f^\eta) \quad &\text{weakly* \ in } \ L^\infty([0,T^*];L^\infty),\\
\nabla u^{\epsilon,\eta} \rightharpoonup  \nabla u^{\eta} \quad &\text{weakly \ \  in } \ \ L^2([0,T^*];H^3),\\
u^{\epsilon,\eta}_{t} \rightharpoonup  u^{\eta}_{t}\quad &\text{weakly \ \  in } \ \ L^2([0,T^*];H^2),\\
\phi^{\epsilon,\eta}_{tt}  \rightharpoonup  \phi^{\eta}_{tt} \quad &\text{weakly \ \  in } \ \ L^2([0,T^*];H^1),\\
(\psi^{\epsilon,\eta}_{tt},h^{\epsilon,\eta}_{tt},u^{\epsilon,\eta}_{tt}) \rightharpoonup  (\psi^{\eta}_{tt},h^{\eta}_{tt},u^{\eta}_{tt})\quad &\text{weakly \ \  in } \ \ L^2([0,T^*];L^2),\\
t^{\frac{1}{2}}(\nabla u^{\epsilon,\eta}_{tt},\nabla^3u^{\epsilon,\eta}_t) \rightharpoonup  t^{\frac{1}{2}}(\nabla u^{\eta}_{tt},\nabla^3u^{\eta})_t) \quad &\text{weakly \ \  in } \ \ L^2([0,T^*];L^2),
\end{split}
\end{equation}
which, along with  the lower semi-continuity of weak or weak* convergence, implies  that $(\phi^\eta,  u^\eta, h^\eta) $  satisfies also the corresponding  estimates in (\ref{jkk}) and (\ref{related}) except those weighted estimates on $u^\eta$.

Collecting the uniform estimates on $(\phi^\eta,  u^\eta, h^\eta) $ obtained above, together with the   strong convergence in (\ref{ert1}) and  the weak or weak* convergence in  (\ref{ruojixian}), one obtains that
\begin{equation}\label{ruojixian2}
\begin{split}
\sqrt{h^{\epsilon,\eta}}(\nabla u^{\epsilon,\eta},  \nabla u^{\epsilon,\eta}_t) \rightharpoonup  \sqrt{h^{\eta}}(\nabla u^{\eta},  \nabla u^{\eta}_t) \quad &\text{weakly* \ in } \ L^\infty([0,T^*];L^2),\\
h^{\epsilon,\eta}\nabla^2 u^{\epsilon,\eta} \rightharpoonup  h^{\eta}\nabla^2 u^{\eta} \quad &\text{weakly* \ in } \ L^\infty([0,T^*];H^1),\\
 (h^{\epsilon,\eta}\nabla^2 u^{\epsilon,\eta})_t \rightharpoonup  (h^{\eta}\nabla^2 u^{\eta})_t \quad &\text{weakly \ \  in } \ \ L^2([0,T^*];L^2),\\
h^{\epsilon,\eta}\nabla^2 u^{\epsilon,\eta} \rightharpoonup  h^{\eta}\nabla^2 u^{\eta}\quad &\text{weakly \ \  in } \ \ L^2([0,T^*]; D^1\cap D^2),\\
\end{split}
\end{equation}
which, along with the lower semi-continuity of weak or weak* convergence again, implies that $(\phi^\eta,  u^\eta, h^\eta) $  satisfies also the uniform weighted  estimates on  $ u^{\eta}$.

Now we  are going to show that $(\phi^\eta,  u^\eta, h^\eta) $ is a weak solution in the sense of distribution  to  (\ref{li4*}).
First, multiplying $(\ref{li4*})_2$ by  test function  $w(t,x)=(w^1,w^2,w^3)\in C^\infty_c ([0,T^*)\times \mathbb{R}^3)$ on both sides, and integrating over $[0,t)\times \mathbb{R}^3$ for $t\in (0,T^*]$, one has
\begin{equation}\label{zhenzheng1}
\begin{split}
&\int_0^t \int  \Big(u^{\epsilon,\eta}\cdot w_t - (v\cdot \nabla) v \cdot w +\phi^{\epsilon,\eta}\text{div}w \Big)\text{d}x\text{d}s\\
=&-\int u_0 \cdot w(0,x)+\int_0^t \int \Big(\sqrt{(h^{\epsilon,\eta})^2+\epsilon^2} Lu^{\epsilon,\eta}\cdot w-\psi^{\epsilon,\eta} \cdot Q(v) \cdot w\Big) \text{d}x\text{d}s.
\end{split}
\end{equation}
It follows from  the  uniform estimates obtained above, the strong convergence in (\ref{ert1}),  and  the weak convergences in (\ref{ruojixian})-(\ref{ruojixian2}) that  and  letting $\epsilon \rightarrow 0$ in (\ref{zhenzheng1}) yields
\begin{equation}\label{zhenzhengxx}
\begin{split}
&\int_0^t \int  \Big(u^{\eta}\cdot w_t - (v\cdot \nabla) v \cdot w +\phi^{\eta}\text{div}w \Big)\text{d}x\text{d}s\\
=&-\int u_0 \cdot w(0,x)+\int_0^t \int \Big(h^\eta Lu^{\eta}\cdot w-\psi^{\eta} \cdot Q(v) \cdot w\Big) \text{d}x\text{d}s.
\end{split}
\end{equation}
Second, one can use the similar argument to show that $(\phi^\eta,h^\eta)$ satisfies also the equations in  $(\ref{li4*})_1$ and $(\ref{li4*})_3$ and the initial data in the sense of distribution.
So it is clear that $(\phi^\eta,  u^\eta, h^\eta) $ is a weak solution in the sense of distribution  to the linearized problem (\ref{li4*}), satisfying the  following regularities
\begin{equation}\label{zheng}
\begin{split}
&\phi^\eta-\phi^\infty\in L^\infty([0,T^*]; H^3), \quad h^\eta\in L^\infty([0,T^*]\times \mathbb{R}^3),\  \ \nabla h^\eta \in L^\infty([0,T^*]; H^2),\\
& h^\eta_t\in L^\infty([0,T^*];H^2), \  \ \ u^\eta\in L^\infty([0,T]; H^3)\cap L^2([0,T^*]; H^4), \\
 &  u^\eta_t \in L^\infty([0,T^*]; H^1)\cap L^2([0,T^*]; D^2),\ \   u^\eta_{tt}\in L^2([0,T^*];L^2),\\
 &  t^{\frac{1}{2}}u^\eta\in L^\infty([0,T^*];D^4),\quad   t^{\frac{1}{2}}u^\eta_t\in L^\infty([0,T^*];D^2)\cap L^2([0,T^*]; D^3),\\
&   t^{\frac{1}{2}}u^\eta_{tt}\in L^\infty([0,T^*];L^2)\cap L^2([0,T^*];D^1).
\end{split}
\end{equation}
Therefore, this weak solution $(\phi^\eta,  u^\eta,h^\eta) $ of \eqref{li4*} is actually
a strong one.

\noindent\textbf{Step 2:} Uniqueness and Time continuity.  Due to the estimate $h^\eta>\frac{1}{2c_0}$,  the uniqueness and the time continuity of the strong solution obtained above can be proved via the completely same arguments as in Lemma \ref{lem1}, so details are omitted.

\end{proof}

\subsection{Construction of the nonlinear approximation solutions away from vacuum} In this subsection, based on the assumption that $\phi_0>\eta$, we will prove the local-in-time well-posedness of the  classical solution to the following Cauchy problem
 \begin{equation}
\begin{cases}
\label{eq:cccq-xxx}
\displaystyle
\phi_t+u\cdot \nabla \phi+(\gamma-1)\phi \text{div} u=0,\\[12pt]
\displaystyle
u_t+u\cdot\nabla u +\nabla \phi+a\phi^{2e}Lu=\psi \cdot Q(u),\\[12pt]
\displaystyle
\psi_t+\nabla (u\cdot \psi)+(\delta-1)\psi\text{div} u +\delta a\phi^{2e}\nabla \text{div} u=0,\\[12pt]
\displaystyle
(\phi,u,\psi)|_{t=0}=(\phi_0,u_0,\psi_0)=\big(\phi_0,u_0,\frac{a\delta}{\delta-1}\nabla (\phi_0)^{2e}\big),\quad x\in \mathbb{R}^3,\\[10pt]
\displaystyle
(\phi,u,\psi)\rightarrow (\phi^\infty,0, 0),\quad \text{as}\quad  |x|\rightarrow +\infty,\quad t\geq0,
 \end{cases}
\end{equation}
whose life span is independent of $\eta$.

\begin{theorem}\label{th1zx} Let (\ref{canshu}) hold and $\phi^\infty$ be a postive constant.
 Assume that  the initial data $(\phi_{0},u_0,h_0=(\phi_0)^{2e},\psi_0=\frac{a\delta}{\delta-1}\nabla (\phi_0)^{2e})$  satisfies the hypothesis of Lemma \ref{lem1}, and there exists a positive constant $c_0$ independent of $\eta$ such that (\ref{houmian}) holds. Then there exist a time $T_*>0$ and a unique classical solution
 $$\Big(\phi, u,h=\phi^{2e}, \psi=\frac{a\delta}{\delta-1}\nabla h\Big)$$ in $[0,T_*]\times \mathbb{R}^3$ to the Cauchy problem (\ref{eq:cccq-xxx}) satisfying (\ref{reggh}) and
\begin{equation*}\begin{split}
&\phi^{-2e} \in L^\infty([0,T_*];L^\infty\cap D^{1,6} \cap D^{2,3} \cap D^3), \ \   \nabla \phi/ \phi \in L^\infty([0,T_*];L^\infty\cap L^6\cap  D^{1,3} \cap D^2),
\end{split}
\end{equation*}
where $T_*$ is independent of $\eta$.
Moreover, if the initial data satisfies  (\ref{houmian}),
then  estimates (\ref{jkk})  hold for  $( \phi,u,h)$ with $T^*$ replaced by $T_*$, and are independent of $\eta$.
\end{theorem}

The proof  is given by an  iteration scheme based on the estimates for the linearized problem    obtained  in  Sections 3.2-3.4. As in Section 3.3, we define constants $c_{i}$ ($i=1,...,5$).

Let $(\phi^0, u^0,h^0)$ be the solution to the  following Cauchy problem
\begin{equation}\label{zheng6}
\begin{cases}
 X_t+u_0 \cdot \nabla X=0 \quad \text{in} \quad (0,+\infty)\times \mathbb{R}^3,\\[12pt]
Y_t- Z\triangle Y=0 \quad \  \ \  \ \text{in} \quad  (0,+\infty)\times \mathbb{R}^3 ,\\[12pt]
Z_t+u_0 \cdot \nabla Z=0\quad \  \text{in} \quad (0,+\infty)\times \mathbb{R}^3,\\[12pt]
(X,Y,Z)|_{t=0}=(\phi_0,u_0,h_0)=(\phi_0,u_0,\phi^{2e}_0) \quad \text{in} \quad \mathbb{R}^3,\\[10pt]
(X,Y, Z)\rightarrow (\phi^\infty, 0, h^\infty=(\phi^\infty)^{2e}) \quad \text{as } \quad |x|\rightarrow +\infty,\quad t> 0.
\end{cases}
\end{equation}
Choose a time $T^{**}\in (0,T^*]$ small enough such that
\begin{equation}\label{jizhu}
\begin{split}
\sup_{0\leq t \leq T^{**}}\|\nabla h^0(t)\|^2_{D^1\cap D^2}\leq c^2_1,\ \ \sup_{0\leq t \leq T^{**}}\| u^0(t)\|^2_{1}+\int_{0}^{T^{**}} \Big( |u^0|^2_{D^2}+|u^0_t|^2_{2}\Big)\text{d}t \leq& c^2_2,\\[9pt]
\sup_{0\leq t \leq T^{**}}\big(|u^0|^2_{D^2}+|u^0_t|^2_{2}+|h^0\nabla^2u^0|^2_{2}\big)(t)+\int_{0}^{T^{**}} \Big( |u^0|^2_{D^3}+|u^0_t|^2_{D^1}\Big)\text{d}t \leq& c^2_3,\\[9pt]
\sup_{0\leq t \leq T^{**}}\big(|u^0|^2_{D^3}+|u^0_t|^2_{D^1}+| h^0_t|^2_{D^1}\big)(t)+\int_{0}^{T^{**}} \Big( | u^0|^2_{D^4}+|u^0_t|^2_{D^2}+|u^0_{tt}|^2_2\Big)\text{d}t \leq& c^2_4,\\[9pt]
\sup_{0\leq t \leq T^{**}}\big(|h^0\nabla^2u^0|^2_{D^1}+|h^0_t|^2_\infty\big)(t)+\int_{0}^{T^{**}} \Big(|(h^0\nabla^2 u^0)_t|^2_{2}+|h^0\nabla^2 u^0|^2_{D^2}\Big)\text{d}t \leq& c^2_4,\\[9pt]
\text{ess}\sup_{0\leq t \leq T^{**}}t\big(|u^0_t|^2_{D^2}+|u^0|^2_{D^4}+|u^0_{tt}|^2_{2}\big)(t)+\int_{0}^{T^{**}} t\big(|u^0_{tt}|^2_{D^1}+|u^0_{t}|^2_{D^3}\big)\text{d}t \leq& c^2_5.
\end{split}
\end{equation}

\begin{proof}The existence, uniqueness and time continuity can be proved as follows.

\textbf{Step 1:} Existence.
Let the beginning step of our iteration be
$(v,g)=(u^0,h^0)$.  Thus one can get a classical solution $(\phi^1, u^1, h^1)$ of problem (\ref{li4*}). Inductively, one  constructs approximate sequences $(\phi^{k+1}, u^{k+1}, h^{k+1})$ as follows: given $(u^{k}, h^k)$ for $k\geq 1$, define $(\phi^{k+1},u^{k+1},h^{k+1})$  by solving the following problem:
\begin{equation}\label{li6}
\begin{cases}
\displaystyle
\phi^{k+1}_t+u^{k}\cdot \nabla \phi^{k+1}+(\gamma-1)\phi^{k+1}\text{div} u^{k}=0,\\[12pt]
\displaystyle
u^{k+1}_t+u^{k}\cdot\nabla u^{k} +\nabla \phi^{k+1}+a h^{k+1}Lu^{k+1}=\psi^{k+1}\cdot Q(u^{k}),\\[12pt]
\displaystyle
h^{k+1}_t+u^{k}\cdot \nabla h^{k+1}+(\delta-1)h^k\text{div} u^k=0,\\[6pt]
\displaystyle
(\phi^{k+1}, u^{k+1},  h^{k+1})|_{t=0}=(\phi_0, u_0, \phi^{2e}_0),\\[12pt]
(\phi^{k+1}, u^{k+1}, h^{k+1})\rightarrow (\phi^\infty, 0, h^\infty=(\phi^\infty)^{2e}),\quad \text{as}\quad  |x|\rightarrow +\infty,\quad t\geq0.
 \end{cases}
\end{equation}
This problem can be solved from \eqref{li4*} by replacing $(v,g)$ with $ (u^{k},h^k)$, and  $(\phi^{k},u^{k}, h^{k})$ $(k=1,2,...)$  satisfy the uniform estimates (\ref{jkk}).

Denote
$$\psi^{k+1}=\frac{a\delta}{\delta-1}\nabla h^{k+1},\quad  f^{k+1}=(h^{k+1})^{-1}\psi^{k+1}=\frac{a\delta}{\delta-1}\nabla h^{k+1}/h^{k+1},\quad \varphi^{k+1}=(h^{k+1})^{-1}.$$
Then
problem (\ref{li6}) can be rewritten as
\begin{equation}\label{li6zx}
\begin{cases}
\displaystyle
\quad \phi^{k+1}_t+u^{k}\cdot \nabla \phi^{k+1}+(\gamma-1)\phi^{k+1}\text{div} u^{k}=0,\\[6pt]
\displaystyle
\quad  \varphi^{k+1}\big(u^{k+1}_t+u^{k}\cdot\nabla u^{k} +\nabla \phi^{k+1}\big)=-a Lu^{k+1}+f^{k+1}\cdot Q(u^{k}),\\[6pt]
\displaystyle
%\quad  \psi^{k+1}_t+\sum_{l=1}^3 A_l(u^k) \partial_l\psi^{k+1}+B^*(u^k)\psi^{k+1}
%+a \delta h^k\nabla \text{div} u^k+(\delta-1)\psi^k\text{div}u^k=0,\\[6pt]
%\displaystyle
\quad  f^{k+1}_t+\sum_{l=1}^3 A_l(u^k) \partial_lf^{k+1}+B^*(u^k)f^{k+1}+a\delta(\varphi^{k})^{-1} \varphi^{k+1}  \nabla \text{div}u^k\\[12pt]
=-a\delta \varphi^{k+1}\nabla h^k\text{div}u^k+(\delta-1)(\varphi^{k})^{-1}\varphi^{k+1}f^{k+1}\text{div}u^k,\\[8pt]
\displaystyle
\quad  \varphi^{k+1}_t+u^k\cdot \nabla \varphi^{k+1}-(\delta-1)(\varphi^{k})^{-1} (\varphi^{k+1})^2\text{div} u^k=0,\\[12pt]
\displaystyle
\quad  (\phi^{k+1},  u^{k+1},  f^{k+1}, \varphi^{k+1})|_{t=0}=\Big(\phi_0,  u_0, \phi^{-2e}_0\psi_0,\phi^{-2e}_0\Big),\\[12pt]
\quad  (\phi^{k+1}, u^{k+1}, f^{k+1}, \varphi^{k+1})\rightarrow (\phi^\infty, 0, 0,\big(\phi^\infty\big)^{-2e}),\ \text{as}\   |x|\rightarrow +\infty,\  t\geq 0.
 \end{cases}
\end{equation}

\textbf{Step 1.1:} Strong convergence of $(\phi^k, u^k,f^k,\varphi^k)$.
Now we are going to prove that the whole sequence $(\phi^k, u^k, f^k, \varphi^k)$ converges  to a limit $(\phi,u,f,\varphi)$ in some strong sense.
Set
\begin{equation*}\begin{split}
&\overline{\phi}^{k+1}=\phi^{k+1}-\phi^k,\    \overline{u}^{k+1}=u^{k+1}-u^k,\  \overline{f}^{k+1}=f^{k+1}-f^k,\ \overline{\varphi}^{k+1}=\varphi^{k+1}-\varphi^k.
\end{split}
\end{equation*}

Notice that
\begin{equation*}\begin{split}
(\varphi^{k})^{-1}\varphi^{k+1}  \nabla \text{div}u^k=&\overline{\varphi}^{k+1}h^k \nabla\text{div}u^k+\nabla \text{div}u^k,\\[8pt]
\varphi^{k+1}\nabla h^k\text{div}u^k=&\frac{\delta-1}{a\delta}f^k\text{div}u^k\Big(1+\overline{\varphi}^{k+1}h^k\Big),\\[8pt]
(\varphi^{k})^{-1}\varphi^{k+1}f^{k+1}\text{div}u^k=&f^{k+1}\text{div}u^k\big(1+\overline{\varphi}^{k+1}h^k\big),\\[8pt]
(\varphi^{k})^{-1} (\varphi^{k+1})^2\text{div} u^k=&\overline{\varphi}^{k+1}\varphi^{k+1}h^k \text{div}u^k+\varphi^k \text{div}u^k+\overline{\varphi}^{k+1}\text{div} u^k.
\end{split}
\end{equation*}
Then it follows from (\ref{li6zx}) that
 \begin{equation}
\label{eq:1.2w}
\begin{cases}
  \displaystyle
\quad \overline{\phi}^{k+1}_t+u^k\cdot \nabla\overline{\phi}^{k+1} +\overline{u}^k\cdot\nabla\phi ^{k}+(\gamma-1)(\overline{\phi}^{k+1} \text{div}u^k +\phi ^{k}\text{div}\overline{u}^k)=0,\\[8pt]
\quad  \varphi^{k+1}\big(\overline{u}^{k+1}_t+u^k\cdot\nabla \overline{u}^{k} +\overline{u}^{k} \cdot \nabla u^{k-1}\big)+\varphi^{k+1}\nabla \overline{\phi}^{k+1}+\overline{\varphi}^{k+1} \nabla \phi^k+a L\overline{u}^{k+1} \\[8pt]
=-\overline{\varphi}^{k+1} (u^k_t+u^{k-1} \cdot \nabla u^{k-1})+f^{k+1}\cdot Q(\overline{u}^k)+\overline{f}^{k+1}\cdot Q(u^{k-1}),\\[8pt]
\displaystyle
\quad \overline{f}^{k+1}_t+\sum_{l=1}^3 A_l(u^k) \partial_l\overline{f}^{k+1}+B^*(u^k)\overline{f}^{k+1}+a \delta \nabla \text{div}\overline{u}^k=\Upsilon^k_1+\Upsilon^k_2+\Upsilon^k_3,\\[8pt]
  \displaystyle
\quad
\overline{\varphi}^{k+1}_t+u^k\cdot \nabla\overline{\varphi}^{k+1} +\overline{u}^k\cdot\nabla\varphi ^{k}+(1-\delta)(\overline{\varphi}^{k} \text{div}u^k +\varphi ^{k-1}\text{div}\overline{u}^k+\Upsilon^k_4)=0,\\[8pt]
\end{cases}
\end{equation}
where  $\Upsilon^k_i$  $(i=1,...,4)$ are defined respectively as:
\begin{equation*}
\begin{split}
\Upsilon^k_1=&(1-\delta)\Big(\overline{f}^{k} \text{div}u^k +f^{k-1}\text{div}\overline{u}^k+\overline{\varphi}^{k+1}f^kh^k\text{div}u^k-\overline{\varphi}^{k}f^{k-1}h^{k-1}\text{div}u^{k-1}\Big),\\
%\end{split}
%\end{equation*}
%and
%\begin{equation*}
%\begin{split}
\Upsilon^k_2=&-\sum_{l=1}^3(A_l(u^k) -A_l(u^{k-1}) )\partial_lf^{k}-(B^*(u^k) -B^*(u^{k-1}))f^{k}\\
&-a\delta\big(\overline{\varphi}^{k+1}h^k \nabla\text{div}u^k-\overline{\varphi}^{k}h^{k-1} \nabla\text{div}u^{k-1}\big),\\
\Upsilon^k_3=&(\delta-1)\Big(f^{k+1}\text{div}u^k\overline{\varphi}^{k+1}h^k-f^{k}\text{div}u^{k-1}\overline{\varphi}^{k}h^{k-1}+\overline{f}^{k+1} \text{div}u^k +f ^{k}\text{div}\overline{u}^k\Big),\\
\Upsilon^k_4=&\overline{\varphi}^{k+1}\varphi^{k+1}h^k \text{div}u^k-\overline{\varphi}^{k}\varphi^{k}h^{k-1} \text{div}u^{k-1}+\overline{\varphi}^{k+1}\text{div} u^k-\overline{\varphi}^{k}\text{div} u^{k-1}.
\end{split}
\end{equation*}

For $\overline{\phi}^{k+1}$,  multiplying $(\ref{eq:1.2w})_1$ by $2\overline{\phi}^{k+1}$ and integrating over $\mathbb{R}^3$ give
\begin{equation}\label{kaishi1}
\begin{split}
\frac{d}{dt}|\overline{\phi}^{k+1}|^2_2
%=& -2\int \Big(u^k\cdot \nabla\overline{\phi}^{k+1} +\overline{u}^k\cdot\nabla\phi ^{k}+(\gamma-1)(\overline{\phi}^{k+1} \text{div}u^k +\phi ^{k}\text{div}\overline{u}^k)\Big)\overline{\phi}^{k+1}\\
\leq& C\big(|\nabla u^k|_\infty|\overline{\phi}^{k+1}|_2+|\overline{u}^k|_6|\nabla \phi^k|_3+|\nabla\overline{u}^k|_2| \phi^k|_\infty\big)|\overline{\phi}^{k+1}|_2.
\end{split}
\end{equation}
%which means that ($0<\epsilon \leq \frac{1}{10}$ is a constant)
%\begin{equation}\label{go64}
%\displaystyle
%\frac{d}{dt}|\overline{\phi}^{k+1}(t)|^2_2\leq C\epsilon^{-1}|\overline{\phi}^{k+1}(t)|^2_2+\epsilon |\nabla \overline{u}^k(t)|^2_2,\quad \text{for}\quad t\in[0,T^{**}],
%\end{equation}

Then, applying  derivative $\partial_{x}^{\zeta}$ ($|\zeta|=1$) to $(\ref{eq:1.2w})_{1}$, multiplying by $2 \partial_{x}^\zeta\overline{\phi}^{k+1}$ and  integrating over $\mathbb{R}^3$, one gets
\begin{equation*}
\begin{split}
\frac{d}{dt}|\partial_{x}^\zeta\overline{\phi}^{k+1}|^2_{2}
%=&-2\int\partial_{x}^\zeta\big(u^k\cdot \nabla\overline{\phi}^{k+1} +\overline{u}^k\cdot\nabla\phi ^{k})\big)\partial_{x}^\zeta\overline{\phi}^{k+1} \\
%&-2\int\partial_{x}^\zeta\big((\gamma-1)(\overline{\phi}^{k+1} \text{div}u^k +\phi ^{k}\text{div}\overline{u}^k)\big)\partial_{x}^\zeta\overline{\phi}^{k+1} \\
\leq& C\big(|\nabla u^k|_\infty |\nabla \overline{\phi}^{k+1}|_{2}+|\nabla \phi^k|_{\infty}| \nabla \overline{u}^k|_2+|  \overline{u}^k|_6|\nabla^2 \phi^{k}|_3\big)|\nabla \overline{\phi}^{k+1}|_2\\
&+C\big(|\nabla^2 u^k|_3 | \overline{\phi}^{k+1}|_{6}+|\nabla \overline{u}^k|_2 |\nabla \phi^{k}|_\infty+| \phi^k|_{\infty}  |\nabla \text{div} \overline{u}^k|_2\big)   |\nabla \overline{\phi}^{k+1}|_{2},
\end{split}
\end{equation*}
%which means that for $t\in[0,T^{**}]$:
%\begin{equation}\label{go64qqqq}
%\begin{cases}
%\displaystyle
%\frac{d}{dt}|\nabla \overline{\phi}^{k+1}(t)|^2_2\leq A^k_\epsilon(t)\|\overline{\phi}^{k+1}(t)\|^2_1+\epsilon \|\nabla \overline{u}^k(t)\|^2_1,\\[10pt]
%\displaystyle
%A^k_\epsilon(t)=C\Big(| \nabla^2 u^k|_{\infty}+\epsilon^{-1}\Big),\ \text{and} \ \int_0^t A^k_\epsilon(s)\text{d}s\leq C+C\epsilon^{-1}t.
%\end{cases}
%\end{equation}
which, along with (\ref{kaishi1}), implies that   for $t\in[0,T^{**}]$,
\begin{equation}\label{go64qqqqqqzxzx}
\frac{d}{dt}\| \overline{\phi}^{k+1}(t)\|^2_1\leq C\sigma^{-1}\| \overline{\phi}^{k+1}(t)\|^2_1+\sigma\|\nabla \overline{u}^k(t)\|^2_1,
\end{equation}
 where $\sigma\in (0,1)$ is a constant to be determined.

For $\overline{f}^{k+1}$, multiplying $(\ref{eq:1.2w})_3$ by $2\overline{f}^{k+1}$ and integrating over $\mathbb{R}^3$ yield
\begin{equation}\label{go64aa}
\begin{split}
\frac{d}{dt}|\overline{f}^{k+1}|^2_2\leq& C|\nabla u^k |_\infty|\overline{f}^{k+1}|^2_2+C\Big(\sum_{i=1}^3|\Upsilon^k_i |_2+|\nabla^2 \overline{u}^k|_2\Big)|\overline{f}^{k+1}|_2\\
\leq & C\sigma^{-1}|\overline{f}^{k+1}|^2_2+C\|\overline{\varphi}^{k+1}\|^2_1+\sigma (\|\nabla \overline{u}^k\|^2_1+ \|\overline{\varphi}^{k}\|^2_1+ |\overline{f}^{k}|^2_2),
\end{split}
\end{equation}
where one has used the fact that
\begin{equation}\label{go64aa1}
\begin{split}
|\Upsilon^k_1|_2\leq &C\big(|\overline{f}^{k}|_2|\nabla u^k|_\infty+|\nabla \overline{u}^{k}|_2|f^{k-1} |_\infty+| f^{k}|_\infty |\overline{\varphi}^{k+1}|_2|h^k \nabla u^k|_\infty\big)\\
&+C| f^{k-1}|_\infty|h^{k-1} \nabla u^{k-1}|_\infty|\overline{\varphi}^k|_2,\\
|\Upsilon^k_2 |_2\leq& C\big(|\nabla f^{k}|_3 |\overline{u}^k|_6+| f^{k}|_\infty | \nabla \overline{u}^k|_2\big)\\
&+C\big(|\overline{\varphi}^{k+1}|_6|h^k \nabla \text{div}u^k|_3+|\overline{\varphi}^{k}|_6|h^{k-1} \nabla\text{div}u^{k-1}|_3\big),\\
|\Upsilon^k_3 |_2\leq& C\big(| f^{k+1}|_\infty |\overline{\varphi}^{k+1}|_2|h^k \text{div}u^k|_\infty+| f^{k}|_\infty|h^{k-1} \text{div}u^{k-1}|_\infty|\overline{\varphi}^k|_2\big)\\
&+\big(|\nabla u^{k}|_\infty |\overline{f}^{k+1}|_2+| f^{k}|_\infty | \nabla \overline{u}^k|_2\big).
\end{split}
\end{equation}

In the rest of the proof of this subsection, set

\begin{equation}
\begin{split}
R^k(t)=&(|\nabla u^k|_\infty+|\varphi^{k+1}|_\infty|h^k\nabla u^{k}|_\infty),\quad
S^k(t)=(|u^k_t|_3+|u^{k-1}|_\infty|\nabla u^{k-1}|_3),\\[6pt]
 M^k(t)=&(|f^k|_\infty+|f^{k-1}|_\infty)|h^{k-1}\text{div}u^{k-1}|_\infty.
\end{split}
\end{equation}
Then for $\overline{\varphi}^{k+1}$,  multiplying $(\ref{eq:1.2w})_4$ by $2\overline{\varphi}^{k+1}$ and integrating over $\mathbb{R}^3$ lead to
\begin{equation}\label{go64zx}
\begin{split}
\frac{d}{dt}|\overline{\varphi}^{k+1}|^2_2
%=& -2\int \Big(u^k\cdot \nabla\overline{\varphi}^{k+1} +\overline{u}^k\cdot\nabla\varphi ^{k}\Big)\overline{\varphi}^{k+1}\\
%& -2(\delta-1)\int (\overline{\varphi}^{k} \text{div}u^k +\varphi ^{k-1}\text{div}\overline{u}^k+\Upsilon^k_4)\overline{\varphi}^{k+1},\\
\leq& CR^k(t)|\overline{\varphi}^{k+1}|^2_2+C|\overline{u}^k|_3|\nabla \varphi^k|_6 |\overline{\varphi}^{k+1}|_2\\
&+C\big(|\nabla\overline{u}^k|_2| \varphi^{k-1}|_\infty+|\overline{\varphi}^{k}|_2 |\nabla u^k|_\infty\big)|\overline{\varphi}^{k+1}|_2+CR^{k-1}|\overline{\varphi}^{k}|_2|\overline{\varphi}^{k+1}|_2.
\end{split}
\end{equation}

Then,  applying  $\partial_{x}^{\zeta}$ ($|\zeta|=1$) to $(\ref{eq:1.2w})_{4}$, multiplying by $2 \partial_{x}^\zeta\overline{\varphi}^{k+1}$ and  integrating over $\mathbb{R}^3$,  one has
\begin{equation*}
\begin{split}
&\frac{d}{dt}|\partial_{x}^\zeta\overline{\varphi}^{k+1}|^2_{2}\\
%=&-2\int\partial_{x}^\zeta\big(u^k\cdot \nabla\overline{\varphi}^{k+1} +\overline{u}^k\cdot\nabla\varphi ^{k})\big)\partial_{x}^\zeta\overline{\varphi}^{k+1} \\
%&-2(\delta-1)\int\partial_{x}^\zeta\big((\overline{\varphi}^{k} \text{div}u^k +\varphi ^{k-1}\text{div}\overline{u}^k+\Upsilon^k_4\big)\partial_{x}^\zeta\overline{\varphi}^{k+1} \\
\leq& C\big(R^k+|h^k\nabla u^k|_6|\nabla \varphi^{k+1} |_6+|\varphi^{k+1}|_\infty|h^k\nabla^2 u^k|_3+|\nabla^2 u^k|_3\big)|\nabla \overline{\varphi}^{k+1}|^2_2\\
&+C\big(|\nabla \varphi^{k-1}|_{6}| \nabla \overline{u}^k|_3+|\overline{\varphi}^{k+1}|_2|\varphi^{k+1}|_\infty|\psi^k|_\infty| \nabla u^k|_\infty)\big)|\nabla \overline{\varphi}^{k+1}|_2\\
&+C\big(|  \overline{u}^k|_6|\nabla^2 \varphi^{k}|_3+|\nabla^2 u^k|_3 | \overline{\varphi}^{k}|_{6} +|\nabla u^k|_\infty|\nabla \overline{\varphi}^k|_2\big) |\nabla \overline{\varphi}^{k+1}|_{2}\\
&+C\big(|\nabla \overline{u}^k|_3 |\nabla \varphi^{k-1}|_6+| \varphi^{k-1}|_{\infty}  |\nabla^2 \overline{u}^k|_2\big)   |\nabla \overline{\varphi}^{k+1}|_{2}\\
&+C( R^{k-1}+|\varphi^k|_\infty|\psi^{k-1}|_\infty|\nabla u^{k-1}|_3+|\nabla^2 u^{k-1}|_3)|\nabla \overline{\varphi}^{k}|_{2}  |\nabla \overline{\varphi}^{k+1}|_{2}\\
&+C\big(|\varphi^k|_\infty|h^{k-1}\nabla^2 u^{k-1}|_3+|\nabla\varphi^k|_6|h^{k-1}\nabla u^{k-1}|_6\big)|\nabla \overline{\varphi}^{k}|_{2} |\nabla \overline{\varphi}^{k+1}|_{2},
\end{split}
\end{equation*}
which, along with (\ref{go64zx}),  implies that for $t\in[0,T^{**}]$
\begin{equation}\label{go64qqqqqq}
\frac{d}{dt}\| \overline{\varphi}^{k+1}(t)\|^2_1\leq C\sigma^{-1}\| \overline{\varphi}^{k+1}(t)\|^2_1+\sigma \big(\|\nabla \overline{u}^k\|^2_1+\| \overline{\varphi}^k\|^2_1\big).
\end{equation}

For $\overline{u}^{k+1}$, multiplying $(\ref{eq:1.2w})_2$ by $2\overline{u}^{k+1}$ and integrating over $\mathbb{R}^3$ yield
\begin{equation}\label{shan101}
\begin{split}
&\frac{d}{dt}|\sqrt{\varphi}^{k+1}\overline{u}^{k+1}|^2_2+2a\alpha |\nabla\overline{u}^{k+1} |^2_2+2a(\alpha+\beta) |\text{div}\overline{u}^{k+1} |^2_2\\
=& \int \Big( \varphi^{k+1}_t|\overline{u}^{k+1}|^2 -2\varphi^{k+1}\big(u^k \cdot \nabla \overline{u}^{k}+\overline{u}^{k} \cdot \nabla u^{k-1}\big)\cdot \overline{u}^{k+1} \Big) \\
&+2\int \Big(-\overline{\varphi}^{k+1} (u^k_t+u^{k-1} \cdot \nabla u^{k-1})-\varphi^{k+1}\nabla \overline{\phi}^{k+1}-\overline{\varphi}^{k+1} \nabla \phi^k\Big) \cdot \overline{u}^{k+1}\\
&+2\int\Big(f^{k+1}\cdot Q(\overline{u}^k)+\overline{f}^{k+1}\cdot Q(u^{k-1}) \Big) \cdot \overline{u}^{k+1}=\sum_{j=1}^{9}J_j.
\end{split}
\end{equation}
Now the terms on the right-hand side of (\ref{shan101}) can be estimated as follows:
\begin{equation}\label{shan102}
\begin{split}
J_1=&\int \varphi^{k+1}_t|\overline{u}^{k+1}|^2
\leq C|u^k|_\infty |\varphi^{k+1}|^{\frac{1}{2}}_\infty |\nabla \overline{u}^{k+1}|_2|\sqrt{\varphi}^{k+1}\overline{u}^{k+1}|_2\\
&+C\big(|\nabla u^k|_\infty+|h^k\nabla u^k|_\infty |\varphi^{k+1}|_\infty\big)|\sqrt{\varphi}^{k+1}\overline{u}^{k+1}|^2_2,\\
\end{split}
\end{equation}
and
\begin{equation}\label{shan102-2}
\begin{split}
J_2+J_3=& \int  -2\varphi^{k+1}\big(u^k \cdot \nabla \overline{u}^{k}+\overline{u}^{k} \cdot \nabla u^{k-1}\big)\cdot \overline{u}^{k+1}  \\
\leq & C|\varphi^{k+1}|^{\frac{1}{2}}_\infty|\sqrt{\varphi}^{k+1}\overline{u}^{k+1}|_2\big(|\nabla \overline{u}^k|_2|u^k|_\infty+|\overline{u}^k|_6|\nabla u^{k-1}|_3\big),\\
J_4+J_5=&-2\int \overline{\varphi}^{k+1} (u^k_t+u^{k-1} \cdot \nabla u^{k-1})\cdot \overline{u}^{k+1}\\
\leq & CS^k(t) |\overline{\varphi}^{k+1}|_2|\overline{u}^{k+1}|_6,\\
J_6+J_7=&-2\int \Big(\varphi^{k+1}\nabla \overline{\phi}^{k+1}+\overline{\varphi}^{k+1} \nabla \phi^k\Big) \cdot \overline{u}^{k+1}\\
\leq & C|\varphi^{k+1}|^{\frac{1}{2}}_\infty|\sqrt{\varphi}^{k+1}\overline{u}^{k+1}|_2|\nabla \overline{\phi}^{k+1}|_2+C|\overline{\varphi}^{k+1}|_2 |\nabla \phi^k|_3 |\overline{u}^{k+1}|_6,\\
J_8+J_9=&2\int\Big(f^{k+1}\cdot Q(\overline{u}^k)+\overline{f}^{k+1}\cdot Q(u^{k-1}) \Big) \cdot \overline{u}^{k+1}\\
=&2\int\Big(\varphi^{k+1}\psi^{k+1}\cdot Q(\overline{u}^k)+\overline{f}^{k+1}\cdot Q(u^{k-1}) \Big) \cdot \overline{u}^{k+1}\\
\leq & C|\varphi^{k+1}|^{\frac{1}{2}}_\infty|\sqrt{\varphi}^{k+1}\overline{u}^{k+1}|_2|\psi^{k+1}|_\infty|\nabla \overline{u}^k|_2+C|\overline{f}^{k+1}|_2 |\nabla u^{k-1}|_3 | \overline{u}^{k+1}|_6,
\end{split}
\end{equation}
where one has used  the equation $(\ref{li6zx})_4$ for the term $\varphi^{k+1}_t$.

It then follows from (\ref{shan101})-(\ref{shan102-2}) and Young's inequality that
\begin{equation}\label{shan201}
\begin{split}
&\frac{d}{dt}|\sqrt{\varphi}^{k+1}\overline{u}^{k+1}|^2_2+2a\alpha |\nabla\overline{u}^{k+1} |^2_2+2a(\alpha+\beta) |\text{div}\overline{u}^{k+1} |^2_2\\
\leq &C\sigma^{-1}|\sqrt{\varphi}^{k+1}\overline{u}^{k+1}|^2_2+\sigma |\nabla \overline{u}^k|^2_2+C\big(\|\overline{\phi}^{k+1}\|^2_1+\|\overline{\varphi}^{k+1}\|^2_1+|\overline{f}^{k+1}|^2_2\big).
\end{split}
\end{equation}
Next multiplying $(\ref{eq:1.2w})_2$ by 2$\overline{u}^{k+1}_t$ and  integrating it over $\mathbb{R}^3$ show that
\begin{equation}\label{shan309}
\begin{split}
&2|\sqrt{\varphi}^{k+1}\overline{u}^{k+1}_t|^2_{2}+\frac{d}{dt}\big(\alpha|\nabla \overline{u}^{k+1} |^2_2+(\alpha+\beta)|  \text{div} \overline{u}^{k+1} |^2_2\big)\\
=&2 \int \Big(  -\varphi^{k+1}\big(u^k \cdot \nabla \overline{u}^{k}+\overline{u}^{k} \cdot \nabla u^{k-1}\big)\cdot \overline{u}^{k+1}_t \Big) \\
&+2\int \Big(-\overline{\varphi}^{k+1} (u^k_t+u^{k-1} \cdot \nabla u^{k-1})-\varphi^{k+1}\nabla \overline{\phi}^{k+1}-\overline{\varphi}^{k+1} \nabla \phi^k\Big) \cdot \overline{u}^{k+1}_t\\
&+2\int\Big(f^{k+1}\cdot Q(\overline{u}^k)+\overline{f}^{k+1}\cdot Q(u^{k-1}) \Big) \cdot \overline{u}^{k+1}_t=\sum_{j=10}^{17} J_j.
\end{split}
\end{equation}
Now the terms on the right-hand side of (\ref{shan309}) admit the following estimates
\begin{equation}\label{shanjigg}
\begin{split}
J_{10}+&J_{11}= 2\int -\varphi^{k+1}\big(u^k \cdot \nabla \overline{u}^{k}+\overline{u}^{k} \cdot \nabla u^{k-1}\big)\cdot \overline{u}^{k+1}_t  \\
\leq & C|\varphi^{k+1}|^{\frac{1}{2}}_\infty|\sqrt{\varphi}^{k+1}\overline{u}^{k+1}_t|_2\big(|\nabla \overline{u}^k|_2|u^k|_\infty+|\overline{u}^k|_6|\nabla u^{k-1}|_3\big),\\
\end{split}
\end{equation}
and
\begin{equation}\label{shanjie1}
\begin{split}
J_{12}+&J_{13}=-2\int \overline{\varphi}^{k+1} (u^k_t+u^{k-1} \cdot \nabla u^{k-1})\cdot \overline{u}^{k+1}_t\\
=&-2\frac{d}{dt}\int \overline{\varphi}^{k+1} (u^k_t+u^{k-1} \cdot \nabla u^{k-1})\cdot \overline{u}^{k+1}\\
&+2\int \big(\overline{\varphi}^{k+1} (u^k_t+u^{k-1} \cdot \nabla u^{k-1})\big)_t\cdot \overline{u}^{k+1}\\
\leq &-2\frac{d}{dt}\int \overline{\varphi}^{k+1} (u^k_t+u^{k-1} \cdot \nabla u^{k-1})\cdot \overline{u}^{k+1}\\
&+ C|\overline{\varphi}^{k+1}|_3|\overline{u}^{k+1}|_6\big( |u^k_{tt}|_2+|u^{k-1}_t|_2|\nabla u^{k-1}|_\infty+|u^{k-1}|_\infty |\nabla u^{k-1}_t|_2\big)\\
&+C|\overline{u}^{k+1}|_6\big(|u^k|_\infty|\nabla\overline{\varphi}^{k+1}|_2+|\overline{u}^k|_3|\nabla \varphi^k|_6+|\nabla u^k|_\infty|\overline{\varphi} ^{k}|_2\big)S^k(t)\\
&+C| \overline{u}^{k+1}|_6\big(|\overline{\varphi}^{k+1}|_2R^k(t)+|\overline{\varphi}^{k}|_2R^{k-1}(t)+|\varphi^{k-1}|_\infty|\nabla\overline{u}^{k}|_2\big)S^k(t),\\
J_{14}=&-2\int \varphi^{k+1}\nabla \overline{\phi}^{k+1} \cdot \overline{u}^{k+1}_t\text{d}x
\leq  C|\varphi^{k+1}|^{\frac{1}{2}}_\infty|\sqrt{\varphi}^{k+1}\overline{u}^{k+1}_t|_2|\nabla \overline{\phi}^{k+1}|_2,\\
J_{15}=&-2\int \overline{\varphi}^{k+1} \nabla \phi^k \cdot \overline{u}^{k+1}_t=-2\frac{d}{dt}\int \overline{\varphi}^{k+1} \nabla \phi^k \cdot \overline{u}^{k+1}\\
&+2\int \overline{\varphi}^{k+1}_t \nabla \phi^k \cdot \overline{u}^{k+1}+2\int \overline{\varphi}^{k+1} \nabla \phi^k_t \cdot \overline{u}^{k+1}\\
%\leq &-\frac{d}{dt}\int\overline{\varphi}^{k+1} \nabla \phi^k \cdot \overline{u}^{k+1}-\int \big(u^k\cdot \nabla\overline{\varphi}^{k+1} +\overline{u}^k\cdot\nabla\varphi ^{k}\big) \nabla \phi^k \cdot \overline{u}^{k+1}\\
%&-\int(\delta-1)(\overline{\varphi}^{k} \text{div}u^k +\varphi^{k-1}\text{div}\overline{u}^k+\Upsilon^k_4) \nabla \phi^k \cdot \overline{u}^{k+1}+C|\overline{\varphi}^{k+1}|_2 |\nabla \phi^k_t|_3 |\overline{u}^{k+1}|_6\\
\leq &-2\frac{d}{dt}\int\overline{\varphi}^{k+1} \nabla \phi^k \cdot \overline{u}^{k+1}\text{d}x+C|\overline{\varphi}^{k+1}|_3 |\nabla \phi^k_t|_2 |\overline{u}^{k+1}|_6\\
&+C|\overline{u}^{k+1}|_6|\nabla \phi^k|_3\big(|\nabla \overline{\varphi}^{k+1}|_2|u^k|_\infty+|\overline{u}^{k}|_2|\nabla \phi^k|_\infty\big)\\
&+C|\overline{u}^{k+1}|_6|\nabla \phi^k|_3\big(| \overline{\varphi}^{k}|_2|\nabla u^k|_\infty+|\varphi^{k-1}|_\infty |\nabla \overline{u}^k|_2\big)\\
&+C| \overline{u}^{k+1}|_6|\nabla \phi^k|_3\big(|\overline{\varphi}^{k+1}|_2R^k(t)+|\overline{\varphi}^{k}|_2R^{k-1}(t)\big),\\
J_{16}=&2\int f^{k+1}\cdot Q(\overline{u}^k)\cdot \overline{u}^{k+1}_t
%=\int\varphi^{k+1}\psi^{k+1}\cdot Q(\overline{u}^k) \cdot \overline{u}^{k+1}_t\\
\leq  C|\varphi^{k+1}|^{\frac{1}{2}}_\infty|\sqrt{\varphi}^{k+1}\overline{u}^{k+1}_t|_2|\psi^{k+1}|_\infty|\nabla \overline{u}^k|_2,\\
%\end{split}
%\end{equation}
%and
%\begin{equation}\label{shanjie1}
%\begin{split}
J_{17}=&2\int \overline{f}^{k+1}\cdot Q(u^{k-1}) \cdot \overline{u}^{k+1}_t
=2\frac{d}{dt}\int \overline{f}^{k+1}\cdot Q(u^{k-1})  \cdot \overline{u}^{k+1}\\
&-2\int \overline{f}^{k+1}\cdot Q(u^{k-1})_t  \cdot \overline{u}^{k+1}\text{d}x-2\int\overline{f}^{k+1}_t\cdot Q(u^{k-1})  \cdot \overline{u}^{k+1}\\
\leq &2\frac{d}{dt}\int\overline{f}^{k+1}\cdot Q(u^{k-1})  \cdot \overline{u}^{k+1}\text{d}x+ C|\nabla\overline{u}^{k+1}|_2|\nabla\overline{u}^{k}|_2|\nabla u^{k-1}|_\infty\\
&+ C|\nabla\overline{u}^{k+1}|_2|\nabla\overline{u}^{k}|_2|\nabla^2 u^{k-1}|_3+C|\overline{f}^{k+1}|_2|\nabla\overline{u}^{k+1}|_2|\nabla u^{k}|_6|\nabla u^{k-1}|_6\\
&+C|\overline{f}^{k+1}|_2|\nabla\overline{u}^{k+1}|_2\big(| u^{k}|_6|\nabla^2 u^{k-1}|_6+| u^{k}|_\infty|\nabla u^{k-1}|_\infty\big)\\
&+C|\overline{f}^{k+1}|_2 | \overline{u}^{k+1}|_6\big( |\nabla u^{k-1}_t|_3+|\nabla u^{k}|_6|\nabla u^{k-1}|_6)\\
&+C|\overline{u}^{k+1}|_6|\nabla u^{k-1}|_3\big(|\overline{u}^{k}|_6|\nabla f^k|_3+|f^k|_\infty|\nabla\overline{u}^{k}|_2\big)\\
&+C\big(|\overline{\varphi}^{k+1}|_2|h^k \nabla \text{div}u^k|_\infty+|\overline{\varphi}^{k}|_2|h^{k-1} \nabla \text{div}u^{k-1}|_\infty\big)|\nabla u^{k-1}|_3|\overline{u}^{k+1}|_6\\
&+C|\nabla u^{k-1}|_3|\overline{u}^{k+1}|_6\big(|\nabla u^k|_\infty(|\overline{f}^k|_2+|\overline{f}^{k+1}|_2)+(|f^k|_\infty+|f^{k-1}|_\infty)|\nabla \overline{u}^{k}|_2\big)\\
&+C\big(M^k(t)|\overline{\varphi}^{k}|_2+M^{k+1}(t)|\overline{\varphi}^{k+1}|_2\big)|\nabla u^{k-1}|_\infty|\overline{u}^{k+1}|_2,
\end{split}
\end{equation}
 where one has used   the fact that
 \begin{equation*}
\begin{cases}
  \displaystyle
\overline{f}^{k+1}_t=-\sum_{l=1}^3 A_l(u^k) \partial_l\overline{f}^{k+1}-B^*(u^k)\overline{f}^{k+1}-a \delta \nabla \text{div}\overline{u}^k+\Upsilon^k_1+\Upsilon^k_2+\Upsilon^k_3,\\[10pt]
  \displaystyle
\overline{\varphi}^{k+1}_t=-u^k\cdot \nabla\overline{\varphi}^{k+1} -\overline{u}^k\cdot\nabla\varphi ^{k}-(1-\delta)(\overline{\varphi}^{k} \text{div}u^k +\varphi ^{k-1}\text{div}\overline{u}^k+\Upsilon^k_4).
\end{cases}
\end{equation*}

It follows from (\ref{shan309})-(\ref{shanjie1}) and Young's inequality that
\begin{equation}\label{shan309zx}
\begin{split}
&|\sqrt{\varphi}^{k+1}\overline{u}^{k+1}_t|^2_{2}+\frac{d}{dt}\big(a\alpha|\nabla \overline{u}^{k+1} |^2_2+a(\alpha+\beta)|  \text{div} \overline{u}^{k+1} |^2_2\big)\\
\leq &J_{18}+E^k_\sigma(t)|\nabla \overline{u}^{k+1}|^2_2+C|\nabla \overline{u}^k|^2_2+\sigma \big(|\overline{\varphi}^{k}|^2_2+|\overline{f}^{k}|^2_2\big)\\
&+C\big(\|\overline{\phi}^{k+1}\|^2_1+\|\overline{\varphi}^{k+1}\|^2_1+|\overline{f}^{k+1}|^2_2\big),
\end{split}
\end{equation}
where
\begin{equation*}\begin{split}
J_{18}=&-\frac{d}{dt}\int\Big(\overline{\varphi}^{k+1} \big(u^k_t+u^{k-1} \cdot \nabla u^{k-1}+ \nabla \phi^k\big)+\overline{f}^{k+1}\cdot Q(u^{k-1}) \Big)\cdot \overline{u}^{k+1},
\end{split}
\end{equation*}
and $E^k_\sigma(t)$ satisfies
\begin{equation}\label{shan333zx-cc}
\int_0^t E^k_\sigma(s)\text{d}s \leq C+C\sigma^{-1}t,\quad \text{for} \quad  0\leq t \leq T^{**}.
\end{equation}

%
%From the classical estimates for elliptic system in Lemma \ref{zhenok} and $(\ref{eq:1.2w})_4$, we have
%\begin{equation}\label{jiabiaozxtgb}
%\begin{split}
%|\overline{u}^{k+1}|_{D^2} \leq C\big(|\sqrt{\varphi}^{k+1}\overline{u}^{k+1}_t|_2+|\nabla \overline{u}^k|_2+\|\overline{\phi}^{k+1}\|_1+\|\overline{\varphi}^{k+1}\|_1+|\overline{f}^{k+1}|_2\big).
%\end{split}
%\end{equation}

Hence, (\ref{go64qqqqqqzxzx})-(\ref{go64aa}), (\ref{go64qqqqqq}),  (\ref{shan201}), (\ref{shan309zx})-(\ref{shan333zx-cc}) and the  Gronwall's inequality imply that
\begin{equation}\label{jiabiaozxlan}
\begin{split}
&\frac{d}{dt}\Big(|\sqrt{\varphi}^{k+1}\overline{u}^{k+1}|^2_2+\|\overline{\phi}^{k+1}\|^2_1+\|\overline{\varphi}^{k+1}\|^2_1+|\overline{f}^{k+1}|^2_2\\
&+a\alpha \nu |\nabla \overline{u}^{k+1} |^2_2+a(\alpha+\beta)\nu |  \text{div} \overline{u}^{k+1} |^2_2\Big)+\Big(a \alpha |\nabla\overline{u}^{k+1}|^2_2+\nu |\sqrt{\varphi}^{k+1}\overline{u}^{k+1}_t|^2_{2}\Big)\\
\leq &\nu J_{18}+ (N^k_\sigma (t)+\nu)\big(|\sqrt{\varphi}^{k+1}\overline{u}^{k+1}|^2_2+ a\alpha |\nabla \overline{u}^{k+1}|^2_2+\|\overline{\phi}^{k+1}\|^2_1+\|\overline{\varphi}^{k+1}\|^2_1+|\overline{f}^{k+1}|^2_2\big)\\
&+C(\nu+\sigma) a\alpha |\nabla \overline{u}^k|^2_2+C\sigma\big(\|\overline{\phi}^{k}\|^2_1+|\nabla^2 \overline{u}^k|^2_2+(1+\nu)(\|\overline{\varphi}^{k}\|^2_1+|\overline{f}^{k}|^2_2)\big),
\end{split}
\end{equation}
where $\nu \in (0,1)$ is a sufficiently small constant, and  $N^k_\epsilon(t)$ satisfies
\begin{equation}\label{shan333zx}
\int_0^t N^k_\sigma(s)\text{d}s \leq C+C\sigma^{-1}t,\quad \text{for} \quad  0\leq t \leq T^{**}.
\end{equation}

For $J_{18}$,    one has
\begin{equation}\label{shan3099zx}
\begin{split}
\int_0^t J_{18} \text{d}s \leq C\big(|\overline{u}^{k+1}|^2_2+\|\overline{\varphi}^{k+1}\|^2_1+|\overline{f}^{k+1}|^2_2\big).
\end{split}
\end{equation}

For the term $\nabla^2 \overline{u}^k$, according to equations $(\ref{eq:1.2w})_2$ and  Lemma \ref{zhenok}, one has
\begin{equation}\label{jiabiaozxtgb}
\begin{split}
|\overline{u}^{k}|_{D^2} \leq C\big(|\sqrt{\varphi}^{k}\overline{u}^{k}_t|_2+|\nabla \overline{u}^{k-1}|_2+\|\overline{\phi}^{k}\|_1+\|\overline{\varphi}^{k}\|_1+|\overline{f}^{k}|_2\big).
\end{split}
\end{equation}
Finally, define
\begin{equation*}\begin{split}
\Gamma^{k+1}(t,\nu)=&\sup_{0\leq s \leq t}\|\overline{\phi}^{k+1}(s)\|^2_{1}+\sup_{0\leq s \leq t}\|\overline{\varphi}^{k+1}(s)\|^2_{ 1}+\sup_{0\leq s \leq t}|\overline{f}^{k+1}(s)|^2_2\\
&+\sup_{0\leq s \leq t}  |\sqrt{\varphi}^{k+1}\overline{u}^{k+1}(s)|^2_{2}+a\nu \sup_{0\leq s \leq t}(\alpha |\nabla \overline{u}^{k+1}(s)|^2_2+(\alpha+\beta) |  \text{div} \overline{u}^{k+1}(s) |^2_2).
\end{split}
\end{equation*}
Then it follows from  (\ref{jiabiaozxlan})-(\ref{jiabiaozxtgb}) that
\begin{equation*}\begin{split}
&\Gamma^{k+1}(t, \nu)+\int_{0}^{t}\Big(a \alpha |\nabla\overline{u}^{k+1}|^2_2+\nu |\sqrt{\varphi}^{k+1}\overline{u}^{k+1}_t|^2_{2}\Big)\text{d}s\\
\leq& C \Big(\int_0^t \Big(a\alpha (\sigma+\nu) (|\nabla \overline{u}^k|^2_2+|\nabla \overline{u}^{k-1}|^2_2)+\sigma |\sqrt{\varphi}^{k}\overline{u}^{k}_t|^2_2\Big)\text{d}s+\sigma t\Gamma^{k}(t, \nu)\Big)\exp{\Big(C+C\sigma^{-1}t\big)}.
\end{split}
\end{equation*}

Now, choose $\nu=\nu_0\in (0,1)$, $\sigma=\sigma_0\in(0,1)$, and $T_0>0$ consecutively so that
$$
C\nu_0\exp{C}\leq \frac{1}{32},
$$
$$
C\sigma_0\exp{C}\leq \frac{\nu_0}{32},
$$
$$
\big(T_*+1) \exp{\Big(C\sigma^{-1}_0T_*\big)}\leq 4.
$$
Then one gets easily
\begin{equation}\label{dingmin}\begin{split}
\sum_{k=1}^{\infty}\Big(  \Gamma^{k+1}(T_*,\nu_0)+\int_{0}^{T_*} \Big(a \alpha|\nabla\overline{u}^{k+1}|^2_2+\nu_0 |\sqrt{\varphi}^{k+1}\overline{u}^{k+1}_t|^2_{2}\Big)\leq C<+\infty.
\end{split}
\end{equation}

Thanks to (\ref{dingmin}) and the local estimates  (\ref{jkk}) independent of $k$,  one has
\begin{equation}\label{dingmin1}\begin{split}
\lim_{k\mapsto +\infty} |\overline{f}^{k+1}|_{6}+\lim_{k\mapsto +\infty} |\overline{\varphi}^{k+1}|_{\infty}=0.
\end{split}
\end{equation}
Thus, by  (\ref{dingmin})-(\ref{dingmin1}), one concludes  that the whole sequence $(\phi^k,u^k,f^k,\varphi^k)$ converges to a limit $(\phi,u,f,\varphi)$ in the following strong sense: for any $s'\in [1,3)$,
\begin{equation}\label{str}
\begin{split}
&\phi^k-\phi^\infty \rightarrow \phi-\phi^\infty\ \text{in}\ L^\infty([0,T_*];H^{s'}(\mathbb{R}^3)),\quad f^k\rightarrow f\ \text{in}\ L^\infty([0,T_*];L^6(\mathbb{R}^3)),\\
&\varphi^k\rightarrow \varphi\ \text{in}\ L^\infty([0,T_*];L^\infty(\mathbb{R}^3)),\quad u^k\rightarrow u\ \text{in}\ L^\infty ([0,T_*];D^1 \cap D^{s'} (\mathbb{R}^3)).
\end{split}
\end{equation}

Again by virtue of   the local estimates  (\ref{jkk}) independent of $k$,  there exists a subsequence (still denoted by $(\phi^k,u^k, f^k, \varphi^k,\psi^k)$) converging  to the  limit $(\phi,u,f, \varphi,\psi)$ in the weak or weak* sense.
%\begin{equation}\label{ruojixianqq}
%\begin{split}
%(\phi^k-\phi^\infty,u^k)\rightharpoonup  (\phi-\phi^\infty,u) \quad &\text{weakly* \ in } \ L^\infty([0,T_*];H^3),\\
%u^k \rightharpoonup  u \quad &\text{weakly \ in } \ L^2([0,T_*];H^4),\\
%\varphi^k\rightarrow \varphi \quad &\text{weakly* \ in } \ L^\infty([0,T_*];L^\infty\cap D^{1,6} \cap D^{2,3} \cap D^3),\\
%f^k \rightarrow f \quad &\text{weakly* \ in } \ L^\infty([0,T_*];L^\infty\cap L^6\cap  D^{1,3} \cap D^2),\\
%\psi^k  \rightharpoonup  \psi \quad &\text{weakly* \ in } \ L^\infty([0,T_*];D^1\cap D^2),\\
%(\phi^k_t,\varphi^k_t)\rightharpoonup ( \phi_t,\varphi_t) \quad &\text{weakly* \ in } \ L^\infty([0,T_*];H^2),\\
% (f^k_t,\psi^k_t, u^k_t)\rightharpoonup  (f_t, \psi_t,u_t) \quad &\text{weakly* \ in } \ L^\infty([0,T_*];H^1),\\
%\varphi^k_t \rightharpoonup \varphi_t \quad &\text{weakly* \ in } \ L^\infty([0,T_*];L^6\cap D^{1,3}\cap D^2),\\
% f^k_t\rightharpoonup  f_t \quad &\text{weakly* \ in } \ L^\infty([0,T_*];L^3\cap D^{1}).
%\end{split}
%\end{equation}
According to the lower semi-continuity of norms,  the corresponding estimates in (\ref{jkk})  for  $(\phi,u, f,\varphi,\psi)$ still hold except those weighted estimates on $u$.
Thus,  $(\phi,u, f,\varphi)$ is a weak solution in the sense of distributions to the following Cauchy problem:
\begin{equation}\label{liyhn}
\begin{cases}
\displaystyle
\phi_t+u\cdot \nabla \phi+(\gamma-1)\phi\text{div} u=0,\\[6pt]
\displaystyle
\varphi\big(u_t+u\cdot\nabla u+\nabla \phi\big)=-a Lu+f\cdot Q(u),\\[12pt]
\displaystyle
f_t+\sum_{l=1}^3 A_l(u) \partial_lf+B^*(u)f+a\delta\nabla \text{div}u=0,\\[12pt]
\displaystyle
\varphi_t+u\cdot \nabla \varphi-(\delta-1)\varphi \text{div} u=0,\\[12pt]
\displaystyle
(\phi, u, f, \varphi)|_{t=0}=\big(\phi_0, u_0, \phi^{-2e}_0\psi_0,\phi^{-2e}_0\big),\\[12pt]
(\phi, u,  f, \varphi)\rightarrow (\phi^\infty, 0, 0,\big(\phi^\infty\big)^{-2e}),\ \text{as}\   |x|\rightarrow +\infty,\  t\geq 0.
 \end{cases}
\end{equation}

\textbf{Step 1.2:} Strong convergence of $\psi^k$ and the existence to the problem (\ref{eq:cccq-xxx}).
However, the conclusions obtained in Step 1.1  are still insufficient to show  the desired existence of the strong solution to the Cauchy problem (\ref{eq:cccq-xxx}).

For this purpose,  one first needs to check the strong convergence of $\psi^{k}$:
\begin{equation}\label{ghjzx}
\begin{split}
|\psi^{k+1}-\psi^k|_6= \Big|\frac{f^{k+1}\varphi^k-f^{k}\varphi^{k+1}}{\varphi^{k+1}\varphi^k}\Big|_6\leq C \big(|\varphi^k|_\infty |\overline{f}^{k+1}|_6+|f^k|_6|\overline{\varphi}^{k+1}|_\infty\big),
\end{split}
\end{equation}
which, along with (\ref{str}),  implies that
\begin{equation}\label{xiaoqiang}
\psi^k\rightarrow \psi\ \text{in}\ L^\infty([0,T_*];L^6(\mathbb{R}^3)).
\end{equation}

Next, one needs to show the relation $f=\psi \varphi$ still holds for the limit functions. Due to
\begin{equation}\label{aqian1}
\begin{split}
|f^k-\psi \varphi|_{6}
\leq C\big(|\varphi^k-\varphi|_\infty  |\psi^k|_{6}+|\psi^k-\psi|_6 |\varphi|_{\infty})\rightarrow 0,\quad \text{as} \quad k\rightarrow \infty,
\end{split}
\end{equation}
then it holds that
\begin{equation}\label{xiaoqiang1}
f(t,x)=\psi \varphi(t,x),\quad a.e. \quad \text{on} \quad [0,T_*] \times \mathbb{R}^3.
\end{equation}

Next to vertify the relations
\begin{equation}\label{relationrecover}
f=\frac{2a e \delta }{\delta-1}\frac{\nabla \phi}{\phi},\quad \varphi=\phi^{-2e} \quad \text{and} \quad  \psi=\frac{a\delta}{\delta-1}\nabla \phi^{2e} ,
\end{equation}
we denote
$$
f^*=f-\frac{2a e \delta }{\delta-1}\frac{\nabla \phi}{\phi}  \quad \text{and} \quad   \varphi^*=\varphi-\phi^{-2e}.
$$
Then it follows from the equations $(\ref{liyhn})_1$ and $(\ref{liyhn})_3$-$(\ref{liyhn})_4$ that
\begin{equation}\label{liyhn-recover}
\begin{cases}
\displaystyle
f^*_t+\sum_{l=1}^3 A_l(u) \partial_lf^*+B^*(u)f^*=0,\\[12pt]
\displaystyle
\varphi^*_t+u\cdot \nabla \varphi^*-(\delta-1)\varphi^* \text{div} u=0,\\[12pt]
\displaystyle
( f^*, \varphi^*)|_{t=0}=\big(0,0\big),\\[12pt]
( f^*, \varphi^*)\rightarrow (0,0),\ \text{as}\   |x|\rightarrow +\infty,\  t\geq 0,
 \end{cases}
\end{equation}
which, together with a standard energy method, implies that
$$
f^*=0\quad \text{and}\quad  \varphi^*=0 \quad \text{for}\quad (t,x)\in [0,T_*]\times \mathbb{R}^3.
$$
Then the first two relations in (\ref{relationrecover}) have been verified, and the last one follows easily from the relation (\ref{xiaoqiang1}).

Moreover,  denoting  $h=\varphi^{-1}$, one needs to show the following weak convergence:
$$
h^k \nabla^2u^{k} \rightharpoonup  \varphi^{-1} \nabla^2 u \quad \text{weakly* \ in } \ L^\infty([0,T_*]; H^1)\cap   L^2([0,T_*] ; D^1).
$$
Indeed, due to the uniform  positivity for $ \varphi^k\geq \underline{\eta}>0$ and $ \varphi\geq \underline{\eta}$,  (\ref{str}),  and the upper bounds of the norms of $(\phi,u,\varphi, f)$, one has
\begin{equation}\label{min}\begin{split}
&\int_0^{T_*}\int_{\mathbb{R}^3} \Big(h^k\nabla^2 u^{k}-\varphi^{-1}\nabla^2 u\Big)w\text{d}x\text{d}t\\
= &\int_0^{T_*}\int_{\mathbb{R}^3} \Big(\Big(   \frac{ \varphi-\varphi^k}{\varphi^k \varphi }  \Big)\nabla^2 u^{k}+\varphi^{-1}(\nabla^2 u^k-\nabla^2 u)\Big)w\text{d}x\text{d}t\\
\leq & C(\underline{\eta})\Big(\sup_{0\leq t \leq T_{*}}|\varphi^k-\varphi |_\infty +\|\nabla^2 u^k-\nabla^2 u\|_{L^\infty([0,T_*];L^2)}\Big)T_* \rightarrow \ 0,\quad \text{as} \quad k\rightarrow +\infty
\end{split}
\end{equation}
for any test functions  $w(t,x)\in \mathbb{R}^3$ and $w(t,x)\in C^\infty_c ([0,T^*)\times \mathbb{R}^3)$, which implies that
\begin{equation}\label{ming}
\begin{split}
h^k\nabla^2 u^{k} \rightharpoonup \varphi^{-1}\nabla^2  u \quad &\text{weakly* \ in } \   \in L^\infty([0,T_*]; H^1).
\end{split}
\end{equation}
Similarly, one  can also obtain that
\begin{equation}\label{mingaz}
\begin{split}
\sqrt{h^{k}}(\nabla u^{k},  \nabla u^{k}_t) \rightharpoonup  \sqrt{h}(\nabla u,  \nabla u_t) \quad &\text{weakly* \ in } \ \  L^\infty([0,T^*];L^2),\\
h^k\nabla^2 u^{k}\rightharpoonup \varphi^{-1}\nabla^2 u \quad &\text{weakly \ in } \  \   L^2([0,T_*]; D^1\cap D^2),\\
(h^k\nabla^2 u^{k})_t \rightharpoonup (\varphi^{-1}\nabla^2 u)_t \quad &\text{weakly \ in } \ \      L^2([0,T_*] ; L^2).
\end{split}
\end{equation}
Then the corresponding weighted  estimates for $u$ shown in the a priori estimates (\ref{jkk}) still hold for the limit functions.

Based on the  estimates (\ref{jkk}), strong convergences shown in (\ref{str}), (\ref{xiaoqiang}), (\ref{ming})-(\ref{mingaz}) and relations (\ref{xiaoqiang1})-(\ref{relationrecover}), it is obvious that
functions
$$\Big(\phi, u, h=\phi^{2e}, \psi=\frac{a\delta}{\delta-1}\nabla \phi^{2e}\Big)$$
 satisfy (\ref{eq:cccq-xxx}) in the sense of distributions. The a priori estimates (\ref{jkk})  hold for $(\phi,u,h)$, and
\begin{equation*}\begin{split}
& \phi -\phi^\infty\in L^\infty([0,T_*];H^3),\quad  \psi \in L^\infty([0,T_*];D^1\cap D^2),\\
&\frac{2}{3}\eta^{-2e}<\varphi \in L^\infty\cap D^{1,6}\cap D^{2,3}\cap D^3,\ \ f\in L^\infty\cap L^6\cap D^{1,3}\cap D^2,\\
& u\in L^\infty([0,T_*]; H^3)\cap L^2([0,T_*] ; H^4), \ \ \ u_t \in L^\infty([0,T_*]; H^1)\cap L^2([0,T_*] ; D^2),\\
& u_{tt}\in L^2([0,T_*];L^2),\quad   t^{\frac{1}{2}}u\in L^\infty([0,T_*];D^4),\\
&t^{\frac{1}{2}}u_t\in L^\infty([0,T_*];D^2)\cap L^2([0,T_*];D^3) ,\ \  t^{\frac{1}{2}}u_{tt}\in L^\infty([0,T_*];L^2)\cap L^2([0,T_*];D^1).
\end{split}
\end{equation*}

\textbf{Step 2:} Uniqueness.  Let $(\phi_1,  u_1, h_1)$ and $(\phi_2, u_2, h_2)$ be two strong solutions to the   Cauchy problem (\ref{eq:cccq-xxx}) satisfying the uniform estimates in (\ref{jkk}).
Set
\begin{equation*}
\begin{split}
\varphi_i=& (h_i)^{-1},\quad \psi_i=\frac{a\delta}{\delta-1}\nabla h_i,\quad f_i=\frac{a\delta}{\delta-1}  \nabla h_i/h_i=\psi_i \varphi_i \quad \text{for}\quad i=1,2,\\
\overline{\phi}=& \phi_1-\phi_2,\  \    \overline{u}=u_1-u_2, \ \  \overline{f}=f_1-f_2,\ \  \overline{\varphi}=\varphi_1-\varphi_2.
\end{split}
\end{equation*}
Then it follows from (\ref{liyhn}) that
 \begin{equation}
\label{eq:1.2wcvb}
\begin{cases}
  \displaystyle
\quad \overline{\phi}_t+u^1\cdot \nabla\overline{\phi}+\overline{u} \cdot\nabla\phi_2+(\gamma-1)(\overline{\phi}\text{div}u_1 +\phi_2\text{div}\overline{u}^k)=0,\\[8pt]
 \displaystyle
\quad  \varphi_1\overline{u}_t+ \varphi_1u_1\cdot\nabla \overline{u}+\varphi_1\nabla \overline{\phi}+L\overline{u} \\[8pt]
=-\overline{\varphi} ((u_2)_t+u_2 \cdot \nabla u_2)- \varphi_1\overline{u} \cdot \nabla u_2-\overline{\varphi} \nabla \phi_2+f\cdot Q(\overline{u})+\overline{f}\cdot Q(u_2),\\[8pt]
\displaystyle
\quad \overline{f}_t+\sum_{l=1}^3 A_l(u_1) \partial_l\overline{f}+B^*(u_1)\overline{f}+a\delta \nabla \text{div}\overline{u}=\overline{\Upsilon}_1+\overline{\Upsilon}_2,\\[8pt]
  \displaystyle
\quad
\overline{\varphi}_t+u_1\cdot \nabla\overline{\varphi} +\overline{u}\cdot\nabla\varphi_2+(\delta-1)(\overline{\varphi}\text{div}u_1+\varphi_2\text{div}\overline{u})=0,
\end{cases}
\end{equation}
where   $\overline{\Upsilon}_1$  and $\overline{\Upsilon}_2$ are defined as
\begin{equation*}
\overline{\Upsilon}_1=-\sum_{l=1}^3(A_l(u_{1}) \partial_l\psi_{2}-A_l(u_{2}) \partial_l\psi_{2}),\quad \overline{\Upsilon}_2=-(B(u_{1}) \psi_{2}-B(u_{2}) \psi_{2}).
\end{equation*}
Set
$$
\Phi(t)=\|\overline{\phi}(t)\|^2_{1}+\|\overline{\varphi}(t)\|^2_{ 1}+|\overline{f}(t)|^2_2+ |\sqrt{\varphi}_1\overline{u}(t)|^2_{2}+ |\nabla \overline{u}(t)|^2_2.
$$
In a similar way for  the derivation of (\ref{go64qqqqqqzxzx})-(\ref{jiabiaozxlan}), one  can show that
\begin{equation}\label{gonm}\begin{split}
\frac{d}{dt}\Phi(t)+C\big( |\nabla \overline{u}(t)|^2_2+ |\sqrt{\varphi}_1\overline{u}_t|^2_2\big)\leq H(t)\Phi (t),
\end{split}
\end{equation}
where
$$ \int_{0}^{t}H(s)ds\leq C, \quad \text{for} \quad 0\leq t\leq T_*.$$ It follows from the Gronwall's inequality that
$\overline{\phi}=\overline{\varphi}=0$ and $\overline{f}=\overline{u}=0$.
Thus the uniqueness is obtained.

\textbf{Step 3.} The time-continuity  follows easily from  the same  procedure as in Lemma \ref{lem1}.

\end{proof}

\subsection{Taking limit from the non-vacuum flows to the flow with  far field vacuum}
Based on the  local (in time) estimates in (\ref{jkk}),  now we are ready to prove Theorem \ref{th1}.

\begin{proof} We divide the proof into four steps.

 \textbf{Step 1:} The locally uniform positivity of $\phi$.
For any $\eta\in (0,1)$, set
$$\phi^\eta_{ 0}=\phi_0+\eta,\quad \psi^\eta_{ 0}=\frac{a\delta}{\delta-1}\nabla (\phi_0+\eta)^{2e},\quad h^\eta_0=(\phi_0+\eta)^{2e}.$$
Then the initial  compatibility conditions can be given as
\begin{equation}\label{th78zxqdf}
\begin{cases}
\displaystyle
\nabla u_0=(\phi_0+\eta)^{-e}g^\eta_1,\quad  Lu_0=(\phi_0+\eta)^{-2e}g^\eta_2,\\[12pt]
\displaystyle
\nabla \Big(a(\phi_0+\eta)^{2e}Lu_0\Big)=(\phi_0+\eta)^{-e}g^\eta_3,
\end{cases}
\end{equation}
where $g^{\eta}_i$ ($i=1,2,3$) are given as
\begin{equation*}\begin{cases}
\displaystyle
g^\eta_1=\frac{\phi^{-e}_0}{(\phi_0+\eta)^{-e}}g_1,\quad g^\eta_2=\frac{\phi^{-2e}_0}{(\phi_0+\eta)^{-2e}}g_2,\\[8pt]
\displaystyle
g^\eta_3=\frac{\phi^{-3e}_0}{(\phi_0+\eta)^{-3e}}\Big(g_3-\frac{a\eta\nabla \phi^{2e}_0}{\phi_0+\eta}\phi^e_0Lu_0\Big).
\end{cases}
\end{equation*}
Then it follows from the   the initial assumption (\ref{th78qq})-(\ref{th78zxq}) and (\ref{chushi1}) that there exists a $\eta_{1}>0$ such that if $0<\eta<\eta_{1}$, then
\begin{equation}\label{co-verfify}
\begin{split}
1+\eta+\|\phi^\eta_{0}-\eta\|_{3}+|\psi^\eta_{0}|_{D^1\cap D^2}+\|u_0\|_{3}+|g^\eta_1|_2+|g^\eta_2|_2+|g^\eta_3|_2\\
+\|(h^{\eta}_0)^{-1}\|_{L^\infty\cap D^{1,6} \cap D^{2,3} \cap D^3}+\|\nabla h^\eta_0 /h^{\eta}_0\|_{L^\infty \cap L^6\cap D^{1,3}\cap D^2}\leq& \overline{c}_0,
\end{split}
\end{equation}
where $\overline{c}_0$ is a positive constant independent of $\eta$.
Therefore, for initial data $(\phi^\eta_{0}, u^\eta_0,\psi^\eta_0)$, the problem (\ref{eq:cccq-xxx}) admits a unique classical solution $(\phi^\eta,  u^\eta, \psi^\eta)$ in  $[0,T_*]\times \mathbb{R}^3$ satisfying the local estimates in (\ref{dingyi45})-(\ref{jkk}) with $c_0$ replaced by $\overline{c}_0$, and the life span $T_*$ is also independent of $\eta$.

Moreover, the following property holds:
\begin{lemma}\label{yhn}
For any $R_0>0$ and $\eta \in (0,1]$, there exists a constant $a_{R_0}$ such that
\begin{equation}\label{zhengxing}
\phi^\eta(t,x)\geq a_{R_0}>0, \quad \forall \ (t,x)\in [0,T_*] \times B_{R_0},
\end{equation}
where $a_R$ is independent of $\eta$.
\end{lemma}
\begin{proof} It suffices to consider the case when  $R_0$ is sufficiently large.

First, due to the  initial assumptions on $\phi$ and $\psi$:
$$\phi_0 \in  H^3,\  \psi_0= \frac{\delta}{\delta-1}\Big(\frac{A\gamma}{\gamma-1}\Big)^{-2e}\nabla \phi^{2e}_0 \in D^1\cap D^2,$$
one has    $\nabla \phi^{2e}_0\in L^\infty$. Therefore,  the initial  vacuum does not  occur in the interior point but  in the far field, and for every $R'>2$, there exists a constant $C_{R'}$ independent of $\eta$ such that
\begin{equation}\label{hao1}
\phi^\eta_0(x)  \geq C_{R'}+\eta>0, \quad \forall \ x\in  B_{R'}.
\end{equation}

Second, let $x(t;x_0)$ be the particle path starting from $x_0$ at $t=0$, i.e.,
\begin{equation}\label{gobn}
\frac{d}{\text{d}t}x(t;x_0)=u(t,x(t;x_0)),\quad x(0;x_0)=x_0.
\end{equation}
Then  denote by $B(t,R')$ the closed regions that are the images of $B_{R'}$  under the flow map (\ref{gobn}):
\begin{equation*}
\begin{split}
&B(t,R')=\{x(t;x_0)|x_0\in B_{R'}\}.
\end{split}
\end{equation*}
It follows from $ (\ref{eq:cccq})_1$ that
\begin{equation}\label{zhengleme}
\phi^\eta(t,x)=\phi^\eta_0(x_0)\exp\Big(-\int_{0}^{t}(\gamma-1)\textrm{div} u^\eta(s;x(s;x_0))\text{d}s\Big).
\end{equation}
According to (\ref{dingyi45})-(\ref{jkk}), it holds that  for $ 0\leq t \leq T_*$,
\begin{equation}\label{zhengle}
\begin{split}
&\int_0^t|\textrm{div} u^\eta(t,x(t;x_0)|\text{d}s\leq \int_0^t|\nabla u^\eta|_\infty\text{d}s\\
\leq &  \int_0^t \|\nabla u^\eta\|_2\text{d}s\leq t^{\frac{1}{2}}\Big(\int_0^t \|\nabla u^\eta\|^2_2\text{d}s\Big)^{\frac{1}{2}}\leq c_3T^{\frac{1}{2}}_*.
\end{split}
\end{equation}
Thus,  by (\ref{hao1}) and (\ref{zhengle}),   one can obtain   that for $ 0\leq t \leq T_*$,
\begin{equation}\label{hao2}
 \phi^\eta(t,x)\geq C^*(C_{R'}+\eta)>0, \quad \forall \ x\in  B(t,R'),
\end{equation}
where $C^*=\exp\big(-(\gamma-1)c_3T^{\frac{1}{2}}_*\big)$.

At last, it follows from (\ref{gobn}) and (\ref{dingyi45})-(\ref{jkk}) that
\begin{equation*}
\begin{split}
&|x_0-x|=|x_0-x(t;x_0)|
\leq  \int_0^t| u^\eta(\tau,x(\tau;x_0))|\text{d}\tau\leq c_3t\leq 1\leq  R'/2,
\end{split}
\end{equation*}
for all $(t,x)\in [0,T_*]\times B_R$,  which  implies $
B_{R'/2} \subset B(t,R')
$.
Thus, one can  choose
$$R'=2R_0,\quad \text{and} \quad a_{R_0}=C^*C_{R'}.$$
\end{proof}

 \textbf{Step 2:} Existence.
 First,  since the estimates  (\ref{jkk}) are independent of $\eta$, then there exists a subsequence (still denoted by $(\phi^\eta, u^\eta, \psi^\eta)$) converging  to a limit $(\phi,u,\psi)$ in weak or weak* sense:
\begin{equation}\label{ruojixianas}
\begin{split}
(\phi^\eta-\eta,u^\eta)\rightharpoonup  (\phi,u) \quad &\text{weakly* \ in } \ L^\infty([0,T_*];H^3),\\
u^\eta \rightharpoonup  u \quad &\text{weakly \ in } \ L^2([0,T_*];H^4),\\
\varphi^\eta \rightarrow \varphi \quad &\text{weakly* \ in } \ L^\infty([0,T_*];L^\infty\cap D^{1,6} \cap D^{2,3} \cap D^3),\\
f^\eta \rightarrow f \quad &\text{weakly* \ in } \ L^\infty([0,T_*];L^\infty\cap L^6\cap  D^{1,3} \cap D^2),\\
\psi^\eta  \rightharpoonup  \psi \quad &\text{weakly* \ in } \ L^\infty([0,T_*];D^1\cap D^2),\\
\phi^\eta_t\rightharpoonup \phi_t \quad &\text{weakly* \ in } \ L^\infty([0,T_*];H^2),\\
 (u^\eta_t, \psi^\eta_t)\rightharpoonup  (u_t, \psi_t) \quad &\text{weakly* \ in } \ L^\infty([0,T_*];H^1),\\
\varphi^\eta_t \rightharpoonup \varphi_t \quad &\text{weakly* \ in } \ L^\infty([0,T_*];L^6\cap D^{1,3}\cap D^2),\\
 f^\eta_t\rightharpoonup  f_t \quad &\text{weakly* \ in } \ L^\infty([0,T_*];L^3\cap D^{1}).
\end{split}
\end{equation}
Then the lower semi-continuity of weak convergences implies that $(\phi, u,f, \varphi, \psi)$  satisfies the corresponding  estimates (\ref{jkk}) except those weighted estimates on $u$.

Second,  for any $R> 0$, due to the Aubin-Lions Lemma  (see \cite{jm}) (i.e., Lemma \ref{aubin}), there exists a subsequence (still denoted by $(\phi^\eta, u^\eta, \psi^\eta)$) satisfying
\begin{equation}\label{ert}\begin{split}
(\phi^\eta-\eta,\psi^\eta,\varphi^\eta)\rightarrow& (\phi,\psi,\varphi) \quad \text{in } \ C([0,T_*];H^1(B_R)),\\
  u^\eta\rightarrow& u \quad \quad  \quad  \ \ \  \text{in } \  C([0,T_*];H^2(B_R)),
\end{split}
\end{equation}
where $B_R$ is a ball centered at origin with radius $R$.

Here,  the following relations  hold for the limit functions:
\begin{equation}\label{relationrecover-1}
f=\frac{2a e \delta }{\delta-1}\frac{\nabla \phi}{\phi},\quad \varphi=\phi^{-2e} \quad \text{and} \quad  \psi=\frac{a\delta}{\delta-1}\nabla \phi^{2e},
\end{equation}
which can be proved by  the  same argument used in the proof of (\ref{relationrecover}).

Let $h=\varphi^{-1}$. According to estimates (\ref{jkk}) except those weighted terms of $u$,   Lemma \ref{yhn}, the weak or weak* convergences shown in (\ref{ruojixianas}) and  the strong convergences shown in (\ref{ert}),  in a similar way for proving (\ref{ming}), one can obtain that
\begin{equation}\label{mingsx}
\begin{split}
\sqrt{h^{\eta}}(\nabla u^{\eta},  \nabla u^{\eta}_t) \rightharpoonup  \sqrt{h}(\nabla u,  \nabla u_t) \quad &\text{weakly* \ in } \ \  L^\infty([0,T^*];L^2),\\
h^\eta \nabla^2 u^{\eta} \rightharpoonup h \nabla^2 u \quad &\text{weakly* \ in } \    L^\infty([0,T_*]; H^1),\\
h^\eta\nabla^2u^{\eta}\rightharpoonup h\nabla^2 u \quad &\text{weakly \ \ in } \     L^2([0,T_*]; D^1\cap D^2),\\
(h^\eta\nabla^2 u^{\eta})_t \rightharpoonup (h\nabla^2 u)_t \quad &\text{weakly \ \ in } \     L^2([0,T_*]; L^2).
\end{split}
\end{equation}
Then the corresponding weighted estimates for $u$  shown in the a priori estimates (\ref{jkk})  hold also for the limit functions.

Thus it is easy to show that $(\phi, u, \psi) $ solves the Cauchy problem (\ref{eq:cccq})-(\ref{sfanb1}) in the sense of distributions.
Moreover, in this step, even though vacuum appears in  the far field, $\psi$ satisfies $\partial_i \psi^{(j)}=\partial_j \psi^{(i)}$  $(i,j=1,2,3)$ and  solves the following positive symmetric hyperbolic system in the sense of distributions:
\begin{equation}\label{zhenzheng}
\psi_t+\sum_{l=1}^3 A_l \partial_l\psi+B\psi+\delta a\phi^{2e}\nabla \text{div} u=0.
\end{equation}

\textbf{Step 3.}  The uniqueness follows easily from  the same  procedure as that for     Theorem \ref{th1zx}.

\textbf{Step 4:} Time continuity. The time continuity of $\phi$ and $\psi$ can be obtained via the similar arguments as used in Lemma \ref{lem1}.

For the velocity $u$,  the a priori estimates obtained above and Sobolev's imbedding theorem imply that
\begin{equation}\label{zheng1}
\begin{split}
 u\in C([0,T_*]; H^2)\cap  C([0,T_*]; \text{weak-}H^3) \quad \text{and} \quad   \phi^{e}\nabla u\in  C([0,T_*]; L^2).
 \end{split}
\end{equation}
Then   equations $(\ref{eq:cccq})_3$ yield
$$
\varphi u_t \in L^2([0,T_*];H^2),\quad (\varphi u_t)_t \in L^2([0,T_*];L^2),
$$
which implies that
$
\varphi u_t \in C([0,T_*];H^1)
$.
Together with classical elliptic estimates and
$$
aLu=-\varphi \big(u_t+u\cdot\nabla u +\nabla \phi-\psi \cdot Q(u)\big),
$$
one gets   $
 u\in C([0,T_*]; H^{3})$ immediately.

Next for  $h\nabla^2 u$, due to
$$
h\nabla^2 u \in L^\infty([0,T_*]; H^1)\cap L^2([0,T_*] ; D^2) \quad \text{and} \quad   (h\nabla^2 u)_t \in  L^2([0,T_*] ; L^2),
$$
and the classical Sobolev imbedding theorem, one gets  quickly that
$$
h\nabla^2 u\in C([0,T_*]; H^1).
$$
Then the  time continuity of $u_t$ follows easily.
\end{proof}

\subsection{The proof for Theorem \ref{th2}}
With Theorem \ref{th1} at hand, now we are  ready to establish the local-in-time well-posedness  of the regular solution to the original Cauchy problem (\ref{eq:1.1})-(\ref{far}) shown in Theorem \ref{th2}.

\begin{proof}  The proof is divided into two steps.

\textbf{Step 1.}
It follows from the initial assumptions (\ref{th78})-(\ref{th78zx}) and Theorem \ref{th1} that there exists  a time $T_{*}> 0$ such that the problem (\ref{eq:cccq})-(\ref{sfanb1}) has a unique regular solution $(\phi,u,\psi)$ satisfying the regularity (\ref{zhengzeA}), which implies that
\begin{equation}\label{reg2}
\begin{split}
\phi\in C^1([0,T_{*}]\times \mathbb{R}^3), \quad (u, \nabla u) \in   C((0,T_{*}]\times \mathbb{R}^3).
\end{split}
\end{equation}
Set $\rho=\Big(\frac{\gamma-1}{A\gamma}\phi\Big)^{\frac{1}{\gamma-1}}$ with $\rho(0,x)=\rho_0$. It follows from the proof of Theorem \ref{th1} (in particular, the proofs of (\ref{relationrecover}) and (\ref{relationrecover-1}) ) that the following relations hold:
\begin{equation}\label{relationrecover-2}
f=\frac{2a e \delta }{\delta-1}\frac{\nabla \phi}{\phi},\quad \varphi=\phi^{-2e} \quad \text{and} \quad  \psi=\frac{a\delta}{\delta-1}\nabla \phi^{2e},
\end{equation}
which implies that
$$
f=a \delta \nabla \log \rho,\quad \varphi=a\rho^{1-\delta} \quad \text{and} \quad  \psi=\frac{\delta}{\delta-1}\nabla \rho^{\delta-1}.
$$

Due to the above regularities and relations of the solution $(\phi,u,\psi)$,  then multiplying $(\ref{eq:cccq})_1$ by
$$
\frac{\partial \rho}{\partial \phi}(t,x)=\frac{1}{\gamma-1}\Big(\frac{\gamma-1}{A\gamma}\Big)^{\frac{1}{\gamma-1}}\phi^{\frac{2-\gamma}{\gamma-1}}(t,x),
$$
yields the continuity equation in (\ref{eq:1.1}), while multiplying $(\ref{eq:cccq})_2$ by $
\rho(t,x)
$ gives the momentum equations in (\ref{eq:1.1}).

Thus we have shown that $(\rho,u)$ satisfies problem (\ref{eq:1.1})-(\ref{far})  in the sense of distributions and has the regularities shown in Definition \ref{d1} and (\ref{reg11}). Finally,
$
\rho(t,x)>0$ for  $(t,x)\in [0,T_{*}]\times \mathbb{R}^3$ follows from the continuity equation.

In summary, the Cauchy problem (\ref{eq:1.1})-(\ref{far}) has a unique regular solution $(\rho,u)$.

\textbf{Step 3.} Now we  show   that, if $\gamma \in (1,2]$,  the regular solution  obtained in the above step is indeed a classical one  within its life span.

First, due to  $1<\gamma \leq 2$, one has
$$(\rho, \nabla \rho,  \rho_t,  u,\nabla u)\in C( [0,T_*]\times \mathbb{R}^3).$$

Second, it follows from the classical Sobolev embedding  result:
\begin{equation}\label{qian}\begin{split}
&L^2([0,T];H^1)\cap W^{1,2}([0,T];H^{-1})\hookrightarrow C([0,T];L^2),
\end{split}
\end{equation}
and the regularity (\ref{reg11}) that
\begin{equation}\label{qiaf}
 tu_t\in C([0,T_*]; H^2), \quad \text{and} \quad u_t\in C([\tau,T]\times \mathbb{R}^3).
\end{equation}

Finally, it remains to  show that
 $ \nabla^2 u \in C([\tau,T_*]\times \mathbb{R}^3)$. According to Step 1 above, the following elliptic system holds
\begin{equation}\label{ekki}
\begin{split}
aLu=-\phi^{-2e}\big(u_t+u\cdot\nabla u +\nabla \phi-\psi \cdot Q(u))=\phi^{-2e} \mathbb{M}.
\end{split}
\end{equation}
The regularity (\ref{reg11}) implies that
\begin{equation}\label{er1}
\begin{split}
t \phi^{-2e} \mathbb{M} \in L^\infty([0,T_*]; H^2).
\end{split}
\end{equation}
Note that
\begin{equation} \label{chejj}
\begin{split}
(t\phi^{-2e} \mathbb{M})_t=&\phi^{-2e} \mathbb{M}+t\phi^{-2e}_t \mathbb{M}  +t\phi^{-2e}\mathbb{M}_t \in L^2([0,T_*]; L^2).
\end{split}
\end{equation}
So it follows from the   classical  Sobolev imbedding theorem:
 \begin{equation}\label{yhnh}
\begin{split}
L^\infty([0,T];H^1)\cap W^{1,2}([0,T];H^{-1})\hookrightarrow C([0,T];L^q),
\end{split}
\end{equation}
for any $q\in [2,6)$, and (\ref{ekki})-(\ref{chejj}) that
$$
t\phi^{-2e} \mathbb{M} \in C([0,T_*]; W^{1,4}), \quad
t\nabla^2 u \in C([0,T_*]; W^{1,4}).
$$
Again  the  Sobolev embedding theorem implies that
$
\nabla^2 u \in C((0,T_*]\times \mathbb{R}^3)
$.

 \end{proof}

Theorem \ref{th2-1} can be proved by the similar argument used in Theorem \ref{th2} under some slight modifications, so the details are omitted.

\section{Non-existence of global solutions with $L^\infty$ decay on $u$}
\subsection{Proof of Theorem \ref{th:2.20}}
Now we are ready to prove Theorem \ref{th:2.20}.  Let $T>0$ be any constant, and    $(\rho,u)\in D(T)$. It  follows from the definitions of  $m(t)$, $\mathbb{P}(t)$ and $E_k(t)$ that
$$
 |\mathbb{P}(t)|\leq \int \rho(t,x)|u|(t,x)\leq  \sqrt{2m(t)E_k(t)},
$$
which, together with the definition of the solution class $D(T)$, implies that
$$
0<\frac{|\mathbb{P}(0)|^2}{2m(0)}\leq E_k(t)\leq \frac{1}{2} m(0)|u(t)|^2_\infty \quad \text{for} \quad t\in [0,T].
$$
Then one obtains that there exists a positive constant $C_u=\frac{|\mathbb{P}(0)|}{m(0)}$ such that
$$
|u(t)|_\infty\geq C_u  \quad \text{for} \quad t\in [0,T].
$$
Thus one obtains     the desired conclusion as shown in Theorem \ref{th:2.20}.

\subsection{Proof of Corollary \ref{th:2.20-c}}

Let $ (\rho,u)(t,x)$ in $[0,T]\times \mathbb{R}^3$ be the regular solution defined in Definition \ref{d1}. Next we just need to show that $(\rho,u)\in D(T)$.

First, it is easy to show that $(\rho,u)$ has finite total mass $m(t)$, finite momentum $\mathbb{P}(t)$, finite total energy $E(t)$.
\begin{lemma}
\label{lemmak-1} Let (\ref{canshu}) and (\ref{decaycanshu}) hold, and  $(\rho,u)$ be the regular solution defined in Definition \ref{d1}, then
$$  m(t)+ |\mathbb{P}(t)|+ E(t) <+\infty \quad \text{for} \quad t\in [0,T]. $$
\end{lemma}
\begin{proof}
Indeed, due to $1<\gamma \leq \frac{3}{2}$,  one has
\begin{equation}\label{finite}\begin{split}
m(t)=\int  \big(\rho^{\gamma-1}\big)^{\frac{1}{\gamma-1}}  \leq C|\rho^{\gamma-1} |^2_2<&+\infty,\\
\displaystyle
\mathbb{P}(t)=\int  \rho u \leq C|\rho|_1|u|_\infty<&+\infty,\\
\displaystyle
E(t)=\int  \Big(\frac{1}{2}\rho|u|^2+\frac{P}{\gamma-1}\Big)  \leq C(|\rho|_\infty|u|^2_2+|\rho|^{\gamma-1}_\infty|\rho|_1)<&+\infty.
\end{split}
\end{equation}
\end{proof}

Second,  the conservation of the total mass and momentum can be verified.
\begin{lemma}
\label{lemmak} Let (\ref{canshu}) and (\ref{decaycanshu}) hold, and  $(\rho,u)$ be the regular solution defined in Definition \ref{d1}. Then
$$\mathbb{P}(t)=\mathbb{P}(0)\quad \text{and} \quad  m(t)= m(0) \quad \text{for} \quad t\in [0,T]. $$
\end{lemma}
\begin{proof}
The momentum equations $(\ref{eq:1.1})_2$ imply   that
\begin{equation}\label{deng1}
\mathbb{P}_t=-\int \text{div}(\rho u \otimes u)-\int \nabla P+\int \text{div}\mathbb{T}=0,
\end{equation}
where one has used the fact that
$$
\rho u^{(i)}u^{(j)},\quad \rho^\gamma \quad \text{and} \quad \rho^\delta \nabla u \in W^{1,1}(\mathbb{R}^3)\quad \text{for} \quad i,\ j=1,\ 2,\ 3.
$$
Similarly, we  can show  the conservation of the mass.
\end{proof}

%\begin{remark}\label{gu66}
%The same conclusion can't be obtained for the strong solutions shown in  \cite{CK3}  due to the different mathematical structure,  even if the initial  density and velocity are both compactly supported. In this sense, the definition of regular solutions with vacuum is consistent with the physical background of the compressible Navier-Stokes equations.
%\end{remark}

According to Theorem \ref{th:2.20} and Lemmas \ref{lemmak-1}-\ref{lemmak}, one  obtains the desired conclusion in Corollary \ref{th:2.20-c}.

\begin{remark}
First, under the assumptions of Corollary \ref{th:2.20-c}, the regular solution $(\rho,u)$  satisfies also the energy equality. Indeed, the continuity equation $(\ref{eq:1.1})_1$ and the definition of $P$ give
\begin{equation}\label{pressure}
P_t+u\cdot \nabla P+\gamma P\text{div}u=0,
\end{equation}
While the momentum equations $(\ref{eq:1.1})_2$ and (\ref{pressure}) yield that
\begin{equation}\label{conservationofenergy}
\begin{split}
&E_t+\int \rho^\delta\Big(\alpha  |\nabla u|^2+(\alpha+\beta)|\text{div}u|^2\big)\\
=& -\frac{1}{2}\int \text{div}(\rho u |u|^2)-\frac{\gamma}{\gamma-1}\int \text{div}(uP)  +\int \text{div}(u\mathbb{T}).
\end{split}
\end{equation}
Due to the definition of the regular solution,  (\ref{canshu}) and (\ref{decaycanshu}),  one gets that
$$
\rho u |u|^2, \quad uP \quad \text{and} \quad u\mathbb{T} \in W^{1,1}(\mathbb{R}^3),
$$
which, along with (\ref{conservationofenergy}) implies the desired energy equality:
\begin{equation}\label{decayenergy}
E(t)+\int_0^t\int \rho^\delta\Big(\alpha  |\nabla u|^2+(\alpha+\beta)|\text{div}u|^2\big)\text{d}x\text{d}s=E(0).
\end{equation}

Second, for the flows of constant viscosities \cite{CK3, HX1} ($\delta=0$), it is not  clear to verify the conservation of momentum for the strong solutions with vacuum for the Cauchy problem. The  reason is that one is not sure whether $\mathbb{T}\in W^{1,1}(\mathbb{R}^3)$, or not.

\end{remark}

\subsection{Proof of Corollary \ref{th:2.20-HLX}}

For the convenience  of the proof, set:
$$
F\triangleq (2\alpha+\beta)\text{div} u-P(\rho),\quad \omega\triangleq\nabla \times u,
$$
where $F$ is the effective viscous flux, and $\omega$ is the vorticity.

\begin{proof} The proof is divided into three steps: 

 \textbf{Step 1:} Local-in-time well-posedness.  It follows easily from the initial assumption (\ref{initialhxl}) and    arguments in \cite{CK3,guahu} that there exists a time $T_0>0$ and a unique classical solution $(\rho,u)$ in $(0,T_0)\times \mathbb{R}^3$  to the Cauchy problem  (\ref{eq:1.1})-(\ref{far})  satisfying
\begin{align}
&   (\rho,P(\rho)) \in C([0,T_0];H^3), \quad \rho^{\frac{1}{2}} \in C([0,T_0];H^1), \label{rrt}\\
& u\in C([0,T_0]; D^1\cap D^3)\cap L^2([0,T_0] ; D^4), \quad  u_t\in L^\infty([0,T_0]; D^1)\cap L^2([0,T_0] ; D^2).\label{rry}
\end{align}

  \textbf{Step 2:} Global-in-time well-posedness.  It follows from the local-in-time well-posedenss obtained in Step 1,  the smallness assumption (\ref{smallness}) and the argument  in \cite{HX1} that   the Cauchy problem  (\ref{eq:1.1})-(\ref{far})
  has a unique  global classical solution $(\rho,u)$ in $(0,\infty)\times \mathbb{R}^3$ satisfying  (\ref{lagrangian1q}) and (\ref{lagrangian3q})-(\ref{largetimehlx}) for any $0<\tau <T<\infty$. It remains to show (\ref{lagrangian2q}).

  First, the continuity equation $(\ref{eq:1.1})_1$ implies that $\rho^{\frac{1}{2}}$ satisfies the following equation
      \begin{align}
\rho^{\frac{1}{2}}_t+u\cdot \nabla\rho^{\frac{1}{2}} +\frac{1}{2} \rho^{\frac{1}{2}} \text{div}u=0\label{equation111},
\end{align}
which, along with the standard energy estimates argument for transport equations and (\ref{lagrangian4q}), yields that
\begin{equation}\label{equation112}\begin{split}
\|\rho^{\frac{1}{2}}(t)\|_{1}\leq& \|\rho^{\frac{1}{2}}_0\|_{1}\exp\Big(C\int_0^t \| u\|_{D^1\cap D^3}\text{d}s\Big)<\infty \quad \text{for} \quad 0<t\leq T,
\end{split}
\end{equation}
where $C>0$ is constant, and $T$ is any positive time. This, together with (\ref{rrt}), shows
\begin{equation}\label{rrq}
\rho^{\frac{1}{2}} \in C([0,T];H^1).\end{equation}

Then (\ref{rrq}) yields the conservation of the mass, since
$$
\frac{d}{dt}\int \rho=-\int \text{div} (\rho u)=0, \quad \text{for} \quad 0<t\leq T,$$
due to $\rho u \in W^{1,1}(\mathbb{R}^3)$.

 \textbf{Step 3:} Verification of the large time behavior on $u$.
First, by the Gagliardo-Nirenberg  inequality  in Lemma \ref{lem2as},  and Lemma 2.3 in \cite{HX1} , one has
\begin{equation}\label{keyobservation}
\begin{split}
|u|_\infty\leq & C|u|^{\frac{1}{2}}_6|u|^{\frac{1}{2}}_6\leq C|\nabla u|^{\frac{1}{2}}_2 (|F|^{\frac{1}{2}}_6+|\omega|^{\frac{1}{2}}_6+|P|^{\frac{1}{2}}_6)\\
\leq &C|\nabla u|^{\frac{1}{2}}_2(|\rho \dot{u}|_2+|P|_6)^{\frac{1}{2}},
\end{split}
\end{equation}
which, together with (\ref{lagrangian1q}), (\ref{lagrangian3q}) and (\ref{largetimehlx}), implies that
\begin{equation}\label{laile}
\limsup_{t\rightarrow +\infty} |u(t,x)|_{\infty}=0.
\end{equation}

Finally, according to  (\ref{laile}), and    Theorem \ref{th:2.20},  if  $m(0)>0$ and $|\mathbb{P}(0)|>0$,  the global solution obtained in Step 2 above can not keep the conservation of momentum for all the time $t\in (0,\infty)$.\end{proof}

\bigskip

{\bf Conflict of Interest:} The authors declare that they have no conflict of interest.

\bigskip

{\bf Acknowledgement:} This research is partially supported by Zheng Ge Ru Foundation, Hong Kong RGC Earmarked Research Grants  CUHK 14300917, CUHK-14305315 and CUHK 4048/13P, NSFC/RGC Joint Research Scheme Grant N-CUHK 443-14, and a Focus Area Grant from The Chinese University of Hong Kong. Zhu's research is also  supported in part
by National Natural Science Foundation of China under grant 11231006,  Natural Science Foundation of Shanghai under grant 14ZR1423100,  Australian Research Council grant DP170100630 and Newton International Fellowships NF170015.

\end{document}